\documentclass[12pt,letterpaper]{article}
\usepackage{geometry}
\usepackage{enumitem}
\usepackage[utf8]{inputenc}
\usepackage{fancyhdr, amsfonts, amsmath, amsthm, amssymb, tikz, array, graphicx, booktabs, setspace, tocloft, titling, blindtext, caption, multicol, hyperref}
\usepackage{parskip}

\newcommand{\rowvector}[1]{\begin{bmatrix} #1 \end{bmatrix}}

\newcommand{\smallrowvector}[1]{\left[\begin{smallmatrix} #1 \end{smallmatrix}\right]}

\newcommand{\GraphNumber}{}
\newcommand{\GraphTitle}{}
\newcommand{\VertexCount}{}
\newcommand{\IntersectionArray}{}

\newcommand{\DateString}{}

\title{\textbf{The Terwilliger algebra for the distance-regular graphs with valency three}}
\author{\\ Kevin Kauflin, Paul Terwilliger, Barnab\'as Valk\'o,\\ Jimmy Vineyard, Hanyi Wu}
\date{}

\newtheorem{theorem}{Theorem}
\newtheorem{definition}[theorem]{Definition}

\begin{document}

\maketitle

\begin{abstract}
In this paper, we discuss a family of highly regular graphs, said to be distance-regular. We are particularly interested in the distance-regular graphs with valency three. It is known that there exist exactly 13 such graphs. Let $\Gamma$ denote a distance-regular graph with vertex set $X$. For any vertex $x \in X$, the corresponding Terwilliger algebra $T=T(x)$ is generated by the adjacency algebra $M$ of $\Gamma$ and the dual adjacency algebra $M^*=M^*(x)$ of $\Gamma$ with respect to $x$. It is known that the algebra $T$ is semisimple. By construction, the vector space $V=\mathbb{C}^X$ is a module for $T$, said to be standard. In this paper we have the following goal. For each of the 13 distance-regular graphs $\Gamma$ with valency three, we will decompose the standard module $V$ into a direct sum of irreducible $T$-modules. Using this information, we will work out the dimension of $T$.

\vspace{0.5cm}

\end{abstract}

\section{Introduction}

In this paper, we discuss a family of highly regular graphs, said to be distance-regular \cite{bannaiitotanaka,bannaiito1984,brouwer}. We are particularly interested in the distance-regular graphs with valency three. By \cite[Theorem~7.5.1]{brouwer}, there are exactly 13 such graphs; these are listed in Table~\ref{tab:13graphs}.

Let $\Gamma$ denote a distance-regular graph with vertex set $X$. For any vertex $x\in X$, there is an algebra $T=T(x)$ called the Terwilliger algebra \cite{terwilliger2022,Terwilliger1992}. The algebra $T$ is generated by the adjacency algebra $M$ of $\Gamma$ and the dual adjacency algebra $M^*=M^*(x)$ of $\Gamma$ with respect to $x$. It is known that the algebra $T$ is semisimple \cite[Sec.\ 1]{Egge2000}. By construction, the vector space $V=\mathbb{C}^X$ is a module for $T$. We call $V$ the standard module.

We now explain the goal of the paper. For each of the 13 distance-regular graphs $\Gamma$ with valency three, we will decompose the standard module $V$ into a direct sum of irreducible $T$-modules. Using this information, we will work out the dimension of $T$.

The paper is organized as follows. Section~\ref{subsec:defs} contains some preliminaries. Section~\ref{subsec:thms} contains some basic theorems that will be used throughout the main body of the paper. In Section~\ref{subsec:method} we summarize our method. Each of Sections~\ref{sec:k4}--\ref{sec:foster} is devoted to one of the 13 graphs. In Section~\ref{sec:conclusion}, we give a summary table.

\section{Preliminaries}\label{subsec:defs}

In this section, we recall some preliminaries concerning distance-regular graphs.

Throughout the paper, $X$ denotes a nonempty finite set. Elements in $X$ are called \emph{vertices}. Let $V=\mathbb{C}^X$ denote the vector space over $\mathbb{C}$ consisting of column vectors with coordinates indexed by $X$ and all entries in $\mathbb{C}$. Let $\mathrm{Mat}_X(\mathbb{C})$ denote the $\mathbb{C}$-algebra consisting of the matrices with rows and columns indexed by $X$ and all entries in $\mathbb{C}$. The algebra $\mathrm{Mat}_X(\mathbb{C})$ acts on $V$ by left multiplication. We call $V$ the \emph{standard module}. We endow $V$ with the Hermitian inner product such that $\langle u, v \rangle = u^t \bar{v}$ for all $u, v \in V$. Here $t$ denotes transpose and $\bar{~}$ denotes complex conjugation.

Let $\Gamma=(X,\mathcal{E})$ denote a finite, undirected, connected graph, without loops or multiple edges, with vertex set $X$ and edge set $\mathcal{E}$. Vertices $y,z \in X$ are said to be \emph{adjacent} whenever they form an edge. For vertices $y,z\in X$, let $d(y,z)$ denote the path-length distance between $y$ and $z$. By the \emph{diameter} of $\Gamma$, we mean the integer $D = \max\{d(y,z) \mid y,z \in X\}$. For the rest of this paper, we assume $D \ge 1$.

 For $x\in X$ and $0 \le i \le D$, define the set $\Gamma_i(x)=\{y \in X \mid d(x,y)=i\}$. We call $\Gamma_i(x)$ the $i$th \textit{subconstituent of $\Gamma$ with respect to $x$}. We abbreviate $\Gamma(x)=\Gamma_1(x)$. The graph $\Gamma$ is called \emph{regular with valency $k$} whenever $|\Gamma(x)|=k$ for every $x\in X$. The graph $\Gamma$ is called \textit{distance-regular} whenever for all $0 \le h,i,j \le D$ and $x,y \in X$ at $d(x,y) =h$, the scalar $p_{i,j}^h=|\Gamma_i(x)\cap \Gamma_j(y)|$ is independent of the choice of $x,y$.

We abbreviate 
\begin{equation*}
c_i=p_{1,i-1}^i\, (1 \le i \le D), \quad
a_i=p_{1,i}^i\, (0 \le i \le D), \quad
b_i=p_{1,i+1}^i\, (0 \le i \le D-1).
\end{equation*}
From now on, assume that $\Gamma$ is distance-regular. Observe that $\Gamma$ is regular with valency $k=b_0$. We have $$c_i+a_i+b_i=k\quad \quad (0\le i \le D),$$
where $c_0=0$ and $b_D=0$. By the \textit{intersection array} of $\Gamma$, we mean the sequence $$\{b_0,b_1,\dots,b_{D-1};c_1,c_2,\dots,c_D\}.$$
We now recall the adjacency algebra of $\Gamma$. Define a matrix $A \in \mathrm{Mat}_X(\mathbb{C})$ with $(y,z)$-entry
$$A_{y,z} = \begin{cases}
    1, & \text{if $y,z$ are adjacent},\\
    0, & \text{if $y,z$ are not adjacent.}
\end{cases} \quad \quad \quad \quad(y,z \in X)$$
We call $A$ the \emph{adjacency matrix} of $\Gamma$. The matrix $A$ is real and symmetric, and therefore diagonalizable. The \textit{spectrum} of $A$ is the multiset consisting of the roots of the characteristic polynomial of $A$. Let $M$ denote the subalgebra of $\mathrm{Mat}_X(\mathbb{C})$ generated by $A$. We call $M$ the \emph{adjacency algebra} of $\Gamma$.

We now recall the dual adjacency algebras of $\Gamma$. From now on, fix a vertex $x \in X$. We call $x$ the \emph{base vertex}. For $0 \le i \le D$, define a diagonal matrix $E_i^* = E_i^*(x) \in \mathrm{Mat}_X(\mathbb{C})$ with $(y,y)$-entry
$${(E_i^*)}_{y,y} = \begin{cases}
    1, & \text{if $d(x,y)=i$},\\
    0, & \text{if $d(x,y)\ne i$.}
\end{cases} \quad \quad \quad \quad(y \in X)$$
By construction, $$E_i^*E_j^* = \delta_{i,j}E_i^*\quad (0 \le i,j \le D), \quad \quad  I = \sum_{i=0}^D E_i^*.$$
By these comments, the matrices $\{E_i^*\}_{i=0}^D$ form a basis for a commutative subalgebra $M^*=M^*(x)$ of $\mathrm{Mat}_X(\mathbb{C})$. We call $M^*$ the \emph{dual adjacency algebra of $\Gamma$ with respect to $x$} \cite[Sec.\ 2]{terwilliger2022}.

Let $T=T(x)$ denote the subalgebra of $\mathrm{Mat}_X(\mathbb{C})$ generated by $M$ and $M^*$. We call $T$ the \textit{Terwilliger algebra of $\Gamma$ with respect to $x$} \cite[Definition~3.3]{Terwilliger1992}. By construction, $T$ is closed under the conjugate transpose map. Consequently $T$ is semisimple \cite[Sec.\ 1]{Egge2000}. 

We now discuss the $T$-modules. By a \emph{$T$-module}, we mean a subspace $W \subseteq V$ such that $BW \subseteq W$ for all $B \in T$. A $T$-module $W$ is said to be \emph{irreducible} whenever $W \ne 0$ and $W$ does not properly contain a nonzero $T$-module.  Let $W_1$ and $W_2$ denote irreducible $T$-modules. By a \textit{$T$-module isomorphism from $W_1$ to $W_2$}, we mean a vector space isomorphism $\sigma : W_1 \to W_2$ such that $\sigma B = B\sigma$ on $W_1$ for all $B \in T$. We say that the $T$-modules $W_1$ and $W_2$ are \textit{isomorphic} whenever there exists a $T$-module isomorphism from $W_1$ to $W_2$. Let $W$ denote a $T$-module. By the \emph{endpoint} of $W$, we mean the integer $\min\{i \mid 0\le i \le D,\;\; E_i^*W \ne 0\}$.

Let $W$ denote an irreducible $T$-module. By the {\it diameter} of $W$, we mean the integer
        $$\bigl \vert \lbrace i \mid 0 \leq i \leq D, \;\; E^*_i W \not=0 \rbrace \bigr\vert -1.$$
By the {\it shape} of $W$, we mean the sequence
       $$ \bigl({\rm dim}\, E^*_rW, {\rm dim}\, E^*_{r+1}W, \ldots, {\rm dim} \, E^*_{r+d}W \bigr),$$
where $r$ (resp. $d$) denotes the endpoint (resp. diameter) of $W$.

\section{Theorems}\label{subsec:thms}

In this section, we give some basic theorems that will be used throughout the rest of the paper. Throughout this section, $\Gamma=(X,\mathcal{E})$ denotes a distance-regular graph. We fix a vertex $x \in X$ and consider the algebra $T=T(x)$.

\begin{definition}\label{def:primary} \rm
For $0\le i \le D$, let $e_i^*=e_i^*(x) \in V$ denote the vector with $y$-th entry
\[
(e_i^*)_y = 
\begin{cases} 1, & \mathrm{if} \text{ } d(x,y) = i, \\
0, & \mathrm{if} \text{ } d(x,y) \ne i. \end{cases} \quad\quad (y \in X)
\]
\end{definition}

\begin{theorem}\label{thm:primary} {\rm (See \cite[Lemma~7.1]{terwilliger2022}).}
The subspace $W_0=\mathrm{span}\{e_i^* \mid 0 \le i \le D\}$ is the unique irreducible $T$-module of endpoint $0$.
\end{theorem}

Referring to Theorem~\ref{thm:primary}, the $T$-module $W_0$ is called {\it primary}.

\begin{theorem}\label{thm:orthocomp}{\rm (See \cite{terwilliger2022,Terwilliger1992}).}
Let $W$ denote a $T$-module. Then its orthogonal complement 
$W^\perp = \{v \in V \mid \langle v,w\rangle = 0 \text{ for all } w \in W\}$ 
is also a $T$-module.
\end{theorem}
\begin{proof}
Let $u \in W^\perp$ and $B \in T$. We show that $B u \in W^\perp$. Let $w\in W$. Then $\langle Bu, w\rangle = \langle u, \bar{B}^t w\rangle = 0$. By these comments, $B u \in W^\perp$.

\end{proof}

\begin{theorem}\label{thm:directsum} {\rm (See \cite[Lemma~3.4]{Terwilliger1992}).}
The standard module $V$ is an orthogonal direct sum of irreducible $T$-modules.
\end{theorem}

By Theorem~\ref{thm:directsum}, the standard module $V$ decomposes into an orthogonal direct sum of irreducible $T$-modules. Let $W$ denote an irreducible $T$-module. By the \textit{multiplicity} of $W$, we mean the number of irreducible $T$-modules in the decomposition that are isomorphic to $W$.

\begin{theorem}\label{thm:Wedderburn}
Let $\Lambda$ denote the set of isomorphism classes of irreducible $T$-modules. 
For $\lambda\in\Lambda$ let $\dim\lambda$ denote the common dimension 
of the irreducible $T$-modules in the $\lambda$ class. Then 
\[
  \dim T = \sum_{\lambda\in\Lambda}(\dim\lambda)^2.
\]
\end{theorem}
\begin{proof}

Follows from the Wedderburn theory \cite{CurtisReiner1962}.
\end{proof}

\begin{theorem}\label{thm:irred}
Let $W$ denote a $T$-module with endpoint $r$ such that $\dim E_r^* W  = 1$. Assume that $W = T E_r^* W$. Then the $T$-module $W$ is irreducible.
\end{theorem}
\begin{proof}
Let $U$ denote a $T$-module contained in $W$. We show that $U=0$ or $U=W$.
By construction $E^*_r U \subseteq E^*_rW$, so $E^*_r U=0$ or $E^*_rU=E^*_rW$.
In the first case, $E^*_rW \subseteq U^\perp$ so $W = TE^*_rW \subseteq T U^\perp
\subseteq U^\perp$ so $U=0$. In the second case, $W=TE^*_rW=T E^*_r U \subseteq U$
so $U=W$.
\end{proof}

\begin{definition}\rm For $0 \ne v \in V$, call $v$ \textit{pure} whenever there exists an integer $i$\, $(0 \le i \le D)$ such that $v \in E_i^*V$. We call $i$ the {\it support} of $v$.
\end{definition} 

\begin{definition} \rm
 Let $W$ denote an irreducible $T$-module. A basis $\lbrace v_1, v_2, \ldots, v_\ell \rbrace$ for $W$ is called {\it pure} whenever
the following conditions are met:
\begin{enumerate}
\item[\rm (i)] the vector $v_i$ is pure for $1 \leq i \leq \ell$; 
\item[\rm (ii)] the support of $v_{i-1}$ is at most the support of $v_i$ for $2 \leq i\leq \ell$.
\end{enumerate}
\end{definition}

\begin{theorem}\label{thm:isom}
Let $W_1$ and $W_2$ denote irreducible $T$-modules. Then the $T$-modules $W_1$ and $W_2$ are isomorphic if and only if the following conditions hold:
\begin{enumerate}
  \item[\rm(i)] $\dim E_i^* W_1 = \dim E_i^* W_2$ for $0\le i\le D$; 
  
  \item[\rm(ii)] there exist pure bases $\mathcal{B}_1$ for $W_1$ and $\mathcal{B}_2$ for $W_2$, such that the matrix representing $A$ with respect to $\mathcal{B}_1$ equals the matrix representing $A$ with respect to $\mathcal{B}_2$.
\end{enumerate}
\end{theorem}
\begin{proof}
Follows from elementary linear algebra.
\end{proof}

We have a comment about notation. Let $v \in V$ denote a pure vector, with support $i$. By the \textit{essential part of $v$}, we mean the vector obtained from $v$ by deleting all the coordinates attached to vertices not at distance $i$ from $x$. For notational convenience, we will display the essential part of $v$ as a row vector.

\section{Method}\label{subsec:method}

By \cite[Theorem~7.5.1]{brouwer}, there are exactly 13 distance-regular graphs with valency three. These graphs are shown in the following table.

    \begin{center}
        \begin{tabular}{lclc}
        \toprule
            Graph $\Gamma$ & Diam. $D$ & Intersection array & $|X|$ \\
            \midrule
            $K_4$  & 1& $\{3;1\}$ & $4$\\
             $K_{3,3}$ & 2& $\{3,2;1,3\}$ &$6$ \\
            Petersen &2& $\{3,2;1,1\}$ & $10$\\
            3-cube &3& $\{3,2,1;1,2,3\}$ & $8$\\
            Heawood &3& $\{3,2,2;1,1,3\}$ & $14$\\
            Pappus &4& $\{3,2,2,1;1,1,2,3\}$ & $18$\\
            Coxeter &4& $\{3,2,2,1;1,1,1,2\}$ & $28$\\
           Tutte's 8-cage  &4& $\{3,2,2,2;1,1,1,3\}$ & $30$\\
            Dodecahedron &5& $\{3,2,1,1,1;1,1,1,2,3\}$ &$20$ \\
           Desargues  &5& $\{3,2,2,1,1;1,1,2,2,3\}$ & $20$\\
            Tutte's 12-cage &6& $\{3,2,2,2,2,2;1,1,1,1,1,3\}$ & $126$\\
           Biggs-Smith  &7& $\{3,2,2,2,1,1,1;1,1,1,1,1,1,3\}$ & $102$ \\
           Foster  &8& $\{3,2,2,2,2,1,1,1;1,1,1,1,2,2,2,3\}$ & $90$ \\
           \bottomrule
      \end{tabular}
    \captionof{table}{\textit{The 13 distance-regular graphs of valency three}\label{tab:13graphs}}
    \end{center}

Assume that $\Gamma$ is a distance-regular graph with valency three.
For the moment, assume that $\Gamma$ is not Tutte's 12-cage. In an upcoming section we will fix
a base vertex $x \in X$ and consider $T=T(x)$. The choice of $x$ is unimportant as $\Gamma$ is distance-transitive by \cite[Theorem~7.5.1]{brouwer}. 

Next assume that $\Gamma$ is Tutte's 12-cage. Then $\Gamma$ is bipartite; let $X=X^+ \cup X^-$ denote the bipartition. By \cite[Theorem~1.1]{iofinova1994} the sets $X^+$ and $X^-$ are the orbits for the automorphism group of $\Gamma$. In two upcoming sections, we will separately consider $T=T(x)$ for $x \in X^+$ and $x \in X^-$.

For each pair $\Gamma, x$ described above,
we will decompose the standard module $V$ into an orthogonal
direct sum of irreducible modules for $T=T(x)$. For each irreducible $T$-module $W$ in our decompositions, we give the endpoint of $W$, the multiplicity of $W$, the dimension of $W$, the diameter of $W$, the shape of $W$, and the isomorphism type of $W$. In addition, we give a pure basis for $W$ and the action of $A$ on that basis.
As we will see, $\dim E_r^* W =1$, where $r$ is the endpoint of $W$. By the \textit{seed vector} for $W$, we mean the unique vector in our basis that is contained in $E_r^* W$. 

We have a comment about how we display each graph $\Gamma$. We order the vertices of $\Gamma$ in a manner consistent with the subconstituents with respect to $x$. In order to increase the clarity of our figures, we will color every vertex. Vertices get the same color whenever they are in the same subconstituent with respect to $x$.

\renewcommand{\GraphNumber}{1}
\renewcommand{\GraphTitle}{$K_4$}
\renewcommand{\VertexCount}{4}
\renewcommand{\IntersectionArray}{\{3;1\}}
\renewcommand{\DateString}{2026.04.18 (orig. 2025.11)}

\section{\texorpdfstring{$K_4$}{K4}}\label{sec:k4}
Throughout this section, we take $\Gamma$ to be the graph $K_4$. $\Gamma$ has 4 vertices and intersection array $\{3;1\}$.
\begin{center}
    \framebox[0.5\linewidth]{\includegraphics[width=0.5\linewidth]{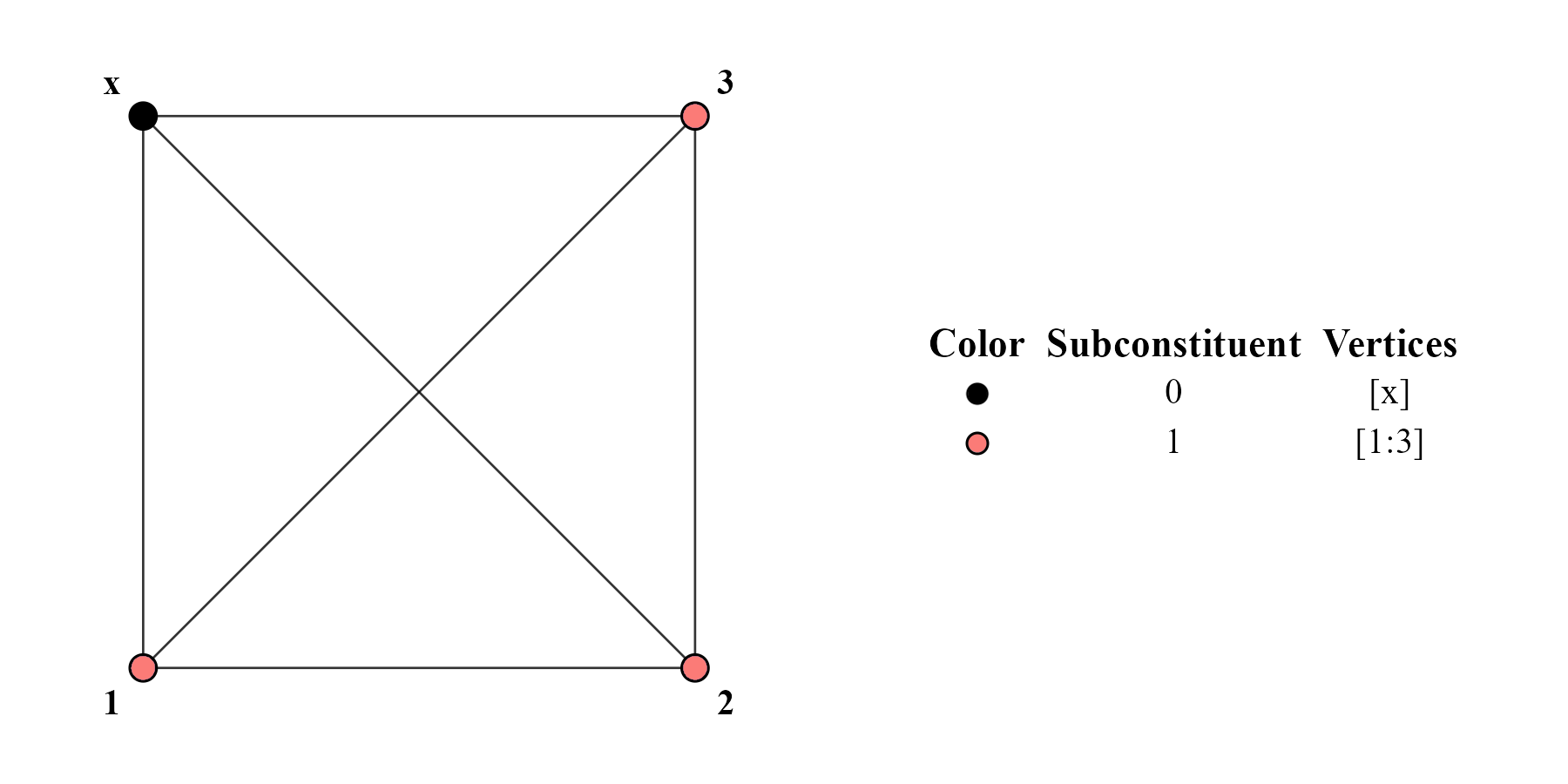}}
    \captionof{figure}{\textit{The graph $K_4$.
}}
\end{center}

The spectrum of the adjacency matrix $A$ is $3^1(-1)^3$.

We decompose the standard module $V$ of $\Gamma$ into an orthogonal direct sum of irreducible $T$-modules. Our decomposition has the form
\[ V = W_0 \oplus W_{1,a} \oplus W_{1,b}. \]
For each irreducible $T$-module in this decomposition, we give the endpoint, the multiplicity, the dimension, the diameter, the shape, the action of $A$ upon an appropriate pure basis, and the isomorphism type.

\begin{center}
\begin{tabular}{c | c c c c c c c}
\toprule
\textbf{Irred. $T$-modules} & \textbf{Endpt.} & \textbf{Mult.} & \textbf{Dim.} & \textbf{Diam.} & \textbf{Shape} & $A$ \textbf{action} & \textbf{Iso. type} \\
\midrule
$W_0$          & 0 & 1 & 2 & 1 & $(1,1)$ & \eqref{eq:k4-I} & I \\
$W_{1,a},\, W_{1,b}$  & 1 & 2 & 1 & 0 & $(1)$ & \eqref{eq:k4-II} & II \\
\bottomrule
\end{tabular}

\smallskip
\captionof{table}{\textit{Irreducible $T$-modules for $K_4$} \\
$(\dim T = 2^2 + 1^2 = 5)$}
\end{center}

For each irreducible $T$-module in the table above, we now give a pure basis and the matrix representing $A$ on that basis. 

\subsection{Endpoint 0}

We now describe the primary irreducible $T$-module $W_0$. The module $W_0$ has a basis $\{e_i^*\}_{i=0}^1$, where $e_i^*$ is from Definition \ref{def:primary}. With respect to this basis the matrix representing $A$ is
\begin{equation}
    A: \begin{bmatrix}
  0 & 3 \\
  1 & 2 
\end{bmatrix}.
\label{eq:k4-I}
\end{equation}
This matrix has eigenvalues $3, -1$.

\subsection{Endpoint 1}

We now describe the Type II irreducible $T$-modules in our decomposition. For Type II, the multiplicity is 2 and the modules are $W_{1,a}$ and $W_{1,b}$. For each module, our basis has the form $\{\nu\}$, where the seed vector $\nu$ is given below.
\begin{center}
    \begin{tabular}{c | c}
    \toprule
       \textbf{Module}  & \textbf{Essential part of }$\nu$\\
       \midrule
       $W_{1,a}$  & $\rowvector{1 & -1 & 0}$ \\[1.67pt]
       $W_{1,b}$  & $\rowvector{1 & 1 & -2}$ \\
    \bottomrule
    \end{tabular}
    \captionof{table}{\textit{The seed vector $\nu$ for each irreducible $T$-module of Type II.}}
\end{center}

With respect to each basis, the matrix representing $A$ is
\begin{equation}
    A: 
\begin{bmatrix}
  -1
\end{bmatrix}.
\label{eq:k4-II}
\end{equation}

This matrix has eigenvalue $-1$.


\renewcommand{\GraphNumber}{2}
\renewcommand{\GraphTitle}{$K_{3,3}$}
\renewcommand{\VertexCount}{6}
\renewcommand{\IntersectionArray}{\{3,2;1,3\}}
\renewcommand{\DateString}{2026.04.18 (orig. 2025.12.02)}

\section{\texorpdfstring{$K_{3,3}$}{K3,3}}
Throughout this section, we take $\Gamma$ to be the graph $K_{3,3}$. $\Gamma$ has 6 vertices and intersection array $\{3,2;1,3\}$.
\begin{center}
    \framebox[0.55\linewidth]{\includegraphics[width=0.55\linewidth]{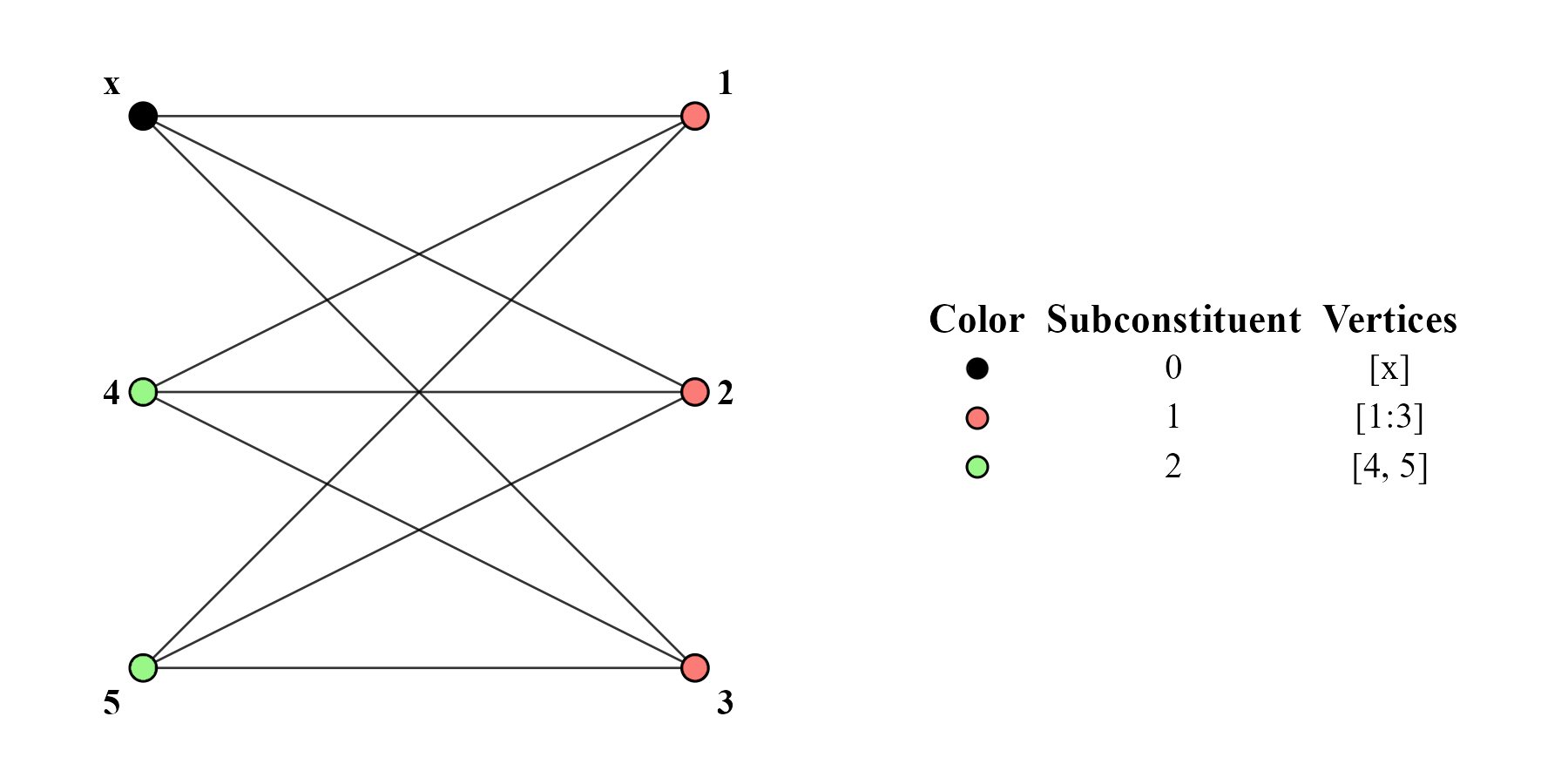}}
    \captionof{figure}{\textit{The graph $K_{3,3}$.
}}
\end{center}

The spectrum of the adjacency matrix $A$ is $3^10^4(-3)^1$.

We decompose the standard module $V$ of $\Gamma$ into an orthogonal direct sum of irreducible $T$-modules. Our decomposition has the form
\[ V = W_0 \oplus W_{1,a} \oplus W_{1,b} \oplus W_2. \]
For each irreducible $T$-module in this decomposition, we give the endpoint, the multiplicity, the dimension, the diameter, the shape, the action of $A$ upon an appropriate pure basis, and the isomorphism type.

\begin{center}
\begin{tabular}{c | c c c c c c c}
\toprule
\textbf{Irred. $T$-modules} & \textbf{Endpt.} & \textbf{Mult.} & \textbf{Dim.} & \textbf{Diam.} & \textbf{Shape} & $A$ \textbf{action} & \textbf{Iso. type} \\
\midrule
$W_0$ & 0 & 1 & 3 & 2 & $(1,1,1)$ & \eqref{eq:k33-I} & I \\
$W_{1,a}, W_{1,b}$ & 1 & 2 & 1 & 0 & $(1)$ & \eqref{eq:k33-II} & II \\
$W_{2}$ & 2 & 1 & 1 & 0 & $(1)$ & \eqref{eq:k33-II} & III \\
\bottomrule
\end{tabular}

\smallskip
\captionof{table}{\textit{Irreducible $T$-modules for $K_{3,3}$} \\ 
$(\dim T = 3^2+1^2+1^2=11)$}
\end{center}

For each irreducible $T$-module in the table above, we now give a pure basis and the matrix representing $A$ on that basis. 

\subsection{Endpoint 0}

We now describe the primary irreducible $T$-module $W_0$. The module $W_0$ has a basis $\{e_i^*\}_{i=0}^2$, where $e_i^*$ is from Definition \ref{def:primary}. With respect to this basis the matrix representing $A$ is
\begin{equation}
    A: \begin{bmatrix}
  0 & 3 & 0 \\
  1 & 0 & 2 \\
  0 & 3 & 0
\end{bmatrix}.\label{eq:k33-I}
\end{equation}
This matrix has eigenvalues $3, 0, -3$.

\subsection{Endpoint 1}

We now describe the Type II irreducible $T$-modules in our decomposition. For Type II, the multiplicity is 2 and the modules are $W_{1,a}$ and $W_{1,b}$. For each module, our basis has the form $\{\nu\}$, where the seed vector $\nu$ is given below.
\begin{center}
    \begin{tabular}{c | c}
    \toprule
       \textbf{Module}  & \textbf {Essential part of }$\nu$ \\
       \midrule
       $W_{1,a}$  & $\rowvector{1 & -1 & 0}$ \\[1.67pt]
       $W_{1,b}$  & $\rowvector{1 & 1 & -2}$ \\
    \bottomrule
    \end{tabular}
    \captionof{table}{\textit{The seed vector $\nu$ for each irreducible $T$-module of Type II.}}
\end{center}

With respect to each basis, the matrix representing $A$ is 
\begin{equation}
A: \begin{bmatrix}
  0
\end{bmatrix}.
\label{eq:k33-II}
\end{equation}

This matrix has eigenvalue $0$.

\subsection{Endpoint 2}

We now describe the Type III irreducible $T$-modules in our decomposition. For Type III, the multiplicity is 1 and the module is  $W_2$. For this module, our basis has the form $\{\nu\}$, where the seed vector $\nu$ is given below.
\begin{center}
    \begin{tabular}{c | c}
    \toprule
       \textbf{Module}  & \textbf {Essential part of }$\nu$ \\
       \midrule
       $W_{2}$  & $\rowvector{1 & -1}$ \\
    \bottomrule
    \end{tabular}
    \captionof{table}{\textit{The seed vector $\nu$ for the irreducible $T$-module of Type III.}}
\end{center}

With respect to this basis, the matrix representing $A$ is identical to \eqref{eq:k33-II}.


\renewcommand{\GraphNumber}{3}
\renewcommand{\GraphTitle}{Petersen Graph}
\renewcommand{\VertexCount}{10}
\renewcommand{\IntersectionArray}{\{3,2;1,1\}}
\renewcommand{\DateString}{2026.04.18 (orig. 2025.11.17)}

\section{Petersen Graph}
Throughout this section, we take $\Gamma$ to be the Petersen Graph. $\Gamma$ has 10 vertices and intersection array $\{3,2;1,1\}$.
\begin{center}
    \framebox[0.55\linewidth]{\includegraphics[width=0.55\linewidth]{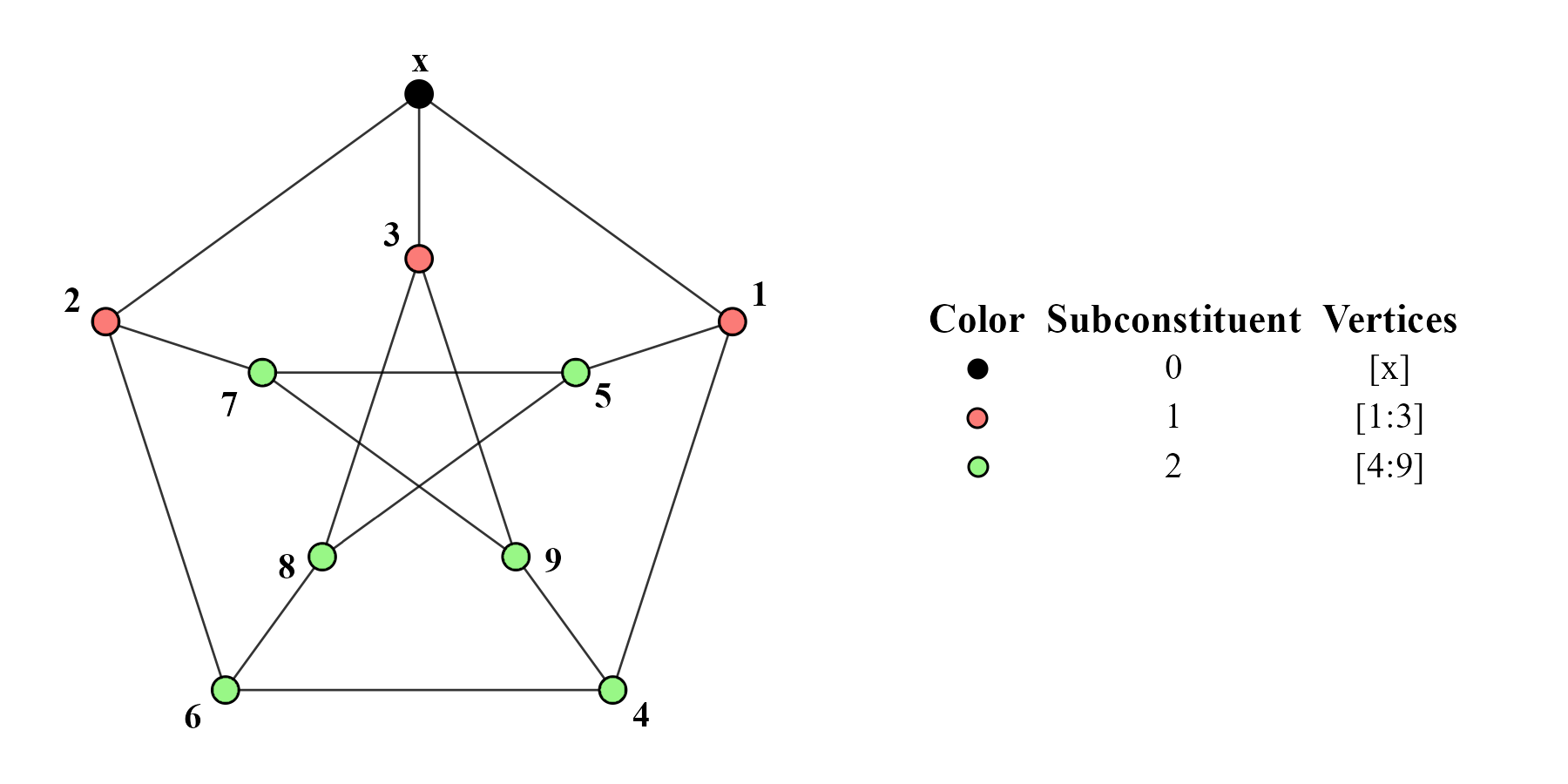}}
    \captionof{figure}{\textit{The Petersen Graph.
}}
\end{center}

The spectrum of the adjacency matrix $A$ is $3^11^5(-2)^4$.

We decompose the standard module $V$ of $\Gamma$ into an orthogonal direct sum of irreducible $T$-modules. Our decomposition has the form
\[ V = W_0 \oplus W_{1,a} \oplus W_{1,b} \oplus W_{2,a} \oplus W_{2,b} \oplus W_{2,c}. \]
For each irreducible $T$-module in this decomposition, we give the endpoint, the multiplicity, the dimension, the diameter, the shape, the action of $A$ upon an appropriate pure basis, and the isomorphism type.

\begin{center}
\begin{tabular}{c | c c c c c c c}
\toprule
\textbf{Irred. $T$-modules} & \textbf{Endpt.} & \textbf{Mult.} & \textbf{Dim.} & \textbf{Diam.} & \textbf{Shape} & $A$ \textbf{action} & \textbf{Iso. type} \\
\midrule
$W_0$ & 0 & 1 & 3 & 2 & $(1,1,1)$ & \eqref{eq:pet-I} & I \\
$W_{1,a}, W_{1,b}$ & 1 & 2 & 2 & 1 & $(1,1)$ & \eqref{eq:pet-II} & II \\
$W_{2,a},W_{2,b}$ & 2 & 2 & 1 & 0 & $(1)$ & \eqref{eq:pet-III} & III \\
$W_{2,c}$ & 2 & 1 & 1 & 0 & $(1)$ & \eqref{eq:pet-IV} & IV \\
\bottomrule
\end{tabular}

\smallskip
\captionof{table}{\textit{Irreducible $T$-modules for the Petersen Graph} \\ 
$(\dim T = 3^2+2^2+1^2+1^2=15)$}
\end{center}

For each irreducible $T$-module in the table above, we now give a pure basis and the matrix representing $A$ on that basis. 

\subsection{Endpoint 0}

We now describe the primary irreducible $T$-module $W_0$. The module $W_0$ has a basis $\{e_i^*\}_{i=0}^2$, where $e_i^*$ is from Definition \ref{def:primary}. With respect to this basis the matrix representing $A$ is
\begin{equation}
    A: \begin{bmatrix}
  0 & 3 & 0 \\
  1 & 0 & 2 \\
  0 & 1 & 2
\end{bmatrix}.
\label{eq:pet-I}
\end{equation}
This matrix has eigenvalues $3, 1, -2$.

\subsection{Endpoint 1}

We now describe the Type II irreducible $T$-modules in our decomposition. For Type II, the multiplicity is 2 and the modules are $W_{1,a}$ and $W_{1,b}$. For each module, our basis has the form $\{\nu, A\nu\}$, where the seed vector $\nu$ is given below.
\begin{center}
    \begin{tabular}{c | c}
    \toprule
       \textbf{Module}  & \textbf {Essential part of }$\nu$ \\
       \midrule
       $W_{1,a}$  & $\rowvector{1 & -1 & 0}$ \\[1.67pt]
       $W_{1,b}$  & $\rowvector{1 & 1 & -2}$ \\
    \bottomrule
    \end{tabular}
    \captionof{table}{\textit{The seed vector $\nu$ for each irreducible $T$-module of Type II. Note that $A\nu=E_2^*A\nu$.}}
\end{center}

With respect to each basis, the matrix representing $A$ is 
\begin{equation}
    A: \begin{bmatrix}
  0 & 2 \\
  1 & -1 \\
\end{bmatrix}.
\label{eq:pet-II}
\end{equation}

This matrix has eigenvalues $1, -2$.

\subsection{Endpoint 2}

We now describe the Type III irreducible $T$-modules in our decomposition. For Type III, the multiplicity is 2 and the modules are $W_{2,a}$ and $W_{2,b}$. For each module, our basis has the form $\{\nu\}$, where the seed vector $\nu$ is given below.
\begin{center}
    \begin{tabular}{c | c}
    \toprule
       \textbf{Module}  & \textbf {Essential part of }$\nu$ \\
       \midrule
       $W_{2,a}$  & $\rowvector{1&-1&1&-1&0&0}$ \\[1.67pt]
       $W_{2,b}$  & $\rowvector{1&-1&-1&1&-2&2}$ \\
    \bottomrule
    \end{tabular}
    \captionof{table}{\textit{The seed vector $\nu$ for each irreducible $T$-module of Type III.}}
\end{center}

With respect to each basis, the matrix representing $A$ is 
\begin{equation}
    A: \begin{bmatrix}
  1
\end{bmatrix}.
\label{eq:pet-III}
\end{equation}

This matrix has eigenvalue $1$.

We now describe the Type IV irreducible $T$-modules in our decomposition. For Type IV, the multiplicity is 1 and the module is  $W_{2,c}$. For this module, our basis has the form $\{\nu\}$, where the seed vector $\nu$ is given below.
\begin{center}
    \begin{tabular}{c | c}
    \toprule
       \textbf{Module}  & \textbf {Essential part of }$\nu$ \\
       \midrule
       $W_{2,c}$  & $\rowvector{1&-1&-1&1&1&-1}$ \\
    \bottomrule
    \end{tabular}
    \captionof{table}{\textit{The seed vector $\nu$ for the irreducible $T$-module of Type IV.}}
\end{center}

With respect to this basis, the matrix representing $A$ is
\begin{equation}
    A: \begin{bmatrix}
  -2
\end{bmatrix}.
\label{eq:pet-IV}
\end{equation}

This matrix has eigenvalue $-2$.


\renewcommand{\GraphNumber}{4}
\renewcommand{\GraphTitle}{3-cube}
\renewcommand{\VertexCount}{8}
\renewcommand{\IntersectionArray}{\{3,2,1;1,2,3\}}
\renewcommand{\DateString}{2026.04.18 (orig. 2025.12.02)}

\section{3-cube}
Throughout this section, we take $\Gamma$ to be the 3-cube. $\Gamma$ has 8 vertices and intersection array $\{3,2,1;1,2,3\}$.
\begin{center}
    \framebox[0.6\linewidth]{\includegraphics[width=0.6\linewidth]{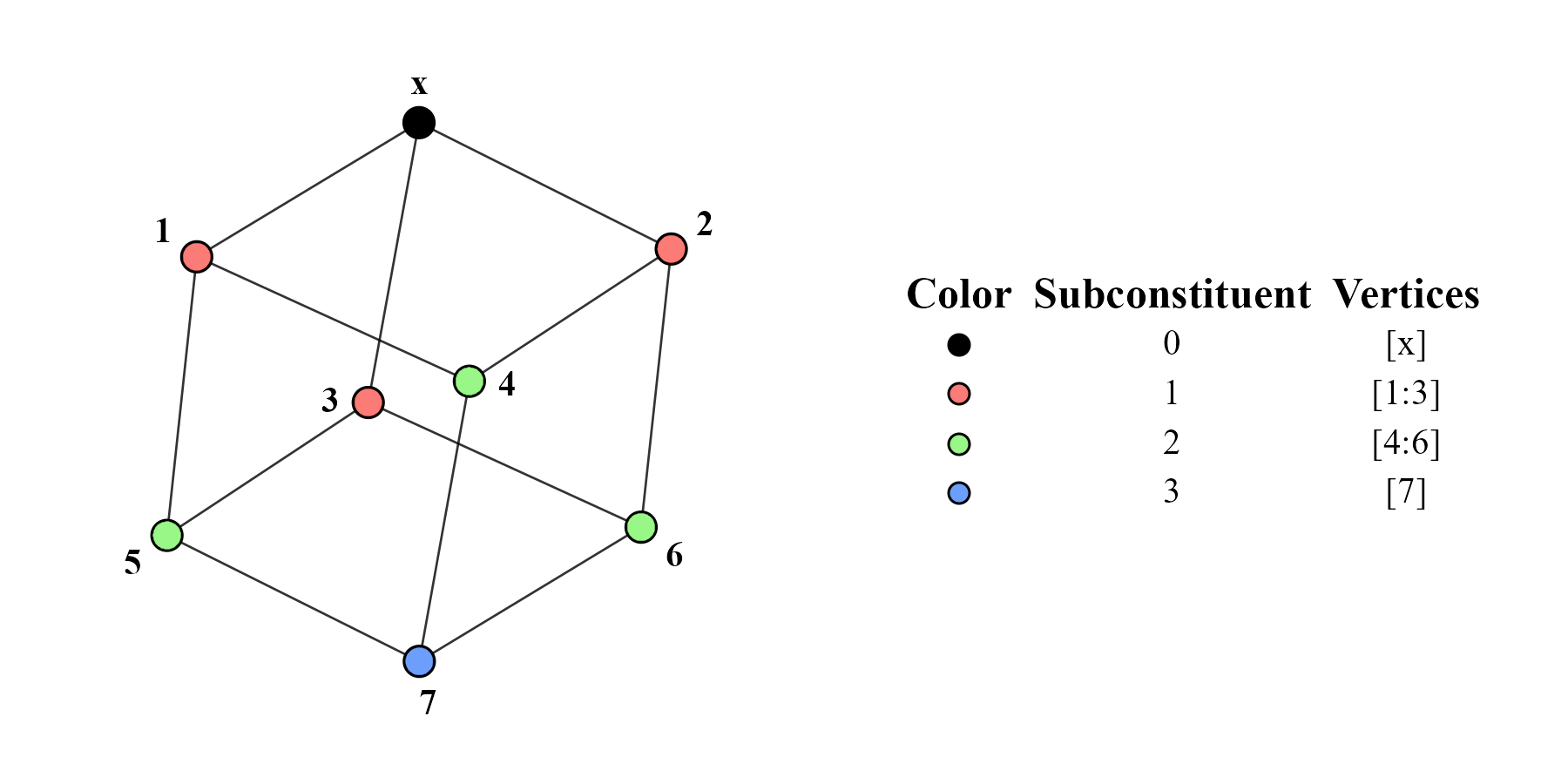}}
    \captionof{figure}{\textit{The $3$-cube.
}}
\end{center}

The spectrum of the adjacency matrix $A$ is $3^11^3(-1)^3(-3)^1$.

We decompose the standard module $V$ of $\Gamma$ into an orthogonal direct sum of irreducible $T$-modules. Our decomposition has the form
\[ V = W_0 \oplus W_{1,a} \oplus W_{1,b}. \]
For each irreducible $T$-module in this decomposition, we give the endpoint, the multiplicity, the dimension, the diameter, the shape, the action of $A$ upon an appropriate pure basis, and the isomorphism type.

\begin{center}
\begin{tabular}{c | c c c c c c c}
\toprule
\textbf{Irred. $T$-modules} & \textbf{Endpt.} & \textbf{Mult.} & \textbf{Dim.} & \textbf{Diam.} & \textbf{Shape} & $A$ \textbf{action} & \textbf{Iso. type} \\
\midrule
$W_0$ & 0 & 1 & 4 & 3 & $(1,1,1,1)$ & \eqref{eq:3cube-I} & I \\
$W_{1,a}, W_{1,b}$ & 1 & 2 & 2 & 1 & $(1,1)$ & \eqref{eq:3cube-II} & II \\
\bottomrule
\end{tabular}

\smallskip
\captionof{table}{\textit{Irreducible $T$-modules for the $3$-cube} \\ 
$(\dim T = 4^2+2^2=20)$}
\end{center}

For each irreducible $T$-module in the table above, we now give a pure basis and the matrix representing $A$ on that basis. 

\subsection{Endpoint 0}

We now describe the primary irreducible $T$-module $W_0$. The module $W_0$ has a basis $\{e_i^*\}_{i=0}^3$, where $e_i^*$ is from Definition \ref{def:primary}. With respect to this basis the matrix representing $A$ is
\begin{equation}
    A: \begin{bmatrix}
  0 & 3 & 0 & 0 \\
  1 & 0 & 2 & 0 \\
  0 & 2 & 0 & 1 \\
  0 & 0 & 3 & 0 \\
\end{bmatrix}.
\label{eq:3cube-I}
\end{equation}
This matrix has eigenvalues $3, 1, -1, -3$.

\subsection{Endpoint 1}

We now describe the Type II irreducible $T$-modules in our decomposition. For Type II, the multiplicity is 2 and the modules are $W_{1,a}$ and $W_{1,b}$. For each module, our basis has the form $\{\nu, A\nu\}$, where the seed vector $\nu$ is given below.
\begin{center}
    \begin{tabular}{c | c}
    \toprule
       \textbf{Module}  & \textbf {Essential part of }$\nu$ \\
       \midrule
       $W_{1,a}$  & $\rowvector{1 & -1 & 0}$ \\[1.67pt]
       $W_{1,b}$  & $\rowvector{1 & 1 & -2}$ \\
    \bottomrule
    \end{tabular}
    \captionof{table}{\textit{The seed vector $\nu$ for each irreducible $T$-module of Type II. Note that $A\nu=E_2^*A\nu$.}}
\end{center}

With respect to each basis, the matrix representing $A$ is 
\begin{equation}
    A: \begin{bmatrix}
  0 & 1 \\
  1 & 0 \\
\end{bmatrix}.
\label{eq:3cube-II}
\end{equation}

This matrix has eigenvalues $1, -1$.


\renewcommand{\GraphNumber}{5}
\renewcommand{\GraphTitle}{Heawood Graph}
\renewcommand{\VertexCount}{14}
\renewcommand{\IntersectionArray}{\{3,2,2;1,1,3\}}
\renewcommand{\DateString}{2026.04.11 (orig. 2025.11.09)}

\section{Heawood Graph}
Throughout this section, we take $\Gamma$ to be the Heawood Graph. The graph $\Gamma$ has 14 vertices and intersection array $\{3,2,2;1,1,3\}$.
\begin{center}
    \framebox[0.72\linewidth]{\includegraphics[width=0.72\linewidth]{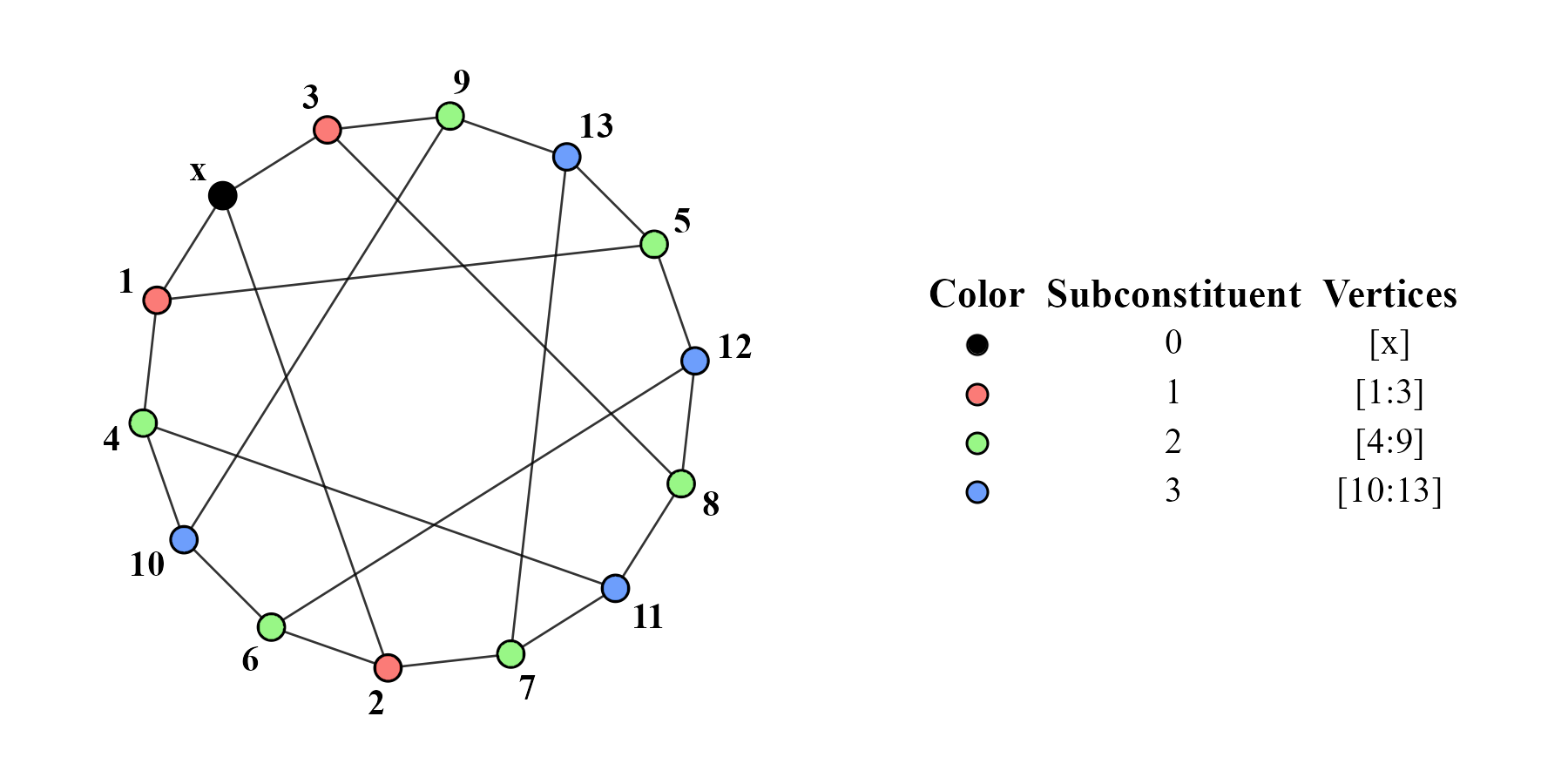}}
    \captionof{figure}{\textit{The Heawood Graph.
}}
\end{center}

The spectrum of the adjacency matrix $A$ is \( 3^1(\sqrt2)^6(-\sqrt2)^6(-3)^1 \).

We decompose the standard module $V$ of $\Gamma$ into an orthogonal direct sum of irreducible $T$-modules. Our decomposition has the form
\[ V = W_0 \oplus W_{1,a} \oplus W_{1,b} \oplus W_{2,a} \oplus W_{2,b} \oplus W_{2,c}. \]
For each irreducible $T$-module in this decomposition, we give the endpoint, the multiplicity, the dimension, the diameter, the shape, the action of $A$ upon an appropriate pure basis, and the isomorphism type.

\begin{center}
\begin{tabular}{c | c c c c c c c}
\toprule
\textbf{Irred. $T$-modules} & \textbf{Endpt.} & \textbf{Mult.} & \textbf{Dim.} & \textbf{Diam.} & \textbf{Shape} & $A$ \textbf{action} & \textbf{Iso. type} \\
\midrule
$W_0$ & 0 & 1 & 4 & 3 & $(1,1,1,1)$ & \eqref{eq:hea-I} & I \\
$W_{1,a}, W_{1,b}$ & 1 & 2 & 2 & 1 & $(1,1)$ & \eqref{eq:hea-II} & II \\
$W_{2,a}, W_{2,b}, W_{2,c}$ & 2 & 3 & 2 & 1 & $(1,1)$ & \eqref{eq:hea-II} & III \\
\bottomrule
\end{tabular}

\smallskip
\captionof{table}{\textit{Irreducible $T$-modules for the Heawood Graph} \\ 
$(\dim T = 4^2+2^2+2^2=24)$}
\end{center}

For each irreducible $T$-module in the table above, we now give a pure basis and the matrix representing $A$ on that basis. 

\subsection{Endpoint 0}

We now describe the primary irreducible $T$-module $W_0$. The module $W_0$ has a basis $\{e_i^*\}_{i=0}^3$, where $e_i^*$ is from Definition \ref{def:primary}. With respect to this basis the matrix representing $A$ is
\begin{equation}
    A: \begin{bmatrix}
  0 & 3 & 0 & 0 \\
  1 & 0 & 2 & 0 \\
  0 & 1 & 0 & 2 \\
  0 & 0 & 3 & 0 \\
\end{bmatrix}.
\label{eq:hea-I}
\end{equation}
This matrix has eigenvalues $3, \sqrt{2}, -\sqrt{2}, -3$.

\subsection{Endpoint 1}

We now describe the Type II irreducible $T$-modules in our decomposition. For Type II, the multiplicity is 2 and the modules are $W_{1,a}$ and $W_{1,b}$. For each module, our basis has the form $\{\nu, A\nu\}$, where the seed vector $\nu$ is given below.
\begin{center}
    \begin{tabular}{c | c}
    \toprule
       \textbf{Module}  & \textbf {Essential part of }$\nu$ \\
       \midrule
       $W_{1,a}$  & $\rowvector{1 & -1 & 0}$ \\[1.67pt]
       $W_{1,b}$  & $\rowvector{1 & 1 & -2}$ \\
    \bottomrule
    \end{tabular}
    \captionof{table}{\textit{The seed vector $\nu$ for each irreducible $T$-module of Type II. Note that $A\nu=E_2^*A\nu$.}}
\end{center}

With respect to each basis, the matrix representing $A$ is 
\begin{equation}
    A: \begin{bmatrix}
  0 & 2 \\
  1 & 0 \\
\end{bmatrix}.
\label{eq:hea-II}
\end{equation}

This matrix has eigenvalues $\sqrt{2}, -\sqrt{2}$.

\subsection{Endpoint 2}

We now describe the Type III irreducible $T$-modules in our decomposition. For Type III, the multiplicity is 3 and the modules are $W_{2,a}, W_{2,b}, W_{2,c}$. For each module, our basis has the form $\{\nu, A\nu\}$, where the seed vector $\nu$ is given below.
\begin{center}
    \begin{tabular}{c | c}
    \toprule
       \textbf{Module}  & \textbf {Essential part of }$\nu$ \\
       \midrule
       $W_{2,a}$  & $\rowvector{1 & -1 & 0 & 0 &0&0}$ \\[1.67pt]
       $W_{2,b}$  & $\rowvector{0 & 0 & 1 & -1 &0&0}$ \\[1.67pt]
       $W_{2,c}$  & $\rowvector{0 & 0 & 0 & 0 &1&-1}$ \\
    \bottomrule
    \end{tabular}
    \captionof{table}{\textit{The seed vector $\nu$ for each irreducible $T$-module of Type III. Note that $A\nu=E_3^*A\nu$.}}
\end{center}

With respect to each basis, the matrix representing $A$ is identical to \eqref{eq:hea-II}.


\renewcommand{\GraphNumber}{6}
\renewcommand{\GraphTitle}{Pappus Graph}
\renewcommand{\VertexCount}{18}
\renewcommand{\IntersectionArray}{\{3,2,2,1;1,1,2,3\}}
\renewcommand{\DateString}{2026.04.11 (orig. 2025.11.11)}

\section{Pappus Graph}
Throughout this section, we take $\Gamma$ to be the Pappus Graph. $\Gamma$ has 18 vertices and intersection array $\{3,2,2,1;1,1,2,3\}$.
\begin{center}
    \framebox[0.7\linewidth]{\includegraphics[width=0.7\linewidth]{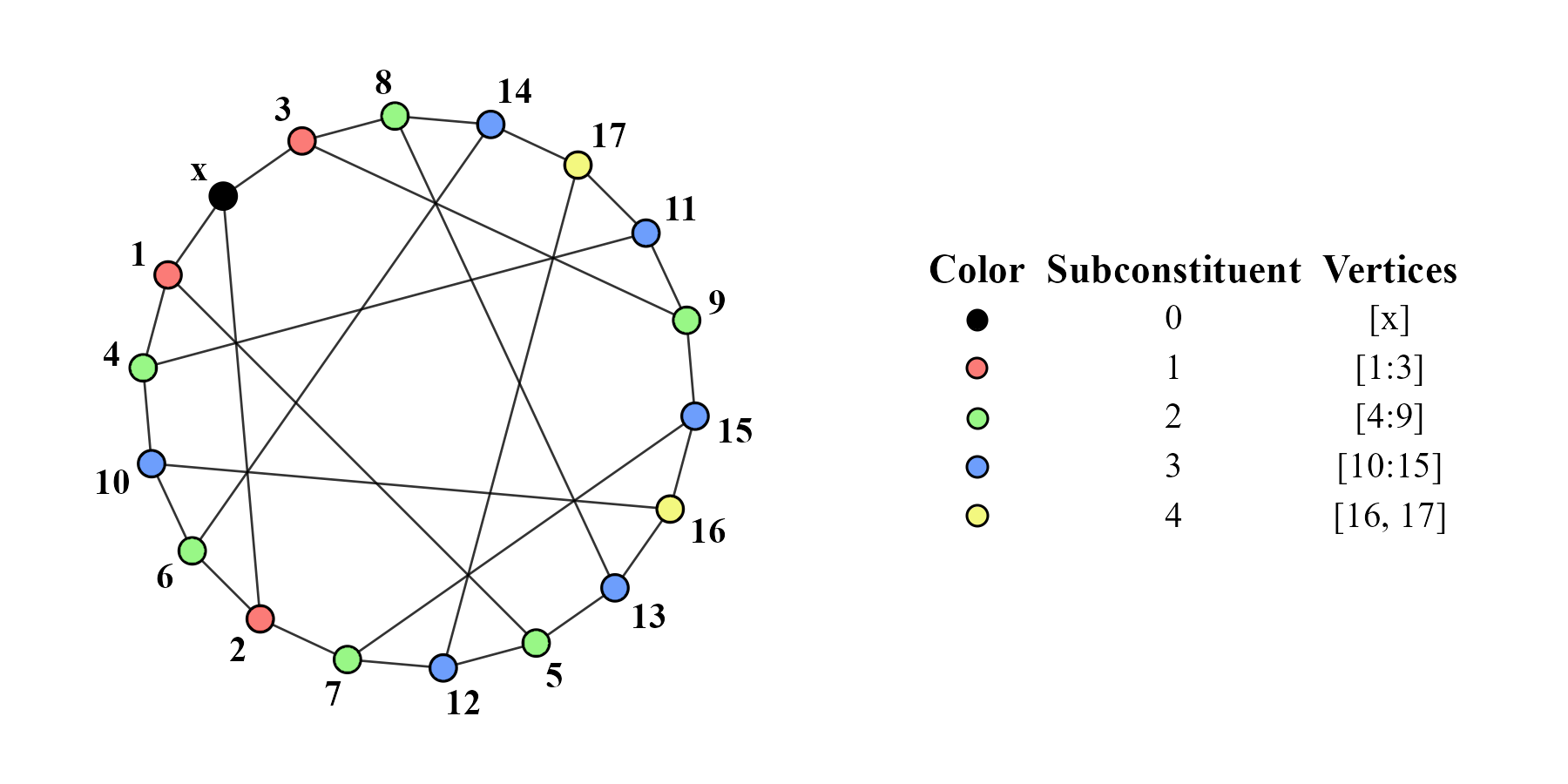}}
    \captionof{figure}{\textit{The Pappus Graph.
}}
\end{center}

The spectrum of the adjacency matrix $A$ is $ 3^1(\sqrt3)^60^4(-\sqrt3)^6(-3)^1 $.

We decompose the standard module $V$ of $\Gamma$ into an orthogonal direct sum of irreducible $T$-modules. Our decomposition has the form
\[ V = W_0 \oplus W_{1,a} \oplus W_{1,b} \oplus W_{2,a} \oplus W_{2,b} \oplus W_{2,c} \oplus W_3. \]
For each irreducible $T$-module in this decomposition, we give the endpoint, the multiplicity, the dimension, the diameter, the shape, the action of $A$ upon an appropriate pure basis, and the isomorphism type.

\begin{center}
\begin{tabular}{c | c c c c c c c}
\toprule
\textbf{Irred. $T$-modules} & \textbf{Endpt.} & \textbf{Mult.} & \textbf{Dim.} & \textbf{Diam.} & \textbf{Shape} & $A$ \textbf{action} & \textbf{Iso. type} \\
\midrule
$W_0$ & 0 & 1 & 5 & 4 & $(1,1,1,1,1)$ & \eqref{eq:pap-I} & I \\
$W_{1,a}, W_{1,b}$ & 1 & 2 & 3 & 2 & $(1,1,1)$ & \eqref{eq:pap-II} & II \\
$W_{2,a}$ & 2 & 1 & 1 & 0 & $(1)$ & \eqref{eq:pap-III} & III \\
$W_{2,b}, W_{2,c}$ & 2 & 2 & 2 & 1 & $(1,1)$ & \eqref{eq:pap-IV} & IV \\
$W_{3}$ & 3 & 1 & 2 & 1 & $(1,1)$ & \eqref{eq:pap-IV} & V \\
\bottomrule
\end{tabular}

\smallskip
\captionof{table}{\textit{Irreducible $T$-modules for the Pappus Graph} \\
$(\dim T = 5^2+3^2+1^2+2^2+2^2=43)$}
\end{center}

For each irreducible $T$-module in the table above, we now give a pure basis and the matrix representing $A$ on that basis.

\subsection{Endpoint 0}

We now describe the primary irreducible $T$-module $W_0$. The module $W_0$ has a basis $\{e_i^*\}_{i=0}^5$, where $e_i^*$ is from Definition \ref{def:primary}. With respect to this basis the matrix representing $A$ is
\begin{equation}
    A: \begin{bmatrix}
0 & 3 & 0 & 0 & 0 \\
1 & 0 & 2 & 0 & 0 \\
0 & 1 & 0 & 2 & 0 \\
0 & 0 & 2 & 0 & 1 \\
0 & 0 & 0 & 3 & 0 \\
\end{bmatrix}.
\label{eq:pap-I}
\end{equation}
This matrix has eigenvalues $3, \sqrt{3}, 0, -\sqrt{3}, -3$.

\subsection{Endpoint 1}

We now describe the Type II irreducible $T$-modules in our decomposition. For Type II, the multiplicity is 2 and the modules are $W_{1,a}$ and $W_{1,b}$. For each module, our basis has the form $\{\nu, A\nu, E_3^*A^2\nu\}$, where the seed vector $\nu$ is given below.
\begin{center}
    \begin{tabular}{c | c}
    \toprule
       \textbf{Module}  & \textbf {Essential part of }$\nu$ \\
       \midrule
       $W_{1,a}$  & $\rowvector{1 & -1 & 0}$ \\[1.67pt]
       $W_{1,b}$  & $\rowvector{1 & 1 & -2}$ \\
    \bottomrule
    \end{tabular}
    \captionof{table}{\textit{The seed vector $\nu$ for each irreducible $T$-module of Type II. Note that $A\nu=E_2^*A\nu$.}}
\end{center}

With respect to each basis, the matrix representing $A$ is 
\begin{equation}
    A: \begin{bmatrix}
  0 & 2 & 0 \\
  1 & 0 & 1 \\
  0 & 1 & 0 \\
\end{bmatrix}.
\label{eq:pap-II}
\end{equation}

This matrix has eigenvalues $\sqrt{3}, 0, -\sqrt{3}$.

\subsection{Endpoint 2}

We now describe the Type III irreducible $T$-modules in our decomposition. For Type III, the multiplicity is 1 and the module is  $W_{2,a}$. For this module, our basis has the form $\{\nu\}$, where the seed vector $\nu$ is given below.
\begin{center}
    \begin{tabular}{c | c}
    \toprule
       \textbf{Module}  & \textbf {Essential part of }$\nu$ \\
       \midrule
       $W_{2,a}$  & $\rowvector{-1 & 1 & 1 & -1 &-1 &1}$ \\
    \bottomrule
    \end{tabular}
    \captionof{table}{\textit{The seed vector $\nu$ for the irreducible $T$-module of Type III.}}
\end{center}

With respect to this basis, the matrix representing $A$ is
\begin{equation}
    A: \begin{bmatrix}
  0
\end{bmatrix}.
\label{eq:pap-III}
\end{equation}

This matrix has eigenvalue $0$.

We now describe the Type IV irreducible $T$-modules in our decomposition. For Type IV, the multiplicity is 2 and the modules are $W_{2,b}$ and $W_{2,c}$. For each module, our basis has the form $\{\nu, A\nu\}$, where the seed vector $\nu$ is given below.
\begin{center}
    \begin{tabular}{c | c}
    \toprule
       \textbf{Module}  & \textbf {Essential part of }$\nu$ \\
       \midrule
       $W_{2,b}$  & $\rowvector{1 & -1 & 0 & 0& -1&1}$ \\[1.67pt]
       $W_{2,c}$  & $\rowvector{1 & -1 & 2 & -2& 1&-1}$ \\
    \bottomrule
    \end{tabular}
    \captionof{table}{\textit{The seed vector $\nu$ for each irreducible $T$-module of Type IV. Note that $A\nu=E_3^*A\nu$.}}
\end{center}

With respect to each basis, the matrix representing $A$ is 
\begin{equation}
    A: \begin{bmatrix}
  0 & 3 \\
  1 & 0 
\end{bmatrix}.
\label{eq:pap-IV}
\end{equation}

This matrix has eigenvalues $\sqrt{3}, -\sqrt{3}$.

\subsection{Endpoint 3}

We now describe the Type V irreducible $T$-modules in our decomposition. For Type V, the multiplicity is 1 and the module is  $W_3$. For this module, our basis has the form $\{\nu, A\nu\}$, where the seed vector $\nu$ is given below.
\begin{center}
    \begin{tabular}{c | c}
    \toprule
       \textbf{Module}  & \textbf {Essential part of }$\nu$ \\
       \midrule
       $W_{3}$  & $\rowvector{1 & -1 & -1 & 1& -1&1}$ \\
    \bottomrule
    \end{tabular}
    \captionof{table}{\textit{The seed vector $\nu$ for the irreducible $T$-module of Type V. Note that $A\nu=E_4^*A\nu$.}}
\end{center}

With respect to this basis, the matrix representing $A$ is identical to \eqref{eq:pap-IV}.


\renewcommand{\GraphNumber}{7}
\renewcommand{\GraphTitle}{Coxeter Graph}
\renewcommand{\VertexCount}{28}
\renewcommand{\IntersectionArray}{\{3,2,2,1;1,1,1,2\}}
\renewcommand{\DateString}{2026.04.18 (orig. 2026.02.01)}

\section{Coxeter Graph}
Throughout this section, we take $\Gamma$ to be the Coxeter Graph. $\Gamma$ has 28 vertices and intersection array $\{3,2,2,1;1,1,1,2\}$.
\begin{center}
    \framebox[0.7\linewidth]{\includegraphics[width=0.75\linewidth]{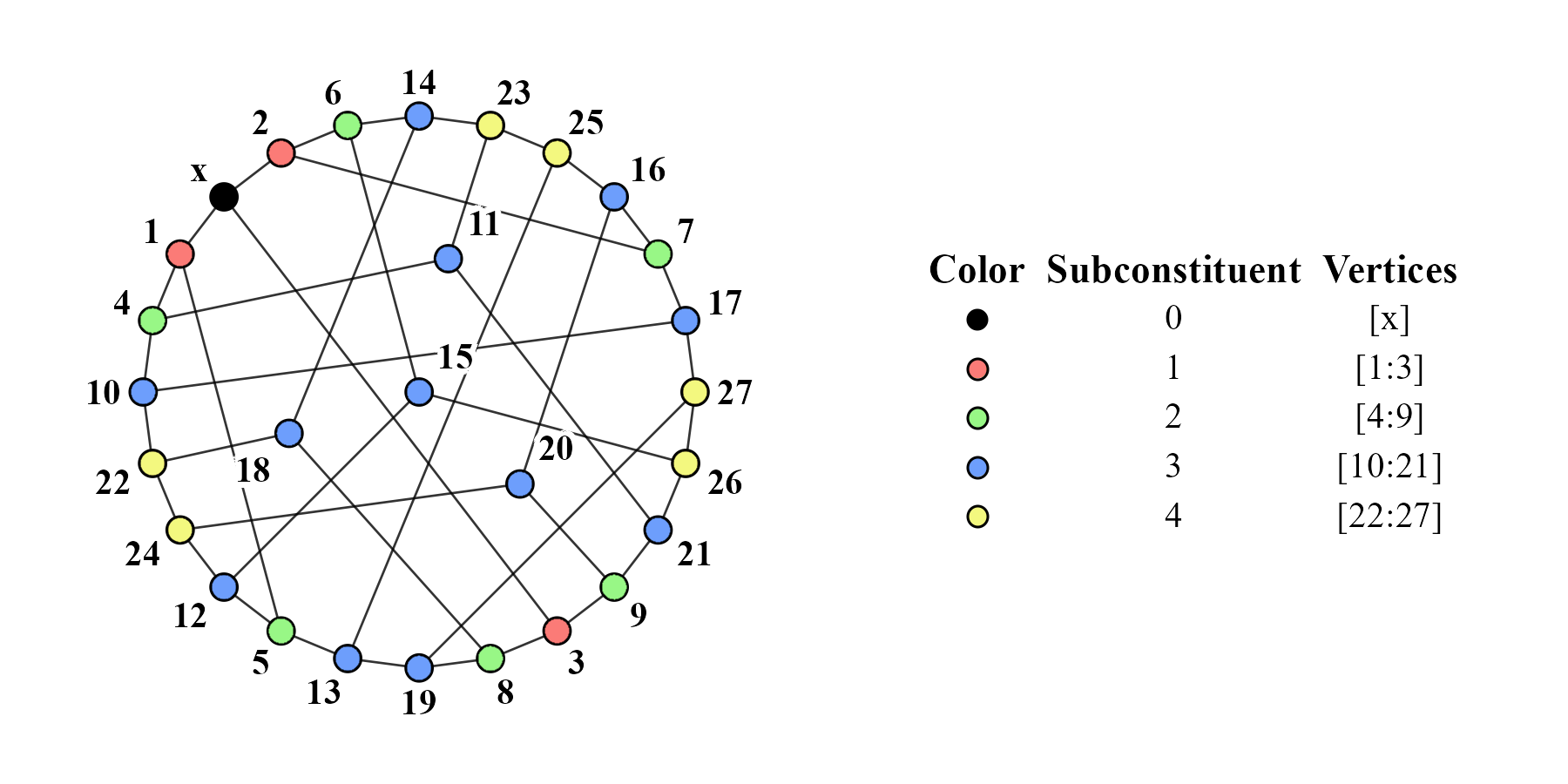}}
    \captionof{figure}{\textit{The Coxeter Graph.
}}
\end{center}

The spectrum of the adjacency matrix $A$ is $ 3^12^8(-1+\sqrt2)^6(-1)^7(-1-\sqrt2)^6 $.

We decompose the standard module $V$ of $\Gamma$ into an orthogonal direct sum of irreducible $T$-modules. Our decomposition has the form
\[ V = W_0 \oplus W_{1,a} \oplus W_{1,b} \oplus W_{2,a} \oplus W_{2,b} \oplus W_{2,c} \oplus W_{3,a} \oplus W_{3,b}. \]
For each irreducible $T$-module in this decomposition, we give the endpoint, the multiplicity, the dimension, the diameter, the shape, the action of $A$ upon an appropriate pure basis, and the isomorphism type.

\begin{center}
\begin{tabular}{c | c c c c c c c}
\toprule
\textbf{Irred. $T$-modules} & \textbf{Endpt.} & \textbf{Mult.} & \textbf{Dim.} & \textbf{Diam.} & \textbf{Shape} & $A$ \textbf{action} & \textbf{Iso. type} \\
\midrule
$W_0$ & 0 & 1 & 5 & 4 & $(1,1,1,1,1)$ & \eqref{eq:cox-I} & I \\
$W_{1,a}, W_{1,b}$ & 1 & 2 & 5 & 3 & $(1,1,2,1)$ & \eqref{eq:cox-II} & II \\
$W_{2,a}$ & 2 & 1 & 2 & 1 & $(1,1)$ & \eqref{eq:cox-III} & III \\
$W_{2,b}, W_{2,c}$ & 2 & 2 & 4 & 2 & $(1,2,1)$ & \eqref{eq:cox-IV} & IV \\
$W_{3,a}$ & 3 & 1 & 1 & 0 & $(1)$ & \eqref{eq:cox-V} & V \\
$W_{3,b}$ & 3 & 1 & 2 & 1 & $(1,1)$ & \eqref{eq:cox-VI} & VI \\
\bottomrule
\end{tabular}

\smallskip
\captionof{table}{\textit{Irreducible $T$-modules for the Coxeter Graph} \\
$(\dim T = 5^2 + 5^2+2^2+4^2+1^2+2^2=75)$}
\end{center}

For each irreducible $T$-module in the table above, we now give a pure basis and the matrix representing $A$ on that basis. 

\subsection{Endpoint 0}

We now describe the primary irreducible $T$-module $W_0$. The module $W_0$ has a basis $\{e_i^*\}_{i=0}^4$, where $e_i^*$ is from Definition \ref{def:primary}. With respect to this basis the matrix representing $A$ is
\begin{equation}
    A: \begin{bmatrix}
0 & 3 & 0 & 0 & 0 \\
1 & 0 & 2 & 0 & 0 \\
0 & 1 & 0 & 2 & 0 \\
0 & 0 & 1 & 1 & 1 \\
0 & 0 & 0 & 2 & 1 \\
\end{bmatrix}.
\label{eq:cox-I}
\end{equation}
This matrix has eigenvalues $3, 2, -1+\sqrt{2}, -1, -1-\sqrt{2}$.

\subsection{Endpoint 1}

We now describe the Type II irreducible $T$-modules in our decomposition. For Type II, the multiplicity is 2 and the modules are $W_{1,a}$ and $W_{1,b}$. For each module, our basis has the form $\{\nu, A\nu, E_3^*A^2\nu, E_3^*A^3\nu, E_4^*A^3\nu\}$, where the seed vector $\nu$ is given below.
\begin{center}
    \begin{tabular}{c | c}
    \toprule
       \textbf{Module}  & \textbf {Essential part of }$\nu$ \\
       \midrule
       $W_{1,a}$  & $\rowvector{1 & -1 & 0}$ \\[1.67pt]
       $W_{1,b}$  & $\rowvector{1 & 1 & -2}$ \\
    \bottomrule
    \end{tabular}
    \captionof{table}{\textit{The seed vector $\nu$ for each irreducible $T$-module of Type II. Note that $A\nu=E_2^*A\nu$.}}
\end{center}

With respect to each basis, the matrix representing $A$ is 
\begin{equation}
    A: \begin{bmatrix}
  0 & 2 & 0 & 0 & 0 \\
  1 & 0 & 2 & -1 & 0 \\
  0 & 1 & 0 & 1 & 0 \\
  0 & 0 & 1 & 0 & -1 \\
  0 & 0 & 1 & -2 & 1 \\
\end{bmatrix}.
\label{eq:cox-II}
\end{equation}

This matrix has eigenvalues $2,2, -1+\sqrt{2}, -1, -1-\sqrt{2}$.

\subsection{Endpoint 2}

We now describe the Type III irreducible $T$-modules in our decomposition. For Type III, the multiplicity is 1 and the module is  $W_{2,a}$. For this module, our basis has the form $\{\nu, A\nu\}$, where the seed vector $\nu$ is given below.
\begin{center}
    \begin{tabular}{c | c}
    \toprule
       \textbf{Module}  & \textbf {Essential part of }$\nu$ \\
       \midrule
       $W_{2,a}$  & $\rowvector{1 & -1 & -1 & 1 & -1 & 1}$ \\
    \bottomrule
    \end{tabular}
    \captionof{table}{\textit{The seed vector $\nu$ for the irreducible $T$-module of Type III. Note that $A\nu=E_3^*A\nu$.}}
\end{center}

With respect to this basis, the matrix representing $A$ is
\begin{equation}
    A: \begin{bmatrix}
  0 & 2 \\
  1 & 1
\end{bmatrix}.
\label{eq:cox-III}
\end{equation}

This matrix has eigenvalues $2, -1$.

We now describe the Type IV irreducible $T$-modules in our decomposition. For Type IV, the multiplicity is 2 and the modules are $W_{2,b}$ and $W_{2,c}$. For each module, our basis has the form $\{\nu, A\nu, E_3^*A^2\nu,E_4^*A^2\nu\}$, where the seed vector $\nu$ is given below.
\begin{center}
    \begin{tabular}{c | c}
    \toprule
       \textbf{Module}  & \textbf {Essential part of }$\nu$ \\
       \midrule
       $W_{2,b}$  & $\rowvector{1 & -1 & 0 & 0 & 1 & -1}$ \\[1.67pt]
       $W_{2,c}$  & $\rowvector{1 & -1 & 2 & -2 & -1 & 1}$ \\
    \bottomrule
    \end{tabular}
    \captionof{table}{\textit{The seed vector $\nu$ for each irreducible $T$-module of Type IV. Note that $A\nu=E_3^*A\nu$.}}
\end{center}

With respect to each basis, the matrix representing $A$ is 
\begin{equation}
    A: \begin{bmatrix}
  0 & 2 & -1 & 0 \\
  1 & 0 & 1 & 2 \\
  0 & 1 & 0 & 1 \\
  0 & 1 & 0 & -1 \\
\end{bmatrix}.
\label{eq:cox-IV}
\end{equation}

This matrix has eigenvalues $2, -1 + \sqrt2, -1, -1-\sqrt2$.

\subsection{Endpoint 3}

We now describe the Type V irreducible $T$-modules in our decomposition. For Type V, the multiplicity is 1 and the module is  $W_{3,a}$. For this module, our basis has the form $\{\nu\}$, where the seed vector $\nu$ is given below.
\begin{center}
    \begin{tabular}{c | c}
    \toprule
       \textbf{Module}  & \textbf {Essential part of }$\nu$ \\
       \midrule
       $W_{3,a}$  & $\rowvector{-1 & 1 & -1 & 1 & -1 & 1 & -1 & 1 & 1 & -1 & 1 & -1}$ \\
    \bottomrule
    \end{tabular}
    \captionof{table}{\textit{The seed vector $\nu$ for the irreducible $T$-module of Type V.}}
\end{center}

With respect to this basis, the matrix representing $A$ is
\begin{equation}
    A: \begin{bmatrix}
  -1
\end{bmatrix}.
\label{eq:cox-V}
\end{equation}

This matrix has eigenvalue $-1$.

We now describe the Type VI irreducible $T$-modules in our decomposition. For Type VI, the multiplicity is 1 and the module is  $W_{3,b}$. For this module, our basis has the form $\{\nu, E_4^*A\nu\}$, where the seed vector $\nu$ is given below.
\begin{center}
    \begin{tabular}{c | c}
    \toprule
       \textbf{Module}  & \textbf {Essential part of }$\nu$ \\
       \midrule
       $W_{3,b}$  & $\rowvector{-1 & 1 & 1 & -1 & 1 & -1 & -1 & 1 & -1 & 1 & 1 & -1}$ \\
    \bottomrule
    \end{tabular}
    \captionof{table}{\textit{The seed vector $\nu$ for the irreducible $T$-module of Type VI.}}
\end{center}

With respect to this basis, the matrix representing $A$ is
\begin{equation}
    A: \begin{bmatrix}
  -1 & 2\\
  1 & -1
\end{bmatrix}.
\label{eq:cox-VI}
\end{equation}

This matrix has eigenvalues $-1+\sqrt{2}, -1-\sqrt{2}$.


\renewcommand{\GraphNumber}{8}
\renewcommand{\GraphTitle}{Tutte's 8-cage}
\renewcommand{\VertexCount}{30}
\renewcommand{\IntersectionArray}{\{3,2,2,2;1,1,1,3\}}
\renewcommand{\DateString}{2026.04.11 (orig. 2026.01.21)}

\section{Tutte's 8-cage}
Throughout this section, we take $\Gamma$ to be Tutte's 8-cage. $\Gamma$ has 30 vertices and intersection array $\{3,2,2,2;1,1,1,3\}$.
\begin{center}
    \framebox[0.77\linewidth]{\includegraphics[width=0.8\linewidth]{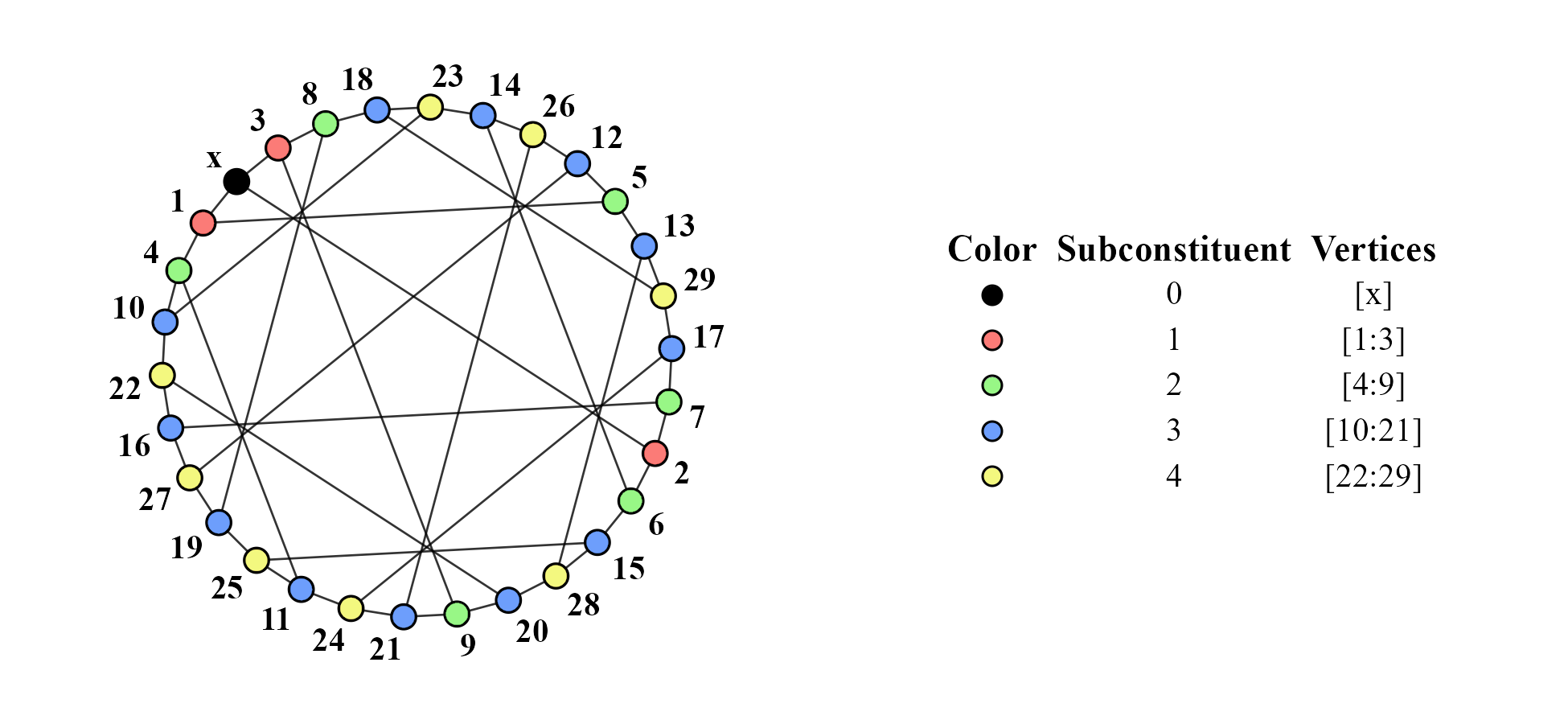}}
    \captionof{figure}{\textit{Tutte's 8-cage.
}}
\end{center}

The spectrum of the adjacency matrix $A$ is $3^12^90^{10}(-2)^9(-3)^1$.

We decompose the standard module $V$ of $\Gamma$ into an orthogonal direct sum of irreducible $T$-modules. Our decomposition has the form
\[ V = W_0 \oplus W_{1,a} \oplus W_{1,b} \oplus W_{2,a} \oplus W_{2,b} \oplus W_{2,c} \oplus W_{3,a} \oplus W_{3,b} \oplus W_{3,c} \oplus W_{3,d} \oplus W_{3,e} \oplus W_{3,f} \oplus W_4. \]
For each irreducible $T$-module in this decomposition, we give the endpoint, the multiplicity, the dimension, the diameter, the shape, the action of $A$ upon an appropriate pure basis, and the isomorphism type.

\begin{center}
\begin{tabular}{c | c c c c c c c}
\toprule
\textbf{Irred. $T$-modules} & \textbf{Endpt.} & \textbf{Mult.} & \textbf{Dim.} & \textbf{Diam.} & \textbf{Shape} & $A$ \textbf{action} & \textbf{Iso. type} \\
\midrule
$W_0$ & 0 & 1 & 5 & 4 & $(1,1,1,1,1)$ & \eqref{eq:t8-I} & I \\
$W_{1,a}, W_{1,b}$ & 1 & 2 & 3 & 2 & $(1,1,1)$ & \eqref{eq:t8-II} & II \\
$W_{2,a}, W_{2,b}, W_{2,c}$ & 2 & 3 & 3 & 2 & $(1,1,1)$ & \eqref{eq:t8-II} & III \\
$W_{3,a}, W_{3,b}, W_{3,c}$ & 3 & 3 & 1 & 0 & $(1)$ & \eqref{eq:t8-IV} & IV \\
$W_{3,d}, W_{3,e}, W_{3,f}$ & 3 & 3 & 2 & 1 & $(1,1)$ & \eqref{eq:t8-V} & V \\
$W_{4}$ & 4 & 1 & 1 & 0 & $(1)$ & \eqref{eq:t8-IV} & VI \\
\bottomrule
\end{tabular}

\smallskip
\captionof{table}{\textit{Irreducible $T$-modules of Tutte's 8-cage} \\
$(\dim T = 5^2+3^2+3^2+1^2+2^2+1^2=49)$}
\end{center}

For each irreducible $T$-module in the table above, we now give a pure basis and the matrix representing $A$ on that basis.

\subsection{Endpoint 0}

We now describe the primary irreducible $T$-module $W_0$. The module $W_0$ has a basis $\{e_i^*\}_{i=0}^4$, where $e_i^*$ is from Definition \ref{def:primary}. With respect to this basis the matrix representing $A$ is
\begin{equation}
    A: \begin{bmatrix}
0 & 3 & 0 & 0 & 0 \\
1 & 0 & 2 & 0 & 0 \\
0 & 1 & 0 & 2 & 0 \\
0 & 0 & 1 & 0 & 2 \\
0 & 0 & 0 & 3 & 0 \\
\end{bmatrix}.
\label{eq:t8-I}
\end{equation}
This matrix has eigenvalues $3, 2, 0, -2, -3$.

\subsection{Endpoint 1}

We now describe the Type II irreducible $T$-modules in our decomposition. For Type II, the multiplicity is 2 and the modules are $W_{1,a}$ and $W_{1,b}$. For each module, our basis has the form $\{\nu, A\nu, E_3^*A^2\nu\}$, where the seed vector $\nu$ is given below.
\begin{center}
    \begin{tabular}{c | c}
    \toprule
       \textbf{Module}  & \textbf {Essential part of }$\nu$ \\
       \midrule
       $W_{1,a}$  & $\rowvector{1 & -1 & 0}$ \\[1.67pt]
       $W_{1,b}$  & $\rowvector{1 & 1 & -2}$ \\
    \bottomrule
    \end{tabular}
    \captionof{table}{\textit{The seed vector $\nu$ for each irreducible $T$-module of Type II. Note that $A\nu=E_2^*A\nu$.}}
\end{center}

With respect to each basis, the matrix representing $A$ is 
\begin{equation}
    A: \begin{bmatrix}
  0 & 2 & 0 \\
  1 & 0 & 2 \\
  0 & 1 & 0 \\
\end{bmatrix}.
\label{eq:t8-II}
\end{equation}

This matrix has eigenvalues $2, 0, -2$.

\subsection{Endpoint 2}

We now describe the Type III irreducible $T$-modules in our decomposition. For Type III, the multiplicity is 3 and the modules are $W_{2,a}, W_{2,b}, W_{2,c}$. For each module, our basis has the form $\{\nu, A\nu, E_4^*A^2\nu\}$, where the seed vector $\nu$ is given below.
\begin{center}
    \begin{tabular}{c | c}
    \toprule
       \textbf{Module}  & \textbf {Essential part of }$\nu$ \\
       \midrule
       $W_{2,a}$  & $\rowvector{1 & -1 & 0 & 0 & 0 & 0}$ \\[1.67pt]
       $W_{2,b}$  & $\rowvector{0 & 0 & 1 & -1 & 0 & 0}$ \\[1.67pt]
       $W_{2,c}$  & $\rowvector{0 & 0 & 0 & 0 & 1 & -1}$ \\
    \bottomrule
   \end{tabular}
    \captionof{table}{\textit{The seed vector $\nu$ for each irreducible $T$-module of Type III. Note that $A\nu = E_3^*A\nu$.}}
\end{center}

With respect to each basis, the matrix representing $A$ is identical to \eqref{eq:t8-II}.

\subsection{Endpoint 3}

We now describe the Type IV irreducible $T$-modules in our decomposition. For Type IV, the multiplicity is 3 and the modules are $W_{3,a} ,W_{3,b}, W_{3,c}$. For each module, our basis has the form $\{\nu\}$, where the seed vector $\nu$ is given below.
\begin{center}
    \begin{tabular}{c | c}
    \toprule
       \textbf{Module}  & \textbf {Essential part of }$\nu$ \\
       \midrule
       $W_{3,a}$  & $\rowvector{2 & -2&0&0&-1&1&-1&1&-1&1&-1&1}$ \\[1.67pt]
       $W_{3,b}$  & $\rowvector{0&0&2&-2&-1&1&-1&1&1&-1&1&-1}$ \\[1.67pt]
       $W_{3,c}$  & $\rowvector{0&0&0&0&-1&1&1&-1&1&-1&-1&1}$ \\
    \bottomrule
    \end{tabular}
    \captionof{table}{\textit{The seed vector $\nu$ for each irreducible $T$-module of Type IV.}}
\end{center}

With respect to each basis, the matrix representing $A$ is
\begin{equation}
    A: \begin{bmatrix}
  0
\end{bmatrix}. \label{eq:t8-IV}
\end{equation}

This matrix has eigenvalue $0$.

We now describe the Type V irreducible $T$-modules in our decomposition. For Type V, the multiplicity is 3 and the modules are $W_{3,d},W_{3,e}, W_{3,f}$. For each module, our basis has the form $\{\nu, A\nu\}$, where the seed vector $\nu$ is given below.
\begin{center}
    \begin{tabular}{c | c}
    \toprule
       \textbf{Module}  & \textbf {Essential part of }$\nu$ \\
       \midrule
       $W_{3,d}$  & $\rowvector{-1&1&-1&1&-1&1&-1&1&0&0&0&0}$ \\[1.67pt]
       $W_{3,e}$  & $\rowvector{-1&1&1&-1&-1&1&1&-1&-2&2&0&0}$ \\[1.67pt]
       $W_{3,f}$  & $\rowvector{-1&1&1&-1&1&-1&-1&1&0&0&-2&2}$ \\
    \bottomrule
    \end{tabular}
    \captionof{table}{\textit{The seed vector $\nu$ for each irreducible $T$-module of Type V. Note that $A\nu = E_4^*A\nu$.}}
\end{center}

With respect to each basis, the matrix representing $A$ is
\begin{equation}
    A: \begin{bmatrix}
  0 & 4 \\
  1 & 0
\end{bmatrix}.
\label{eq:t8-V}
\end{equation}

This matrix has eigenvalues $2, -2$.

\subsection{Endpoint 4}

We now describe the Type VI irreducible $T$-modules in our decomposition. For Type VI, the multiplicity is 1 and the module is  $W_4$. For this module, our basis has the form $\{\nu\}$, where the seed vector $\nu$ is given below.
\begin{center}
    \begin{tabular}{c | c}
    \toprule
       \textbf{Module}  & \textbf {Essential part of }$\nu$ \\
       \midrule
       $W_{4}$  & $\rowvector{1&-1&-1&1&1&-1&-1&1}$ \\
    \bottomrule
    \end{tabular}
    \captionof{table}{\textit{The seed vector $\nu$ for the irreducible $T$-module of Type VI.}}
\end{center}

With respect to this basis, the matrix representing $A$ is identical to \eqref{eq:t8-IV}.


\renewcommand{\GraphNumber}{9}
\renewcommand{\GraphTitle}{Dodecahedron}
\renewcommand{\VertexCount}{20}
\renewcommand{\IntersectionArray}{\{3,2,1,1,1;1,1,1,2,3\}}
\renewcommand{\DateString}{2026.04.12 (orig. 2025.11.20)}

\section{Dodecahedron}
Throughout this section, we take $\Gamma$ to be the Dodecahedron. $\Gamma$ has 20 vertices and intersection array $\{3,2,1,1,1;1,1,1,2,3\}$.
\begin{center}
    \framebox[0.77\linewidth]{\includegraphics[width=0.8\linewidth]{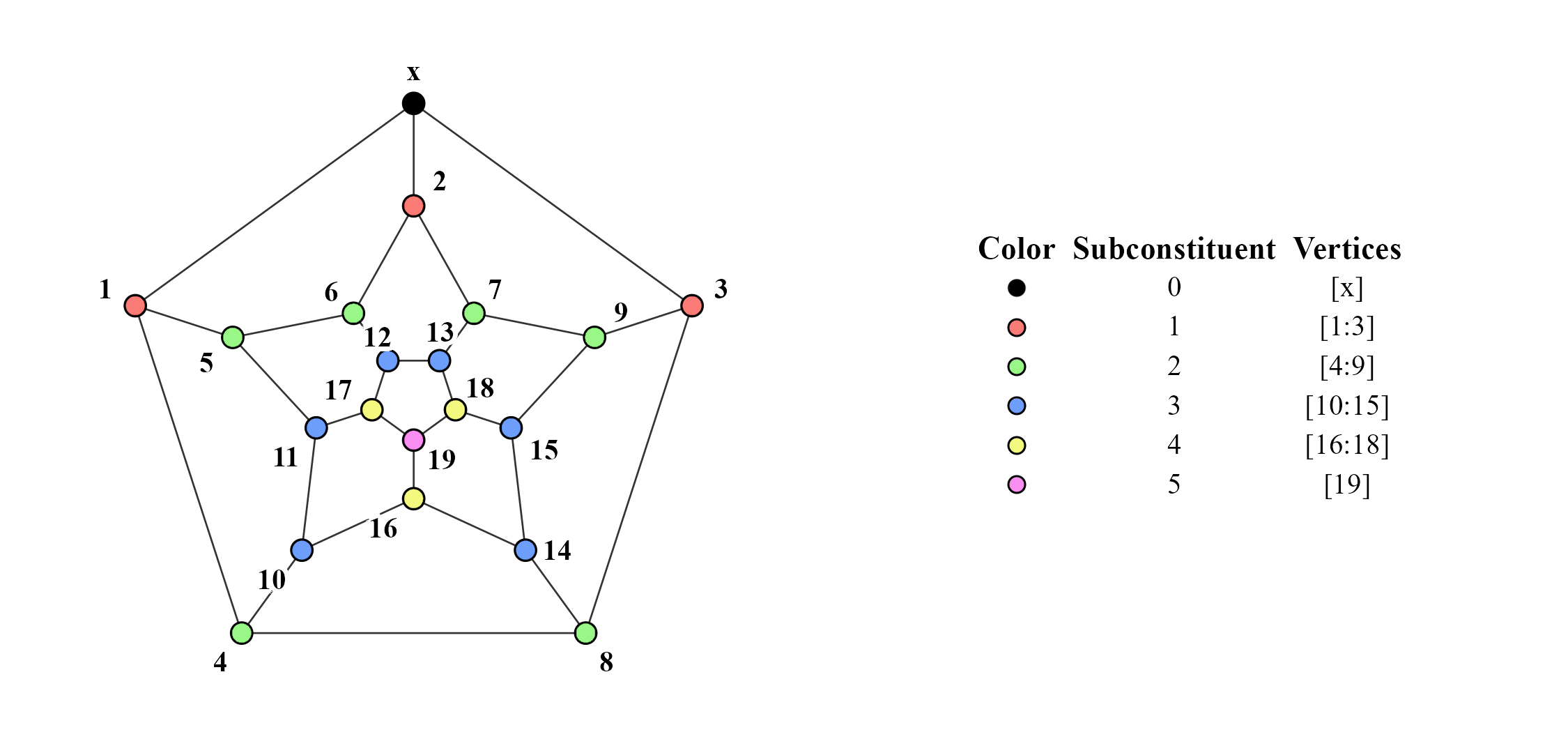}}
    \captionof{figure}{\textit{The Dodecahedron.
}}
    \end{center}

The spectrum of the adjacency matrix $A$ is $ 3^1(\sqrt5)^31^50^4(-2)^4(-\sqrt5)^3 $.

We decompose the standard module $V$ of $\Gamma$ into an orthogonal direct sum of irreducible $T$-modules. Our decomposition has the form
\[ V = W_0 \oplus W_{1,a} \oplus W_{1,b} \oplus W_2. \]
For each irreducible $T$-module in this decomposition, we give the endpoint, the multiplicity, the dimension, the diameter, the shape, the action of $A$ upon an appropriate pure basis, and the isomorphism type.



\begin{center}
\resizebox{\linewidth}{!}{%
\begin{tabular}{c | c c c c c c c}
\toprule
\textbf{Irred. $T$-modules} & \textbf{Endpt.} & \textbf{Mult.} & \textbf{Dim.} & \textbf{Diam.} & \textbf{Shape} & $A$ \textbf{action} & \textbf{Iso. type} \\
\midrule
$W_0$ & 0 & 1 & 6 & 5 & $(1,1,1,1,1,1)$ & \eqref{eq:dod-I} & I \\
$W_{1,a}, W_{1,b}$ & 1 & 2 & 6 & 3 & $(1,2,2,1)$ & \eqref{eq:dod-II} & II \\
$W_{2}$ & 2 & 1 & 2 & 1 & $(1,1)$ & \eqref{eq:dod-III} & III \\
\bottomrule
\end{tabular}
}
\smallskip
\captionof{table}{\textit{Irreducible $T$-modules for the Dodecahedron} \\ 
$(\dim T = 6^2+6^2+2^2=76)$}
\end{center}

For each irreducible $T$-module in the table above, we now give a pure basis and the matrix representing $A$ on that basis. 

\subsection{Endpoint 0}

We now describe the primary irreducible $T$-module $W_0$. The module $W_0$ has a basis $\{e_i^*\}_{i=0}^5$, where $e_i^*$ is from Definition \ref{def:primary}. With respect to this basis the matrix representing $A$ is
\begin{equation}
    A: \begin{bmatrix}
0 & 3 & 0 & 0 & 0 & 0 \\
1 & 0 & 2 & 0 & 0 & 0 \\
0 & 1 & 1 & 1 & 0 & 0 \\
0 & 0 & 1 & 1 & 1 & 0 \\
0 & 0 & 0 & 2 & 0 & 1 \\
0 & 0 & 0 & 0 & 3 & 0 \\
\end{bmatrix}.\label{eq:dod-I}
\end{equation}

This matrix has eigenvalues $3, \sqrt{5}, 1, 0, -2, -\sqrt{5}$.

\subsection{Endpoint 1}

We now describe the Type II irreducible $T$-modules in our decomposition. For Type II, the multiplicity is 2 and the modules are $W_{1,a}$ and $W_{1,b}$. For each module, our basis has the form $\{\nu, A\nu, E_2^*A^2\nu, E_3^*A^2\nu, E_3^*A^3\nu, E_4^*A^3\nu\}$, where the seed vector $\nu$ is given below.
\begin{center}
    \begin{tabular}{c | c}
    \toprule
       \textbf{Module}  & \textbf {Essential part of }$\nu$ \\
       \midrule
       $W_{1,a}$  & $\rowvector{1 & -1 & 0}$ \\[1.67pt]
       $W_{1,b}$  & $\rowvector{1 & 1 & -2}$ \\
    \bottomrule
    \end{tabular}
    \captionof{table}{\textit{The seed vector $\nu$ for each irreducible $T$-module of Type II. Note that $A\nu=E_2^*A\nu$.}}
\end{center}

With respect to each basis, the matrix representing $A$ is 
\begin{equation}
    A: \begin{bmatrix}
  0 & 2 & -1 & 0 & 0 & 0 \\
  1 & 0 & 1 & 1 & 1 & 0 \\
  0 & 1 & 0 & 0 & 1 & 0 \\
  0 & 1 & -1 & 1 & 1 & 0 \\
  0 & 0 & 1 & 0 & -1 & 1 \\
  0 & 0 & 0 & 1 & 2 & 0 \\
\end{bmatrix}.
\label{eq:dod-II}
\end{equation}

This matrix has eigenvalues $\sqrt{5}, 1, 1, 0, -2, -\sqrt{5}$.

\subsection{Endpoint 2}

We now describe the Type III irreducible $T$-modules in our decomposition. For Type III, the multiplicity is 1 and the module is  $W_2$. For this module, our basis has the form $\{\nu, E_3^*A\nu\}$, where the seed vector $\nu$ is given below.
\begin{center}
    \begin{tabular}{c | c}
    \toprule
       \textbf{Module}  & \textbf {Essential part of }$\nu$ \\
       \midrule
       $W_{2}$  & $\rowvector{1 &-1&1&-1&-1&1}$ \\
    \bottomrule
    \end{tabular}
    \captionof{table}{\textit{The seed vector $\nu$ for the irreducible $T$-module of Type III.}}
\end{center}

With respect to this basis, the matrix representing $A$ is
\begin{equation}
    A: \begin{bmatrix}
  -1 & 1 \\
  1 & -1
\end{bmatrix}.
\label{eq:dod-III}
\end{equation}

This matrix has eigenvalues $0, -2$.


\renewcommand{\GraphNumber}{10}
\renewcommand{\GraphTitle}{Desargues Graph}
\renewcommand{\VertexCount}{20}
\renewcommand{\IntersectionArray}{\{3,2,2,1,1;1,1,2,2,3\}}
\renewcommand{\DateString}{2026.04.12 (orig. 2025.12.09)}

\section{Desargues Graph}
Throughout this section, we take $\Gamma$ to be the Desargues Graph. $\Gamma$ has 20 vertices and intersection array $\{3,2,2,1,1;1,1,2,2,3\}$.
\begin{center}
    \framebox[0.77\linewidth]{\includegraphics[width=0.8\linewidth]{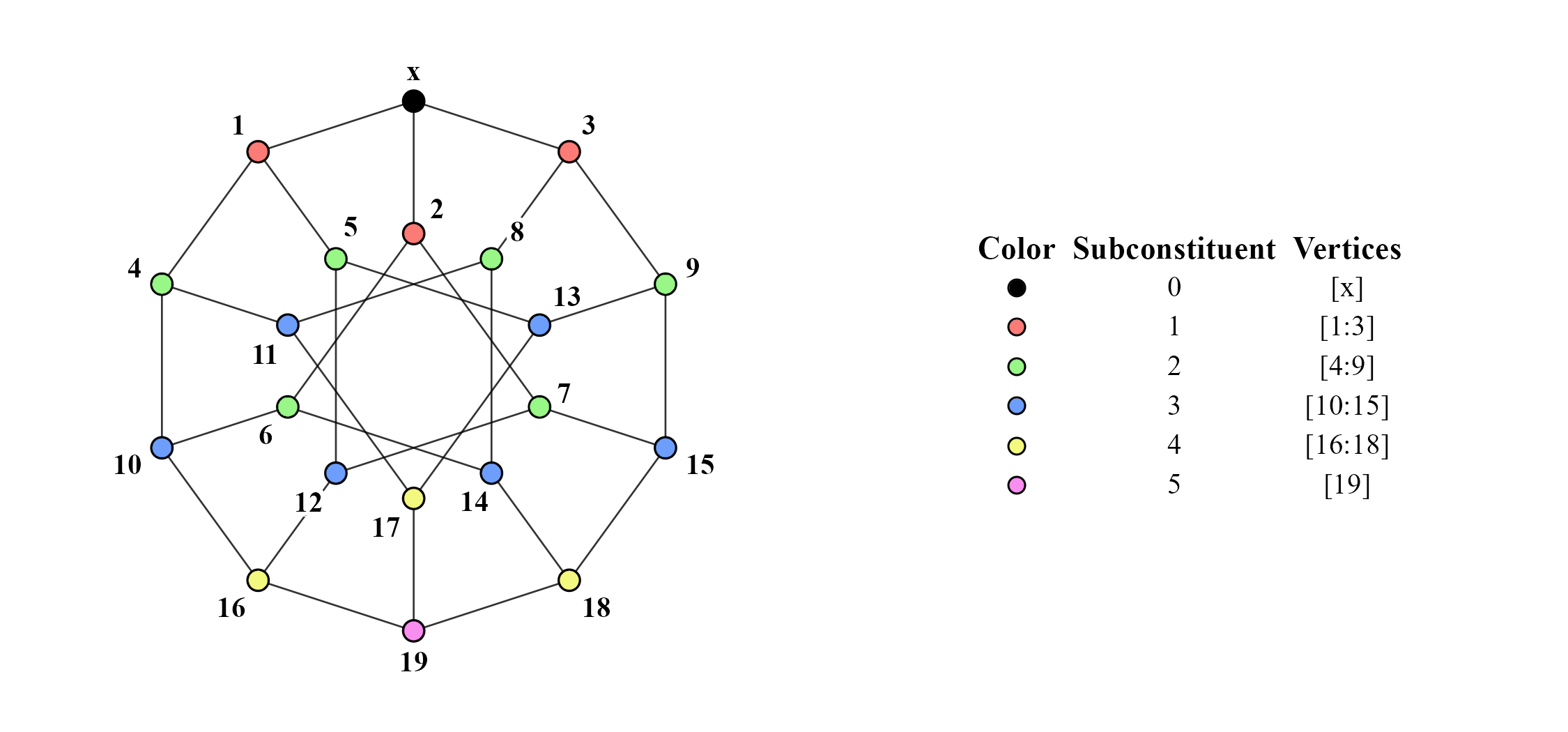}}
    \captionof{figure}{\textit{The Desargues Graph.
}}
    \end{center}

    The spectrum of the adjacency matrix $A$ is $3^12^41^5(-1)^5(-2)^4(-3)^1 $.

We decompose the standard module $V$ of $\Gamma$ into an orthogonal direct sum of irreducible $T$-modules. Our decomposition has the form
\[ V = W_0 \oplus W_{1,a} \oplus W_{1,b} \oplus W_{2,a} \oplus W_{2,b} \oplus W_{2,c}. \]
For each irreducible $T$-module in this decomposition, we give the endpoint, the multiplicity, the dimension, the diameter, the shape, the action of $A$ upon an appropriate pure basis, and the isomorphism type.



\begin{center}
\resizebox{\linewidth}{!}{%
\begin{tabular}{c | c c c c c c c}
\toprule
\textbf{Irred. $T$-modules} & \textbf{Endpt.} & \textbf{Mult.} & \textbf{Dim.} & \textbf{Diam.} & \textbf{Shape} & $A$ \textbf{action} & \textbf{Iso. type} \\
\midrule
$W_0$ & 0 & 1 & 6 & 5 & $(1,1,1,1,1,1)$ & \eqref{eq:des-I} & I \\
$W_{1,a}, W_{1,b}$ & 1 & 2 & 4 & 3 & $(1,1,1,1)$ & \eqref{eq:des-II} & II \\
$W_{2,a}$ & 2 & 1 & 2 & 1 & $(1,1)$ & \eqref{eq:des-III} & III \\
$W_{2,b}, W_{2,c}$ & 2 & 2 & 2 & 1 & $(1,1)$ & \eqref{eq:des-IV} & IV \\
\bottomrule
\end{tabular}
}
\smallskip
\captionof{table}{\textit{Irreducible $T$-modules for the Desargues Graph} \\
$(\dim T = 6^2+4^2+2^2+2^2=60)$}
\end{center}

For each irreducible $T$-module in the table above, we now give a pure basis and the matrix representing $A$ on that basis. 

\subsection{Endpoint 0}

We now describe the primary irreducible $T$-module $W_0$. The module $W_0$ has a basis $\{e_i^*\}_{i=0}^5$, where $e_i^*$ is from Definition \ref{def:primary}. With respect to this basis the matrix representing $A$ is
\begin{equation}
    A: \begin{bmatrix}
0 & 3 & 0 & 0 & 0 & 0 \\
1 & 0 & 2 & 0 & 0 & 0 \\
0 & 1 & 0 & 2 & 0 & 0 \\
0 & 0 & 2 & 0 & 1 & 0 \\
0 & 0 & 0 & 2 & 0 & 1 \\
0 & 0 & 0 & 0 & 3 & 0 \\
\end{bmatrix}.
\label{eq:des-I}
\end{equation}
This matrix has eigenvalues $3, 2, 1, -1, -2, -3$.

\subsection{Endpoint 1}

We now describe the Type II irreducible $T$-modules in our decomposition. For Type II, the multiplicity is 2 and the modules are $W_{1,a}$ and $W_{1,b}$. For each module, our basis has the form $\{\nu, A\nu, E_3^*A^2\nu, E_4^*A^3\nu\}$, where the seed vector $\nu$ is given below.
\begin{center}
    \begin{tabular}{c | c}
    \toprule
       \textbf{Module}  & \textbf {Essential part of }$\nu$ \\
       \midrule
       $W_{1,a}$  & $\rowvector{1 & -1 & 0}$ \\[1.67pt]
       $W_{1,b}$  & $\rowvector{1 & 1 & -2}$ \\
    \bottomrule
    \end{tabular}
    \captionof{table}{\textit{The seed vector $\nu$ for each irreducible $T$-module of Type II. Note that $A\nu=E_2^*A\nu$.}}
\end{center}

With respect to each basis, the matrix representing $A$ is 
\begin{equation}
    A: \begin{bmatrix}
  0 & 2 & 0 & 0 \\
  1 & 0 & 1 & 0 \\
  0 & 1 & 0 & 2 \\
  0 & 0 & 1 & 0 \\
\end{bmatrix}.
\label{eq:des-II}
\end{equation}

This matrix has eigenvalues $2, 1, -1, -2$.

\subsection{Endpoint 2}

We now describe the Type III irreducible $T$-modules in our decomposition. For Type III, the multiplicity is 1 and the module is  $W_{2,a}$. For this module, our basis has the form $\{\nu, A\nu\}$, where the seed vector $\nu$ is given below.
\begin{center}
    \begin{tabular}{c | c}
    \toprule
       \textbf{Module}  & \textbf {Essential part of }$\nu$ \\
       \midrule
       $W_{2,a}$  & $\rowvector{1 & -1 & 1 & -1 & 1 & -1}$ \\
    \bottomrule
    \end{tabular}
    \captionof{table}{\textit{The seed vector $\nu$ for the irreducible $T$-module of Type III. Note that $A\nu=E_3^*A\nu$.}}
\end{center}

With respect to this basis, the matrix representing $A$ is
\begin{equation}
    A: \begin{bmatrix}
  0 & 4 \\
  1 & 0 \\
\end{bmatrix}.
\label{eq:des-III}
\end{equation}

This matrix has eigenvalues $2, -2$.

We now describe the Type IV irreducible $T$-modules in our decomposition. For Type IV, the multiplicity is 2 and the modules are $W_{2,b}$ and $W_{2,c}$. For each module, our basis has the form $\{\nu, A\nu\}$, where the seed vector $\nu$ is given below.
\begin{center}
    \begin{tabular}{c | c}
    \toprule
       \textbf{Module}  & \textbf {Essential part of }$\nu$ \\
       \midrule
       $W_{2,b}$  & $\rowvector{2 & -2 & -1 & 1 & -1 & 1}$ \\[1.67pt]
       $W_{2,c}$  & $\rowvector{0 & 0 & 1 & -1 & -1 & 1}$ \\
    \bottomrule
    \end{tabular}
    \captionof{table}{\textit{The seed vector $\nu$ for each irreducible $T$-module of Type IV. Note that $A\nu = E_3^*A\nu$.}}
\end{center}

With respect to each basis, the matrix representing $A$ is 
\begin{equation}
    A: \begin{bmatrix}
  0 & 1 \\
  1 & 0 \\
\end{bmatrix}.
\label{eq:des-IV}
\end{equation}

This matrix has eigenvalues $1, -1$.


\renewcommand{\GraphNumber}{11}
\renewcommand{\GraphTitle}{Tutte's 12-cage}
\renewcommand{\VertexCount}{126}
\renewcommand{\IntersectionArray}{\{3,2,2,2,2,2;1,1,1,1,1,3\}}
\renewcommand{\DateString}{2026.04.12 (orig. 2026.01)}

\section{\texorpdfstring{Tutte's 12-cage, $x \in X^+$}{Tutte's 12-cage, x in X+}}
Throughout this section, we take $\Gamma$ to be Tutte's 12-cage. $\Gamma$ has 126 vertices and intersection array $\{3,2,2,2,2,2;1,1,1,1,1,3\}$.

Recall that $\Gamma$ is bipartite, with bipartition $X=X^+\cup X^-$. By \cite[Theorem~1.1]{iofinova1994}
the sets $X^+$ and $X^-$ are the orbits for the automorphism group of $\Gamma$.
In this section, we fix a base vertex $x \in X^+$.

\begin{center}
    \framebox[0.97\linewidth]{\includegraphics[width=\linewidth]{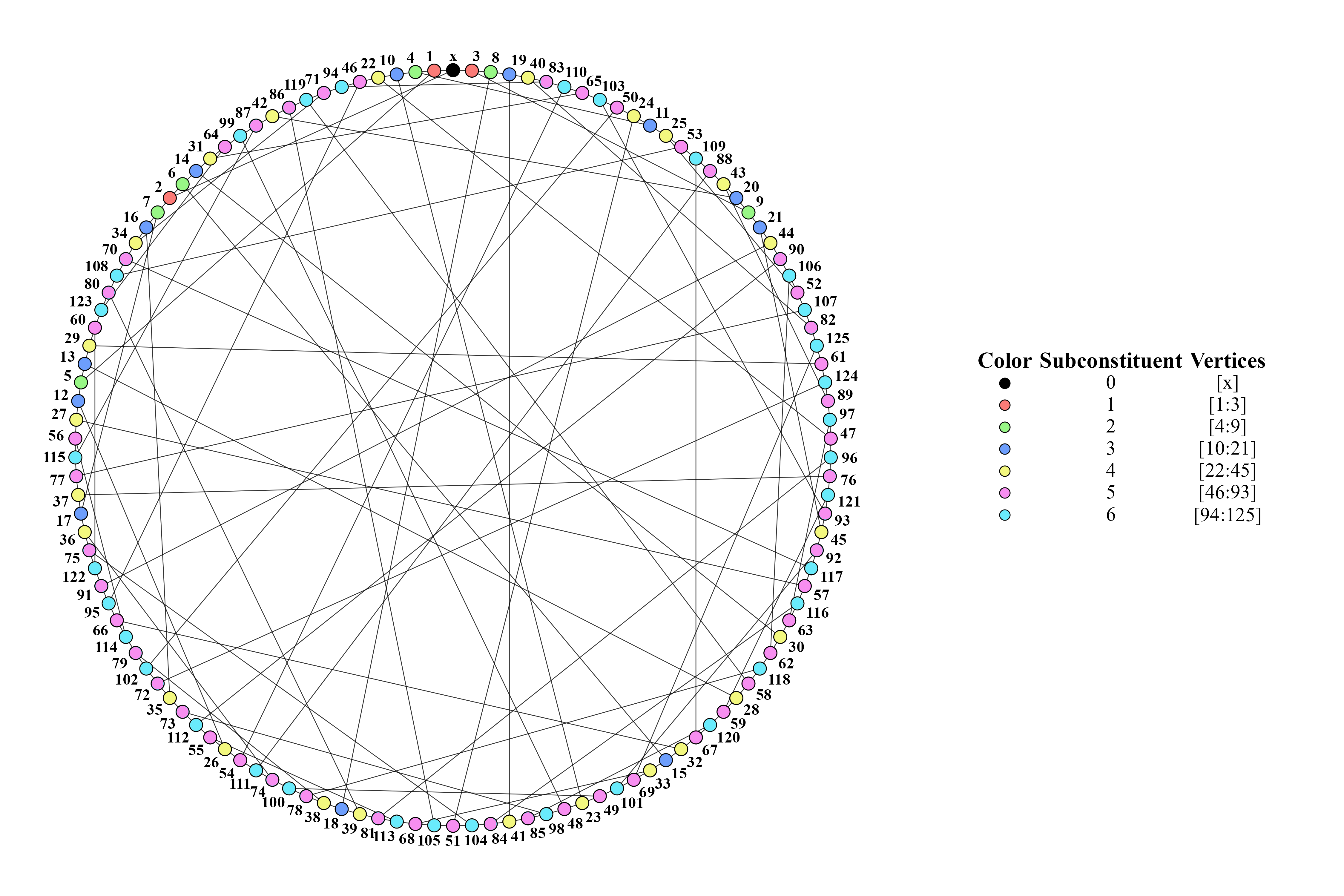}}
    \captionof{figure}{\textit{Tutte's 12-cage, $x \in X^+$.}\label{fig:12cage_B1}}
\end{center}

The spectrum of the adjacency matrix $A$ is 
\begin{equation}\label{eq:12cage-spec}
    3^1(\sqrt6)^{21}(\sqrt2)^{27}0^{28}(-\sqrt2)^{27}(-\sqrt6)^{21}(-3)^1.
\end{equation}
We decompose the standard module $V$ of $\Gamma$ into an orthogonal direct sum of irreducible $T$-modules. Our decomposition has the form

\[ 
\begin{aligned}
V = {} & W_0 \oplus W_{1,a} \oplus W_{1,b} \oplus W_{2,a} \oplus W_{2,b} \oplus W_{2,c} \oplus W_{3,a} \oplus \cdots \\
    & \cdots \oplus W_{3,f} \oplus W_{4,a} \oplus \cdots \oplus W_{4,l} \oplus W_{5,a} \oplus \cdots \oplus W_{5,l} \oplus W_6. 
\end{aligned}
\]


For each irreducible $T$-module in this decomposition, we give the endpoint, the multiplicity, the dimension, the diameter, the shape, the action of $A$ upon an appropriate pure basis, and the isomorphism type.





\begin{center}
\resizebox{\linewidth}{!}{%
\begin{tabular}{c | c c c c c c c}
\toprule
\textbf{Irred. $T$-modules} & \textbf{Endpt.} & \textbf{Mult.} & \textbf{Dim.} & \textbf{Diam.} & \textbf{Shape} & $A$ \textbf{action} & \textbf{Iso. type} \\
\midrule
$W_0$ & 0 & 1 & 7 & 6 & $(1,1,1,1,1,1,1)$ & \eqref{eq:t12-I} & I \\
$W_{1,a}, W_{1,b}$ & 1 & 2 & 5 & 4 & $(1,1,1,1,1)$ & \eqref{eq:t12-II} & II \\
$W_{2,a}, W_{2,b}, W_{2,c}$ & 2 & 3 & 5 & 4 & $(1,1,1,1,1)$ & \eqref{eq:t12-II} & III \\
$W_{3,a}, \dots, W_{3,f}$ & 3 & 6 & 5 & 3 & $(1,1,2,1)$ & \eqref{eq:t12-IV} & IV \\
$W_{4,a}, W_{4,b}, W_{4,c}$ & 4 & 3 & 2 & 1 & $(1,1)$ & \eqref{eq:t12-V} & V \\
$W_{4,d}, W_{4,e}, W_{4,f}$ & 4 & 3 & 3 & 2 & $(1,1,1)$ & \eqref{eq:t12-VI} & VI \\
$W_{4,g}, \dots, W_{4,l}$ & 4 & 6 & 5 & 2 & $(1,2,2)$ & \eqref{eq:t12-VII} & VII \\
$W_{5,a}, \dots, W_{5,f}$ & 5 & 6 & 2 & 1 & $(1,1)$ & \eqref{eq:t12-V} & VIII \\
$W_{5,g}, \dots, W_{5,l}$ & 5 & 6 & 1 & 0 & $(1)$ & \eqref{eq:t12-IX} & IX \\
$W_{6}$ & 6 & 1 & 1 & 0 & $(1)$ & \eqref{eq:t12-IX} & X \\
\bottomrule
\end{tabular}
}
\smallskip
\captionof{table}{\textit{Irreducible $T$-modules for Tutte's 12-cage, $x\in X^+$} \\ 
$(\dim T = 7^2+5^2+5^2 +5^2+2^2+3^2+5^2+2^2+1^2+1^2=168)$}
\end{center}

For each irreducible $T$-module in the table above, we now give a pure basis and the matrix representing $A$ on that basis. 

\subsection{Endpoint 0}

We now describe the primary irreducible $T$-module $W_0$. The module $W_0$ has a basis $\{e_i^*\}_{i=0}^6$, where $e_i^*$ is from Definition \ref{def:primary}. With respect to this basis the matrix representing $A$ is
\begin{equation}
    A: \begin{bmatrix}
0 & 3 & 0 & 0 & 0 & 0 & 0 \\
1 & 0 & 2 & 0 & 0 & 0 & 0 \\
0 & 1 & 0 & 2 & 0 & 0 & 0 \\
0 & 0 & 1 & 0 & 2 & 0 & 0 \\
0 & 0 & 0 & 1 & 0 & 2 & 0 \\
0 & 0 & 0 & 0 & 1 & 0 & 2 \\
0 & 0 & 0 & 0 & 0 & 3 & 0 \\
\end{bmatrix}.
\label{eq:t12-I}
\end{equation}
This matrix has eigenvalues $3, \sqrt{6}, \sqrt{2}, 0, -\sqrt{2}, -\sqrt{6}, -3$.

\subsection{Endpoint 1}

We now describe the Type II irreducible $T$-modules in our decomposition. For Type II, the multiplicity is 2 and the modules are $W_{1,a}$ and $W_{1,b}$. For each module, our basis has the form $\{\nu, A\nu, E_3^*A^2\nu, E_4^*A^3\nu,  E_5^*A^4\nu\}$, where the seed vector $\nu$ is given below.
\begin{center}
    \begin{tabular}{c | c}
    \toprule
       \textbf{Module}  & \textbf {Essential part of }$\nu$ \\
       \midrule
       $W_{1,a}$  & $\rowvector{1 & -1 & 0}$ \\[1.67pt]
       $W_{1,b}$  & $\rowvector{1 & 1 & -2}$ \\
    \bottomrule
    \end{tabular}
    \captionof{table}{\textit{The seed vector $\nu$ for each irreducible $T$-module of Type II. Note that $A\nu=E_2^*A\nu$.}}
\end{center}

With respect to each basis, the matrix representing $A$ is 
\begin{equation}
    A: \begin{bmatrix}
 0 & 2 & 0 & 0 & 0 \\
 1 & 0 & 2 & 0 & 0 \\
 0 & 1 & 0 & 2 & 0 \\
 0 & 0 & 1 & 0 & 2 \\
 0 & 0 & 0 & 1 & 0 \\
\end{bmatrix}.
\label{eq:t12-II}
\end{equation}

This matrix has eigenvalues $\sqrt{6}, \sqrt{2}, 0, -\sqrt{2}, -\sqrt{6}$.

\subsection{Endpoint 2}

We now describe the Type III irreducible $T$-modules in our decomposition. For Type III, the multiplicity is 3 and the modules are $W_{2,a},W_{2,b},W_{2,c}$. For each module, our basis has the form $\{\nu, A\nu, E_4^*A^2\nu, E_5^*A^3\nu, E_6^*A^4\nu\}$, where the seed vector $\nu$ is given below.
\begin{center}
    \begin{tabular}{c | c}
    \toprule
       \textbf{Module}  & \textbf {Essential part of }$\nu$ \\
       \midrule
       $W_{2,a}$  & $\rowvector{1 & -1 & 0 & 0 & 0 & 0}$ \\[1.67pt]
       $W_{2,b}$  & $\rowvector{0 & 0 & 1 & -1 & 0 & 0}$ \\[1.67pt]
       $W_{2,c}$  & $\rowvector{0 & 0 & 0 & 0 & 1 & -1}$ \\
    \bottomrule
   \end{tabular}
    \captionof{table}{\textit{The seed vector $\nu$ for each irreducible $T$-module of Type III. Note that  $A\nu=E_3^*A\nu$.}}
\end{center}

With respect to each basis, the matrix representing $A$ is identical to \eqref{eq:t12-II}.

\subsection{Endpoint 3}

We now describe the Type IV irreducible $T$-modules in our decomposition. For Type IV, the multiplicity is 6 and the modules are $W_{3,a},\dots,W_{3,f}$. For each module, our basis has the form 
\begin{equation*}
    \{\nu, A\nu, E_5^*A^2\nu, E_5^*A^4\nu - 6E_5^*A^2\nu, E_6^*A^3\nu\},
\end{equation*} 
where the seed vector $\nu$ is given below.
\begin{center}
    \begin{tabular}{c | c}
    \toprule
       \textbf{Module}  & \textbf {Essential part of }$\nu$ \\
       \midrule
       $W_{3,a}$  & $\rowvector{1 & -1 & 0 & 0 & 0 & 0 & 0 & 0 & 0 & 0 & 0 & 0}$ \\[1.67pt]
       $W_{3,b}$  & $\rowvector{0 & 0 & 1 & -1 & 0 & 0 & 0 & 0 & 0 & 0 & 0 & 0}$ \\[1.67pt]
      $W_{3,c}$  & $\rowvector{0 & 0 & 0 & 0 & 1 & -1 & 0 & 0 & 0 & 0 & 0 & 0}$ \\[1.67pt]
       $W_{3,d}$  & $\rowvector{0 & 0 & 0 & 0 & 0 & 0 & 1 & -1 & 0 & 0 & 0 & 0}$ \\[1.67pt]
       $W_{3,e}$  & $\rowvector{0 & 0 & 0 & 0 & 0 & 0 & 0 & 0 & 1 & -1 & 0 & 0}$ \\[1.67pt]
     $W_{3,f}$  & $\rowvector{0 & 0 & 0 & 0 & 0 & 0 & 0 & 0 & 0 & 0 & 1 & -1}$ \\
    \bottomrule
    \end{tabular}
    \captionof{table}{\textit{The seed vector $\nu$ for each irreducible $T$-module of Type IV. Note that $A\nu=E_4^*A\nu$.}}
\end{center}

With respect to each basis, the matrix representing $A$ is 
\begin{equation}
    A: \begin{bmatrix}
  0 & 2 & 0 & 0 & 0 \\
  1 & 0 & 2 & 0 & 0 \\
  0 & 1 & 0 & 0 & 2 \\
  0 & 0 & 0 & 0 & 1 \\
  0 & 0 & 1 & 2 & 0 \\
\end{bmatrix}.
\label{eq:t12-IV}
\end{equation}

This matrix has eigenvalues $\sqrt{6}, \sqrt{2}, 0, -\sqrt{2}, -\sqrt{6}$.

\subsection{Endpoint 4}

We now describe the Type V irreducible $T$-modules in our decomposition. For Type V, the multiplicity is 3 and the modules are $W_{4,a},W_{4,b},W_{4,c}$. For each module, our basis has the form $\{\nu, A\nu\}$, where the seed vector $\nu$ is given below.
{
\setlength{\arraycolsep}{3pt}
\begin{center}
    \begin{tabular}{c | c}
    \toprule
       \textbf{Module}  & \textbf {Essential part of }$\nu$ \\
       \midrule
       $W_{4,a}$  & $\rowvector{2 & -2 & -2 & 2 & 0 & 0 & 0 & 0 & -1 & 1 & -1 & 1 & -1 & 1 & 1 & -1 & 1 & -1 & -1 & 1 & 1 & -1 & -1 & 1}$ \\[1.67pt]
       $W_{4,b}$  & $\rowvector{0 & 0 & 0 & 0 & 2 & -2 & -2 & 2 & 1 & -1 & 1 & -1 & 1 & -1 & -1 & 1 & 1 & -1 & -1 & 1 & 1 & -1 & -1 & 1}$ \\[1.67pt]
       $W_{4,c}$  & $\rowvector{0 & 0 & 0 & 0 & 0 & 0 & 0 & 0 & 1 & -1 & 1 & -1 & -1 & 1 & 1 & -1 & -1 & 1 & 1 & -1 & 1 & -1 & -1 & 1}$ \\
    \bottomrule
    \end{tabular}
    \captionof{table}{\textit{The seed vector $\nu$ for each irreducible $T$-module of Type V. Note that $A\nu=E_5^*A\nu$.}}
\end{center}
}

With respect to each basis, the matrix representing $A$ is 
\begin{equation}
    A: \begin{bmatrix}
  0 & 2 \\
  1 & 0 \end{bmatrix}.
\label{eq:t12-V}
\end{equation}

This matrix has eigenvalues $\sqrt{2}, -\sqrt{2}$.

We now describe the Type VI irreducible $T$-modules in our decomposition. For Type VI, the multiplicity is 3 and the modules are $W_{4,d},W_{4,e},W_{4,f}$. For each module, our basis has the form $\{\nu, A\nu, E_6^*A^2\nu\}$,
where the seed vector $\nu$ is given below.
\begin{center}
\setlength{\arraycolsep}{3pt}
    \begin{tabular}{c | c}
    \toprule
       \textbf{Module}  & \textbf {Essential part of }$\nu$ \\
       \midrule
       $W_{4,d}$  & $\rowvector{2 & -2 & -2 & 2 & 0 & 0 & 0 & 0 & 1 & -1 & 1 & -1 & 1 & -1 & -1 & 1 & -1 & 1 & 1 & -1 & -1 & 1 & 1 & -1}$ \\[1.67pt]
       $W_{4,e}$  & $\rowvector{0 & 0 & 0 & 0 & 0 & 0 & 0 & 0 & -1 & 1 & -1 & 1 & 1 & -1 & -1 & 1 & -1 & 1 & 1 & -1 & 1 & -1 & -1 & 1}$ \\[1.67pt]
       $W_{4,f}$  & $\rowvector{0 & 0 & 0 & 0 & 2 & -2 & -2 & 2 & -1 & 1 & -1 & 1 & -1 & 1 & 1 & -1 & -1 & 1 & 1 & -1 & -1 & 1 & 1 & -1}$ \\
    \bottomrule
    \end{tabular}
    \captionof{table}{\textit{The seed vector $\nu$ for each irreducible $T$-module of Type VI. Note that $A\nu=E_5^*A\nu$.}}
\end{center}

With respect to each basis, the matrix representing $A$ is 
\begin{equation}
    A: \begin{bmatrix}
  0 & 2 & 0 \\
  1 & 0 & 4 \\
  0 & 1 & 0 \end{bmatrix}.
\label{eq:t12-VI}
\end{equation}

This matrix has eigenvalues $\sqrt{6}, 0, -\sqrt{6}$.

We now describe the Type VII irreducible $T$-modules in our decomposition. For Type VII, the multiplicity is 6 and the modules are $W_{4,g},\dots,W_{4,l}$. For each module, our basis has the form 
\begin{equation*}
    \{\nu, A\nu, E_5^*A^3\nu - 4A\nu, E_6^*A^2\nu,  E_6^*A^4\nu - 6E_6^*A^2\nu\},
\end{equation*}
where the seed vector $\nu$ is given below.
\begin{center}
    \begin{tabular}{c | c}
    \toprule
       \textbf{Module}  & \textbf {Essential part of }$\nu$ \\
       \midrule
       $W_{4,g}$  & $\rowvector{1 & -1 & 1 & -1 & 0 & 0 & 0 & 0 & 0 & 0 & 0 & 0 & 0 & 0 & 0 & 0 & 0 & 0 & 0 & 0 & 0 & 0 & 0 & 0}$ \\[1.67pt]
       $W_{4,h}$  & $\rowvector{0 & 0 & 0 & 0 & 1 & -1 & 1 & -1 & 0 & 0 & 0 & 0 & 0 & 0 & 0 & 0 & 0 & 0 & 0 & 0 & 0 & 0 & 0 & 0}$ \\[1.67pt]
       $W_{4,i}$  & $\rowvector{0 & 0 & 0 & 0 & 0 & 0 & 0 & 0 & 1 & -1 & -1 & 1 & 0 & 0 & 0 & 0 & 0 & 0 & 0 & 0 & 0 & 0 & 0 & 0}$ \\[1.67pt]
       $W_{4,j}$  & $\rowvector{0 & 0 & 0 & 0 & 0 & 0 & 0 & 0 & 0 & 0 & 0 & 0 & 1 & -1 & 1 & -1 & 0 & 0 & 0 & 0 & 0 & 0 & 0 & 0}$ \\[1.67pt]
       $W_{4,k}$  & $\rowvector{0 & 0 & 0 & 0 & 0 & 0 & 0 & 0 & 0 & 0 & 0 & 0 & 0 & 0 & 0 & 0 & 1 & -1 & 1 & -1 & 0 & 0 & 0 & 0}$ \\[1.67pt]
       $W_{4,l}$  & $\rowvector{0 & 0 & 0 & 0 & 0 & 0 & 0 & 0 & 0 & 0 & 0 & 0 & 0 & 0 & 0 & 0 & 0 & 0 & 0 & 0 & 1 & -1 & 1 & -1}$ \\
    \bottomrule
    \end{tabular}
    \captionof{table}{\textit{The seed vector $\nu$ for each irreducible $T$-module of Type VII. Note that $A\nu=E_5^*A\nu$.}}
\end{center}

With respect to each basis, the matrix representing $A$ is 
\begin{equation}
    A: \begin{bmatrix}
  0 & 2 & 0 & 0 & 0 \\
  1 & 0 & 0 & 2 & 0 \\
  0 & 0 & 0 & 1 & 2 \\
  0 & 1 & 2 & 0 & 0 \\
  0 & 0 & 1 & 0 & 0 \\
\end{bmatrix}.
\label{eq:t12-VII}
\end{equation}

This matrix has eigenvalues $\sqrt{6}, \sqrt{2}, 0, -\sqrt{2}, -\sqrt{6}$.

\subsection{Endpoint 5}

We now describe the Type VIII irreducible $T$-modules in our decomposition. For Type VIII, the multiplicity is 6 and the modules are $W_{5,a},\dots,W_{5,f}$. For each module, our basis has the form $\{\nu, A\nu\}$, where the seed vector $\nu$ is given below.
\begin{center}
    \resizebox{\linewidth}{!}{%
    \begin{tabular}{c | c}
    \toprule
       \textbf{Module}  & \textbf {Essential part of }$\nu$ \\
       \midrule
       $W_{5,a}$ & $\smallrowvector{1 & -1 & 1 & -1 & -1 & 1 & 1 & -1 & 1 & -1 & -1 & 1 & -1 & 1 & -1 & 1 & 1 & -1 & 1 & -1 & 1 & -1 & 1 & -1 & 1 & -1 & 1 & -1 & 1 & -1 & 1 & -1 & 0 & 0 & 0 & 0 & 0 & 0 & 0 & 0 & 0 & 0 & 0 & 0 & 0 & 0 & 0 & 0}$ \\[1.67pt]
       $W_{5,b}$ & $\smallrowvector{1 & -1 & -1 & 1 & 1 & -1 & 1 & -1 & -1 & 1 & -1 & 1 & -1 & 1 & 1 & -1 & -1 & 1 & 1 & -1 & -1 & 1 & 1 & -1 & -1 & 1 & 1 & -1 & 1 & -1 & -1 & 1 & 0 & 0 & 0 & 0 & 0 & 0 & 0 & 0 & 0 & 0 & 0 & 0 & 0 & 0 & 0 & 0}$ \\[1.67pt]
       $W_{5,c}$ & $\smallrowvector{1 & -1 & 0 & 0 & 0 & 0 & -1 & 1 & -1 & 1 & 0 & 0 & 0 & 0 & -1 & 1 & 0 & 0 & -1 & 1 & 0 & 0 & 1 & -1 & 1 & -1 & 0 & 0 & 0 & 0 & -1 & 1 & 2 & -2 & 0 & 0 & 0 & 0 & 0 & 0 & -1 & 1 & -1 & 1 & -1 & 1 & -1 & 1}$ \\[1.67pt]
       $W_{5,d}$ & $\smallrowvector{-1 & 1 & 0 & 0 & 0 & 0 & 1 & -1 & 1 & -1 & 0 & 0 & 0 & 0 & 1 & -1 & -2 & 2 & -1 & 1 & 2 & -2 & 1 & -1 & 1 & -1 & -2 & 2 & 2 & -2 & -1 & 1 & 0 & 0 & 2 & -2 & 2 & -2 & -2 & 2 & 3 & -3 & -1 & 1 & -1 & 1 & -1 & 1}$ \\[1.67pt]
       $W_{5,e}$ & $\smallrowvector{1 & -1 & 3 & -3 & 3 & -3 & -1 & 1 & -1 & 1 & -3 & 3 & 3 & -3 & -1 & 1 & -1 & 1 & 1 & -1 & 1 & -1 & -1 & 1 & -1 & 1 & -1 & 1 & 1 & -1 & 1 & -1 & 0 & 0 & -2 & 2 & 4 & -4 & 2 & -2 & 0 & 0 & 4 & -4 & -2 & 2 & -2 & 2}$ \\[1.67pt]
       $W_{5,f}$ & $\smallrowvector{1 & -1 & -1 & 1 & -1 & 1 & -1 & 1 & -1 & 1 & 1 & -1 & -1 & 1 & -1 & 1 & -1 & 1 & 1 & -1 & 1 & -1 & -1 & 1 & -1 & 1 & -1 & 1 & 1 & -1 & 1 & -1 & 0 & 0 & 2 & -2 & 0 & 0 & 2 & -2 & 0 & 0 & 0 & 0 & 2 & -2 & -2 & 2}$ \\
    \bottomrule
    \end{tabular}}
    \captionof{table}{\textit{The seed vector $\nu$ for each irreducible $T$-module of Type VIII. Note that $A\nu = E_6^*A\nu$.}}
\end{center}

With respect to each basis, the matrix representing $A$ is identical to \eqref{eq:t12-V}.

We now describe the Type IX irreducible $T$-modules in our decomposition. For Type IX, the multiplicity is 6 and the modules are $W_{5,g},\dots,W_{5,l}$. For each module, our basis has the form $\{\nu\}$, where the seed vector $\nu$ is given below.
\begin{center}
    \resizebox{\linewidth}{!}{%
    \begin{tabular}{c | c}
    \toprule
       \textbf{Module}  & \textbf {Essential part of }$\nu$ \\
       \midrule
       $W_{5,g}$  & $\smallrowvector{0 & 0 & 1 & -1 & -1 & 1 & 0 & 0 & -1 & 1 & 0 & 0 & 0 & 0 & 1 & -1 & 0 & 0 & -1 & 1 & 0 & 0 & -1 & 1 & 0 & 0 & 1 & -1 & 1 & -1 & 0 & 0 & 0 & 0 & 0 & 0 & 0 & 0 & 0 & 0 & 0 & 0 & 0 & 0 & 0 & 0 & 0 & 0}$ \\[1.67pt]
       $W_{5,h}$  & $\smallrowvector{1 & -1 & 0 & 0 & 0 & 0 & 1 & -1 & 0 & 0 & 1 & -1 & 1 & -1 & 0 & 0 & -1 & 1 & 0 & 0 & -1 & 1 & 0 & 0 & 1 & -1 & 0 & 0 & 0 & 0 & 1 & -1 & 0 & 0 & 0 & 0 & 0 & 0 & 0 & 0 & 0 & 0 & 0 & 0 & 0 & 0 & 0 & 0}$ \\[1.67pt]
       $W_{5,i}$  & $\smallrowvector{0 & 0 & 1 & -1 & 1 & -1 & 0 & 0 & 0 & 0 & 1 & -1 & -1 & 1 & 0 & 0 & -1 & 1 & 0 & 0 & 1 & -1 & 0 & 0 & 0 & 0 & 1 & -1 & -1 & 1 & 0 & 0 & 2 & -2 & 0 & 0 & 0 & 0 & 0 & 0 & 1 & -1 & 1 & -1 & 1 & -1 & 1 & -1}$ \\[1.67pt]
       $W_{5,j}$  & $\smallrowvector{1 & -1 & 0 & 0 & 0 & 0 & -1 & 1 & 1 & -1 & 0 & 0 & 0 & 0 & 1 & -1 & 0 & 0 & 1 & -1 & 0 & 0 & -1 & 1 & 1 & -1 & 0 & 0 & 0 & 0 & -1 & 1 & 0 & 0 & -2 & 2 & 0 & 0 & 0 & 0 & 1 & -1 & -1 & 1 & 1 & -1 & -1 & 1}$ \\[1.67pt]
       $W_{5,k}$  & $\smallrowvector{1 & -1 & 0 & 0 & 0 & 0 & -1 & 1 & 1 & -1 & 0 & 0 & 0 & 0 & 1 & -1 & 0 & 0 & -1 & 1 & 0 & 0 & 1 & -1 & -1 & 1 & 0 & 0 & 0 & 0 & 1 & -1 & 0 & 0 & 0 & 0 & 2 & -2 & 0 & 0 & -1 & 1 & -1 & 1 & 1 & -1 & 1 & -1}$ \\[1.67pt]
       $W_{5,l}$  & $\smallrowvector{0 & 0 & 1 & -1 & 1 & -1 & 0 & 0 & 0 & 0 & 1 & -1 & -1 & 1 & 0 & 0 & 1 & -1 & 0 & 0 & -1 & 1 & 0 & 0 & 0 & 0 & -1 & 1 & 1 & -1 & 0 & 0 & 0 & 0 & 0 & 0 & 0 & 0 & 2 & -2 & 1 & -1 & -1 & 1 & -1 & 1 & 1 & -1}$ \\
    \bottomrule
    \end{tabular}}
    \captionof{table}{\textit{The seed vector $\nu$ for each irreducible $T$-module of Type IX.}}
\end{center}

With respect to each basis, the matrix representing $A$ is
\begin{equation}
  A: \begin{bmatrix}
  0
\end{bmatrix}.
\label{eq:t12-IX}
\end{equation}
This matrix has eigenvalue $0$.

\subsection{Endpoint 6}

We now describe the Type X irreducible $T$-modules in our decomposition. For Type X, the multiplicity is 1 and the module is  $W_6$. For this module, our basis has the form $\{\nu\}$, where the seed vector $\nu$ is given below.
\begin{center}
    \begin{tabular}{c | c}
    \toprule
       \textbf{Module}  & \textbf {Essential part of }$\nu$ \\
       \midrule
       $W_{6}$  & $\smallrowvector{1 & -1 & 1 & -1 & -1 & 1 & -1 & 1 & -1 & 1 & -1 & 1 & -1 & 1 & 1 & -1 & -1 & 1 & 1 & -1 & 1 & -1 & 1 & -1 & 1 & -1 & 1 & -1 & 1 & -1 & 1 & -1}$ \\
    \bottomrule
    \end{tabular}
    \captionof{table}{\textit{The seed vector $\nu$ for the irreducible $T$-module of Type X.}}
\end{center}

With respect to this basis, the matrix representing $A$ is identical to \eqref{eq:t12-IX}.

\renewcommand{\GraphNumber}{11b}
\renewcommand{\GraphTitle}{Tutte's 12-cage (Second Base Vertex)}
\renewcommand{\VertexCount}{126}
\renewcommand{\IntersectionArray}{\{3,2,2,2,2,2;1,1,1,1,1,3\}}
\renewcommand{\DateString}{2026.05.17}

\section{\texorpdfstring{Tutte's 12-cage, $x \in X^-$}{Tutte's 12-cage, x in X-}}\label{sec:12cage_b2}

Throughout this section, we take $\Gamma$ to be Tutte's 12 cage. Recall that
$\Gamma$ has 126 vertices and intersection array $\{3,2,2,2,2,2;1,1,1,1,1,3\}$. Recall that $\Gamma$ is bipartite,
with bipartition $X=X^+ \cup X^-$. In this section, we fix a base vertex $x \in X^-$. 

\begin{center}
    \framebox[0.97\linewidth]{\includegraphics[width=\linewidth]{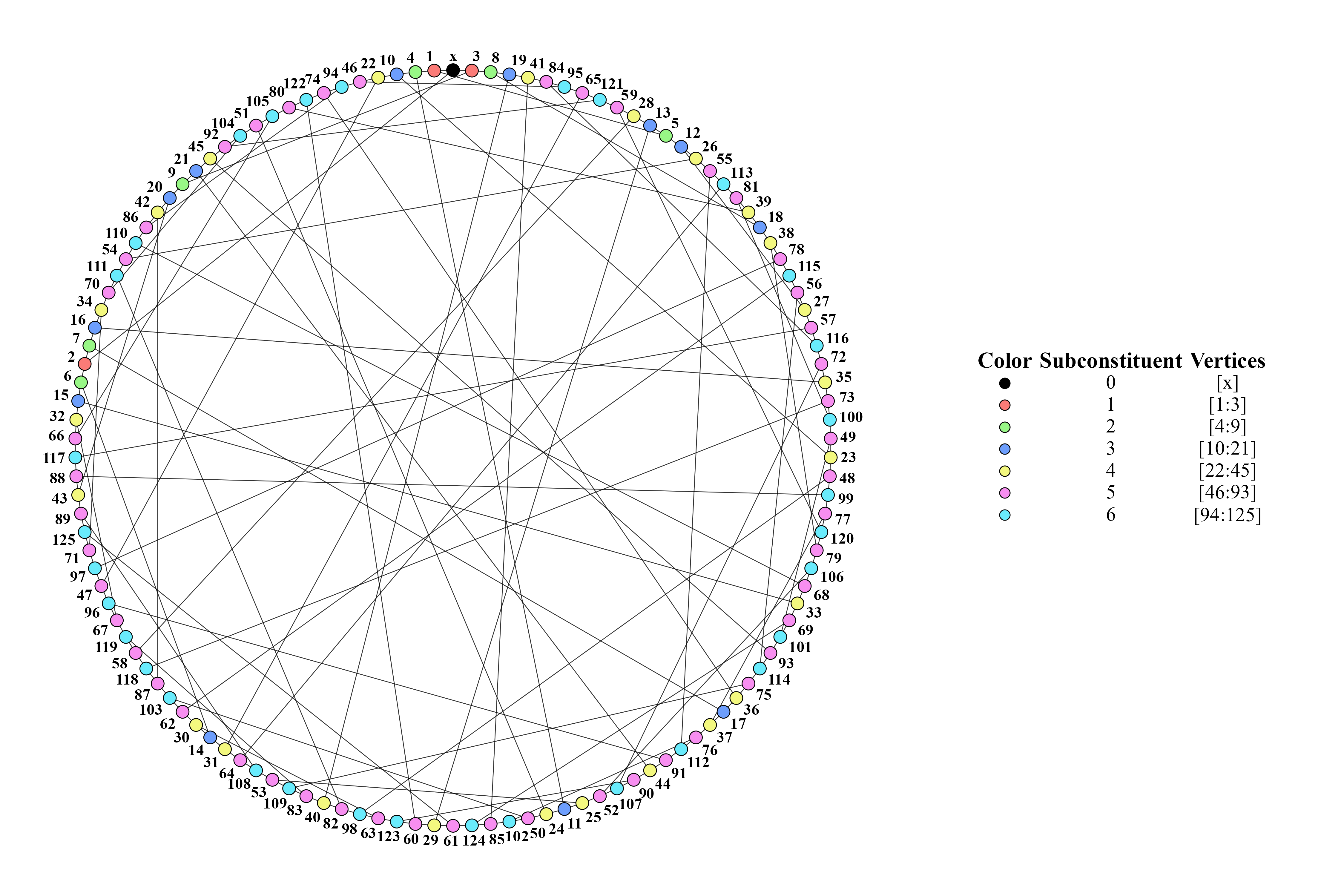}}
    \captionof{figure}{\textit{Tutte's 12-cage, $x \in X^-$.}\label{fig:12cage_B2}}
\end{center}

The spectrum of the adjacency matrix $A$ remains unchanged from \eqref{eq:12cage-spec}.

Our new decomposition has the form
\[
\begin{aligned}
V ={}& W_0 \oplus W_{1,a} \oplus W_{1,b} \oplus W_{2,a}
       \oplus W_{2,b} \oplus W_{2,c} \oplus W_{3,a} \oplus \cdots \\ & \cdots \oplus W_{3,f}
       \oplus W_{4,a} \oplus \cdots \oplus W_{4,l}
       \oplus W_{5,a} \oplus \cdots \oplus W_{5,j}
       \oplus W_{6,a} \oplus W_{6,b}.
\end{aligned}
\]

For each irreducible $T$-module in this decomposition, we give the isomorphism type, the endpoint, the dimension, the multiplicity, the diameter, the shape, and the action of $A$ upon an appropriate pure basis.
\begin{center}
\resizebox{\linewidth}{!}{%
\begin{tabular}{c | c c c c c c c}
\toprule
\textbf{Irred. $T$-modules} & \textbf{Endpt.} & \textbf{Mult.} & \textbf{Dim.} & \textbf{Diam.} & \textbf{Shape} & $A$ \textbf{action} & \textbf{Iso. type} \\
\midrule
$W_0$ & 0 & 1 & 7 & 6 & $(1,1,1,1,1,1,1)$ & \eqref{eq:t12v2-I} & I \\
$W_{1,a}, W_{1,b}$ & 1 & 2 & 5 & 4 & $(1,1,1,1,1)$ & \eqref{eq:t12v2-II} & II \\
$W_{2,a}, W_{2,b}, W_{2,c}$ & 2 & 3 & 5 & 4 & $(1,1,1,1,1)$ & \eqref{eq:t12v2-II} & III \\
$W_{3,a},\dots,W_{3,f}$ & 3 & 6 & 5 & 3 & $(1,1,2,1)$ & \eqref{eq:t12v2-IV} & IV \\
$W_{4,a},\dots,W_{4,h}$ & 4 & 8 & 5 & 2 & $(1,2,2)$ & \eqref{eq:t12v2-V} & V \\
$W_{4,i}, W_{4,j}, W_{4,k}, W_{4,l}$ & 4 & 4 & 2 & 1 & $(1,1)$ & \eqref{eq:t12v2-VI} & VI \\
$W_{5,a}, W_{5,b}, W_{5,c}$ & 5 & 3 & 2 & 1 & $(1,1)$ & \eqref{eq:t12v2-VI} & VII \\
$W_{5,d}$ & 5 & 1 & 2 & 1 & $(1,1)$ & \eqref{eq:t12v2-VIII} & VIII \\
$W_{5,e},\dots, W_{5,j}$ & 5 & 6 & 1 & 0 & $(1)$ & \eqref{eq:t12v2-IX} & IX \\
$W_{6,a}, W_{6,b}$ & 6 & 2 & 1 & 0 & $(1)$ & \eqref{eq:t12v2-IX} & X \\
\bottomrule
\end{tabular}
}

\smallskip
\captionof{table}{\textit{Irreducible $T$-modules for Tutte's 12-cage, $x \in X^-$} \\ 
$(\dim T = 7^2+5^2+5^2+5^2+5^2+2^2+2^2+2^2+1^2+1^2 = 163)$}
\end{center}

For each irreducible $T$-module in the table above, we now give a pure basis and the matrix representing $A$ on that basis.

\subsection{Endpoint 0}

We now describe the primary irreducible $T$-module $W_0$.
The module $W_0$ has a basis $\{e_i^*\}_{i=0}^6$, where $e_i^*$ is from Definition \ref{def:primary}.
With respect to this basis the matrix representing $A$ is
\begin{equation}
    A: \begin{bmatrix}
0 & 3 & 0 & 0 & 0 & 0 & 0 \\
1 & 0 & 2 & 0 & 0 & 0 & 0 \\
0 & 1 & 0 & 2 & 0 & 0 & 0 \\
0 & 0 & 1 & 0 & 2 & 0 & 0 \\
0 & 0 & 0 & 1 & 0 & 2 & 0 \\
0 & 0 & 0 & 0 & 1 & 0 & 2 \\
0 & 0 & 0 & 0 & 0 & 3 & 0 \\
\end{bmatrix}.
\label{eq:t12v2-I}
\end{equation}
This matrix has eigenvalues $3, \sqrt{6}, \sqrt{2}, 0, -\sqrt{2}, -\sqrt{6}, -3$.

\subsection{Endpoint 1}

We now describe the Type II irreducible $T$-modules in our decomposition.
For Type II, the multiplicity is 2 and the modules are $W_{1,a}$ and $W_{1,b}$.
For each module, our basis has the form $\{\nu, A\nu, E_3^*A^2\nu, E_4^*A^3\nu, E_5^*A^4\nu\}$, where the seed vector $\nu$ is given below.

\begin{center}
    \begin{tabular}{c | c}
    \toprule
       \textbf{Module}  & \textbf{Essential part of }$\nu$ \\
       \midrule
       $W_{1,a}$  & $\rowvector{1 & -1 & 0}$ \\[1.67pt]
       $W_{1,b}$  & $\rowvector{1 & 1 & -2}$ \\
    \bottomrule
    \end{tabular}
    \captionof{table}{\textit{The seed vector $\nu$ for each irreducible $T$-module of Type II. Note that $A \nu = E_2^*A\nu$.}}
\end{center}

With respect to each basis, the matrix representing $A$ is 
\begin{equation}
    A: \begin{bmatrix}
 0 & 2 & 0 & 0 & 0 \\
 1 & 0 & 2 & 0 & 0 \\
 0 & 1 & 0 & 2 & 0 \\
 0 & 0 & 1 & 0 & 2 \\
 0 & 0 & 0 & 1 & 0 \\
\end{bmatrix}.
\label{eq:t12v2-II}
\end{equation}
This matrix has eigenvalues $\sqrt{6}, \sqrt{2}, 0, -\sqrt{2}, -\sqrt{6}$.

\subsection{Endpoint 2}

We now describe the Type III irreducible $T$-modules in our decomposition.
For Type III, the multiplicity is 3 and the modules are $W_{2,a},W_{2,b},W_{2,c}$.
For each module, our basis has the form $\{\nu, A\nu, E_4^*A^2\nu, E_5^*A^3\nu, E_6^*A^4\nu\}$, where the seed vector $\nu$ is given below.

\begin{center}
    \begin{tabular}{c | c}
    \toprule
       \textbf{Module}  & \textbf{Essential part of }$\nu$ \\
       \midrule
       $W_{2,a}$  & $\rowvector{1 & -1 & 0 & 0 & 0 & 0}$ \\[1.67pt]
       $W_{2,b}$  & $\rowvector{0 & 0 & 1 & -1 & 0 & 0}$ \\[1.67pt]
       $W_{2,c}$  & $\rowvector{0 & 0 & 0 & 0 & 1 & -1}$ \\
    \bottomrule
    \end{tabular}
    \captionof{table}{\textit{The seed vector $\nu$ for each irreducible $T$-module of Type III. Note that $A \nu = E_3^*A\nu$.}}
\end{center}

With respect to each basis, the matrix representing $A$ is identical to \eqref{eq:t12v2-II}.

\subsection{Endpoint 3}

We now describe the Type IV irreducible $T$-modules in our decomposition.
For Type IV, the multiplicity is 6 and the modules are $W_{3,a},\dots,W_{3,f}$.
For each module, our basis has the form $$\{\nu, A\nu, E_5^*A^2\nu,E_5^*A^4\nu - 6E_5^*A^2\nu, E_6^*A^3\nu \},$$ where the seed vector $\nu$ is given below.

\begin{center}
    \begin{tabular}{c | c}
    \toprule
       \textbf{Module}  & \textbf{Essential part of }$\nu$ \\
       \midrule
       $W_{3,a}$  & $\rowvector{1 & -1 & 0 & 0 & 0 & 0 & 0 & 0 & 0 & 0 & 0 & 0}$ \\[1.67pt]
       $W_{3,b}$  & $\rowvector{0 & 0 & 1 & -1 & 0 & 0 & 0 & 0 & 0 & 0 & 0 & 0}$ \\[1.67pt]
       $W_{3,c}$  & $\rowvector{0 & 0 & 0 & 0 & 1 & -1 & 0 & 0 & 0 & 0 & 0 & 0}$ \\[1.67pt]
       $W_{3,d}$  & $\rowvector{0 & 0 & 0 & 0 & 0 & 0 & 1 & -1 & 0 & 0 & 0 & 0}$ \\[1.67pt]
       $W_{3,e}$  & $\rowvector{0 & 0 & 0 & 0 & 0 & 0 & 0 & 0 & 1 & -1 & 0 & 0}$ \\[1.67pt]
       $W_{3,f}$  & $\rowvector{0 & 0 & 0 & 0 & 0 & 0 & 0 & 0 & 0 & 0 & 1 & -1}$ \\
    \bottomrule
    \end{tabular}
    \captionof{table}{\textit{The seed vector $\nu$ for each irreducible $T$-module of Type IV. Note that $A \nu = E_4^*A\nu$.}}
\end{center}

With respect to each basis, the matrix representing $A$ is 
\begin{equation}
    A: \begin{bmatrix}
  0 & 2 & 0 & 0 & 0 \\
  1 & 0 & 2 & 0 & 0 \\
  0 & 1 & 0 & 0 & 2 \\
  0 & 0 & 0 & 0 & 1 \\
  0 & 0 & 1 & 2 & 0 \\
\end{bmatrix}.
\label{eq:t12v2-IV}
\end{equation}
This matrix has eigenvalues $\sqrt{6}, \sqrt{2}, 0, -\sqrt{2}, -\sqrt{6}$.

\subsection{Endpoint 4}

We now describe the Type V irreducible $T$-modules in our decomposition.
For Type V, the multiplicity is 8 and the modules are $W_{4,a}, \dots, W_{4,h}$.
For each module, our basis has the form $\{\nu, A\nu, E_5^*A^3\nu, E_6^*A^2\nu, E_6^*A^4\nu\}$, where the seed vector $\nu$ is given below.

{
\setlength{\arraycolsep}{3pt}
\begin{center}
    \begin{tabular}{c | c}
    \toprule
       \textbf{Module}  & \textbf{Essential part of }$\nu$ \\
       \midrule
       $W_{4,a}$ & $\rowvector{2 & -2 & -2 & 2 & 0 & 0 & 0 & 0 & -1 & 1 & 0 & 0 & 0 & 0 & 1 & -1 & 1 & -1 & 0 & 0 & 0 & 0 & 1 & -1}$ \\[1.67pt]
       $W_{4,b}$ & $\rowvector{0 & 0 & 0 & 0 & 2 & -2 & -2 & 2 & 0 & 0 & -1 & 1 & 1 & -1 & 0 & 0 & -1 & 1 & 0 & 0 & 0 & 0 & 1 & -1}$ \\[1.67pt]
       $W_{4,c}$ & $\rowvector{0 & 0 & 0 & 0 & 1 & -1 & 1 & -1 & 1 & -1 & 0 & 0 & 0 & 0 & 1 & -1 & 0 & 0 & 2 & -2 & 2 & -2 & 0 & 0}$ \\[1.67pt]
       $W_{4,d}$ & $\rowvector{1 & -1 & 1 & -1 & 0 & 0 & 0 & 0 & 0 & 0 & 2 & -2 & 2 & -2 & 0 & 0 & 0 & 0 & 1 & -1 & -1 & 1 & 0 & 0}$ \\[1.67pt]
       $W_{4,e}$ & $\rowvector{0 & 0 & 0 & 0 & 0 & 0 & 0 & 0 & 1 & -1 & 0 & 0 & 0 & 0 & -1 & 1 & 1 & -1 & 0 & 0 & 0 & 0 & 1 & -1}$ \\[1.67pt]
       $W_{4,f}$ & $\rowvector{0 & 0 & 0 & 0 & -1 & 1 & -1 & 1 & 1 & -1 & 0 & 0 & 0 & 0 & 1 & -1 & 0 & 0 & 0 & 0 & 0 & 0 & 0 & 0}$ \\[1.67pt]
       $W_{4,g}$ & $\rowvector{1 & -1 & 1 & -1 & 0 & 0 & 0 & 0 & 0 & 0 & 0 & 0 & 0 & 0 & 0 & 0 & 0 & 0 & -1 & 1 & 1 & -1 & 0 & 0}$ \\[1.67pt]
       $W_{4,h}$ & $\rowvector{0 & 0 & 0 & 0 & 0 & 0 & 0 & 0 & 0 & 0 & 1 & -1 & -1 & 1 & 0 & 0 & -1 & 1 & 0 & 0 & 0 & 0 & 1 & -1}$ \\
    \bottomrule
    \end{tabular}
    \captionof{table}{\textit{The seed vector $\nu$ for each irreducible $T$-module of Type V. Note that $A \nu = E_5^*A\nu$.}}
\end{center}
}

With respect to each basis, the matrix representing $A$ is 
\begin{equation}
    A: \begin{bmatrix}
 0 & 2 & 10 & 0 & 0 \\
 1 & 0 & 0 & -2 & -22 \\
 0 & 0 & 0 & 1 & 8 \\
 0 & 1 & 0 & 0 & 0 \\
 0 & 0 & 1 & 0 & 0 \\
\end{bmatrix}.
\label{eq:t12v2-V}
\end{equation}
This matrix has eigenvalues $\sqrt{6}, \sqrt{2}, 0, -\sqrt{2}, -\sqrt{6}$.

We now describe the Type VI irreducible $T$-modules in our decomposition.
For Type VI, the multiplicity is 4 and the modules are $W_{4,i}, W_{4,j}, W_{4,k}, W_{4,l}$.
For each module, our basis has the form $\{\nu, A\nu\}$, where the seed vector $\nu$ is given below.

{
\setlength{\arraycolsep}{3pt}
\begin{center}
    \begin{tabular}{c | c}
    \toprule
       \textbf{Module}  & \textbf{Essential part of }$\nu$ \\
       \midrule
       $W_{4,i}$ & $\rowvector{-1 & 1 & 1 & -1 & -1 & 1 & -1 & 1 & -2 & 2 & 0 & 0 & 0 & 0 & 0 & 0 & 1 & -1 & 1 & -1 & 1 & -1 & 1 & -1}$ \\[1.67pt]
       $W_{4,j}$ & $\rowvector{1 & -1 & 0 & 0 & -1 & 1 & 0 & 0 & 0 & 0 & -1 & 1 & 0 & 0 & -1 & 1 & -1 & 1 & 1 & -1 & 0 & 0 & 0 & 0}$ \\[1.67pt]
       $W_{4,k}$ & $\rowvector{1 & -1 & 2 & -2 & 1 & -1 & 0 & 0 & 0 & 0 & -1 & 1 & -2 & 2 & 1 & -1 & 1 & -1 & 1 & -1 & -2 & 2 & 0 & 0}$ \\[1.67pt]
       $W_{4,l}$ & $\rowvector{1 & -1 & -1 & 1 & 1 & -1 & -3 & 3 & 0 & 0 & 2 & -2 & -2 & 2 & -2 & 2 & 1 & -1 & 1 & -1 & 1 & -1 & -3 & 3}$ \\
    \bottomrule
    \end{tabular}
    \captionof{table}{\textit{The seed vector $\nu$ for each irreducible $T$-module of Type VI. Note that $A \nu = E_5^*A\nu$.}}
\end{center}
}

With respect to each basis, the matrix representing $A$ is 
\begin{equation}
    A: \begin{bmatrix}
  0 & 2 \\
  1 & 0 \end{bmatrix}.
\label{eq:t12v2-VI}
\end{equation}
This matrix has eigenvalues $\sqrt{2}, -\sqrt{2}$.

\subsection{Endpoint 5}

We now describe the Type VII irreducible $T$-modules in our decomposition.
For Type VII, the multiplicity is 3 and the modules are $W_{5,a}, W_{5,b}, W_{5,c}$.
For each module, our basis has the form $\{\nu, A\nu\}$, where the seed vector $\nu$ is given below.

\begin{center}
    \resizebox{\linewidth}{!}{%
    \begin{tabular}{c | c}
    \toprule
       \textbf{Module}  & \textbf{Essential part of }$\nu$ \\
       \midrule
       $W_{5,a}$ & $\smallrowvector{0 & 0 & 0 & 0 & 0 & 0 & 0 & 0 & 0 & 0 & 0 & 0 & 0 & 0 & 0 & 0 & 1 & -1 & 1 & -1 & -1 & 1 & -1 & 1 & 1 & -1 & 1 & -1 & 1 & -1 & -1 & 1 & -1 & 1 & 1 & -1 & 1 & -1 & 1 & -1 & -1 & 1 & -1 & 1 & -1 & 1 & -1 & 1}$ \\[1.67pt]
       $W_{5,b}$ & $\smallrowvector{1 & -1 & 1 & -1 & -1 & 1 & 1 & -1 & -1 & 1 & 1 & -1 & 1 & -1 & -1 & 1 & 0 & 0 & 0 & 0 & 0 & 0 & 0 & 0 & 0 & 0 & 0 & 0 & 0 & 0 & 0 & 0 & 0 & 0 & 0 & 0 & 0 & 0 & 0 & 0 & 0 & 0 & 0 & 0 & 0 & 0 & 0 & 0}$ \\[1.67pt]
       $W_{5,c}$ & $\smallrowvector{0 & 0 & 0 & 0 & 0 & 0 & 0 & 0 & 0 & 0 & 0 & 0 & 0 & 0 & 0 & 0 & 1 & -1 & 1 & -1 & -1 & 1 & -1 & 1 & 1 & -1 & 1 & -1 & 1 & -1 & -1 & 1 & 1 & -1 & -1 & 1 & -1 & 1 & -1 & 1 & 1 & -1 & 1 & -1 & 1 & -1 & 1 & -1}$ \\
    \bottomrule
    \end{tabular}}
    \captionof{table}{\textit{The seed vector $\nu$ for each irreducible $T$-module of Type VII. Note that $A \nu = E_6^*A\nu$.}}
\end{center}

With respect to each basis, the matrix representing $A$ is identical to \eqref{eq:t12v2-VI}.

We now describe the Type VIII irreducible $T$-modules in our decomposition.
For Type VIII, the multiplicity is 1 and the module is $W_{5,d}$.
For this module, our basis has the form $\{\nu, A\nu\}$, where the seed vector $\nu$ is given below.

\begin{center}
    \resizebox{\linewidth}{!}{%
    \begin{tabular}{c | c}
    \toprule
       \textbf{Module}  & \textbf{Essential part of }$\nu$ \\
       \midrule
       $W_{5,d}$ & $\smallrowvector{1 & -1 & 1 & -1 & -1 & 1 & 1 & -1 & 1 & -1 & -1 & 1 & -1 & 1 & 1 & -1 & -1 & 1 & -1 & 1 & 1 & -1 & 1 & -1 & 1 & -1 & 1 & -1 & 1 & -1 & -1 & 1 & -1 & 1 & 1 & -1 & 1 & -1 & 1 & -1 & 1 & -1 & 1 & -1 & 1 & -1 & 1 & -1}$ \\
    \bottomrule
    \end{tabular}}
    \captionof{table}{\textit{The seed vector $\nu$ for the irreducible $T$-module of Type VIII. Note that $A \nu = E_6^*A\nu$.}}
\end{center}

With respect to this basis, the matrix representing $A$ is 
\begin{equation}
    A: \begin{bmatrix}
  0 & 6 \\
  1 & 0 \end{bmatrix}.
\label{eq:t12v2-VIII}
\end{equation}
This matrix has eigenvalues $\sqrt{6}, -\sqrt{6}$.

We now describe the Type IX irreducible $T$-modules in our decomposition.
For Type IX, the multiplicity is 6 and the modules are $W_{5,e}, \dots, W_{5,j}$.
For each module, our basis has the form $\{\nu\}$, where the seed vector $\nu$ is given below.

\begin{center}
    \resizebox{\linewidth}{!}{%
    \begin{tabular}{c | c}
    \toprule
       \textbf{Module}  & \textbf{Essential part of }$\nu$ \\
       \midrule
       $W_{5,e}$ & $\smallrowvector{1 & -1 & 0 & 0 & 0 & 0 & 1 & -1 & 0 & 0 & -1 & 1 & -1 & 1 & 0 & 0 & 0 & 0 & 1 & -1 & -1 & 1 & 0 & 0 & 0 & 0 & -1 & 1 & -1 & 1 & 0 & 0 & 1 & -1 & 1 & -1 & 0 & 0 & 0 & 0 & 0 & 0 & 0 & 0 & 0 & 0 & 0 & 0}$ \\[1.67pt]
       $W_{5,f}$ & $\smallrowvector{1 & -1 & -1 & 1 & -1 & 1 & -1 & 1 & 1 & -1 & 1 & -1 & -1 & 1 & -1 & 1 & -1 & 1 & 1 & -1 & -1 & 1 & 1 & -1 & -1 & 1 & 1 & -1 & 1 & -1 & 1 & -1 & 0 & 0 & 0 & 0 & 0 & 0 & 0 & 0 & -2 & 2 & 2 & -2 & 0 & 0 & 0 & 0}$ \\[1.67pt]
       $W_{5,g}$ & $\smallrowvector{3&-3&-1&1&-3&3&-1&1&3&-3&1&-1&5&-5&7&-7&7&-7&3&-3&-3&3&-7&7&-11&11&1&-1&-7&7&3&-3&-8&8&0&0&8&-8&0&0&4&-4&4&-4&0&0&8&-8}$ \\[1.67pt]
       $W_{5,h}$ & $\smallrowvector{2 & -2 & -5 & 5 & 0 & 0 & -3 & 3 & -5 & 5 & -2 & 2 & 3 & -3 & 0 & 0 & 0 & 0 & -3 & 3 & -2 & 2 & 5 & -5 & 0 & 0 & 3 & -3 & -2 & 2 & -5 & 5 & 2 & -2 & 2 & -2 & 5 & -5 & -5 & 5 & 0 & 0 & 0 & 0 & 0 & 0 & 0 & 0}$ \\[1.67pt]
       $W_{5,i}$ & $\smallrowvector{9 & -9 & 23 & -23 & -35 & 35 & -3 & 3 & 9 & -9 & -23 & 23 & -11 & 11 & 21 & -21 & 21 & -21 & -17 & 17 & -9 & 9 & 5 & -5 & -7 & 7 & 3 & -3 & 5 & -5 & 9 & -9 & 2 & -2 & -26 & 26 & -2 & 2 & -26 & 26 & -14 & 14 & -14 & 14 & 0 & 0 & -28 & 28}$ \\[1.67pt]
       $W_{5,j}$ & $\smallrowvector{3 & -3 & -1 & 1 & 1 & -1 & 3 & -3 & 3 & -3 & 1 & -1 & 1 & -1 & 3 & -3 & -3 & 3 & 1 & -1 & 3 & -3 & -1 & 1 & -1 & 1 & 3 & -3 & -1 & 1 & -3 & 3 & 2 & -2 & -2 & 2 & -2 & 2 & -2 & 2 & -2 & 2 & -2 & 2 & -6 & 6 & 2 & -2}$ \\
    \bottomrule
    \end{tabular}}
    \captionof{table}{\textit{The seed vector $\nu$ for each irreducible $T$-module of Type IX.}}
\end{center}

With respect to each basis, the matrix representing $A$ is
\begin{equation}
  A: \begin{bmatrix}
  0
\end{bmatrix}.
\label{eq:t12v2-IX}
\end{equation}
This matrix has eigenvalue $0$.

\subsection{Endpoint 6}

We now describe the Type X irreducible $T$-modules in our decomposition.
For Type X, the multiplicity is 2 and the modules are $W_{6,a}$ and $W_{6,b}$.
For each module, our basis has the form $\{\nu\}$, where the seed vector $\nu$ is given below.

\begin{center}
    \begin{tabular}{c | c}
    \toprule
       \textbf{Module}  & \textbf{Essential part of }$\nu$ \\
       \midrule
       $W_{6,a}$  & $\smallrowvector{1 & -1 & 0 & 0 & -1 & 1 & 0 & 0 & 0 & 0 & -1 & 1 & 1 & -1 & 0 & 0 & -1 & 1 & 0 & 0 & 0 & 0 & 1 & -1 & 0 & 0 & -1 & 1 & -1 & 1 & 0 & 0}$ \\[1.67pt]
       $W_{6,b}$  & $\smallrowvector{0 & 0 & 1 & -1 & 0 & 0 & -1 & 1 & 1 & -1 & 0 & 0 & 0 & 0 & -1 & 1 & 0 & 0 & -1 & 1 & -1 & 1 & 0 & 0 & 1 & -1 & 0 & 0 & 0 & 0 & -1 & 1}$ \\
    \bottomrule
    \end{tabular}
    \captionof{table}{\textit{The seed vector $\nu$ for each irreducible $T$-module of Type X.}}
\end{center}

With respect to this basis, the matrix representing $A$ is identical to \eqref{eq:t12v2-IX}.


\renewcommand{\GraphNumber}{12}
\renewcommand{\GraphTitle}{Biggs-Smith Graph}
\renewcommand{\VertexCount}{102}
\renewcommand{\IntersectionArray}{\{3,2,2,2,1,1,1;1,1,1,1,1,1,3\}}
\renewcommand{\DateString}{2026.04.12 (orig. 2026.03)}

\section{Biggs-Smith Graph}
Throughout this section, we take $\Gamma$ to be the Biggs-Smith Graph. $\Gamma$ has 102 vertices and intersection array $\{3,2,2,2,1,1,1;1,1,1,1,1,1,3\}$.
\begin{center}
    \framebox[0.97\linewidth]{\includegraphics[width=0.99\linewidth]{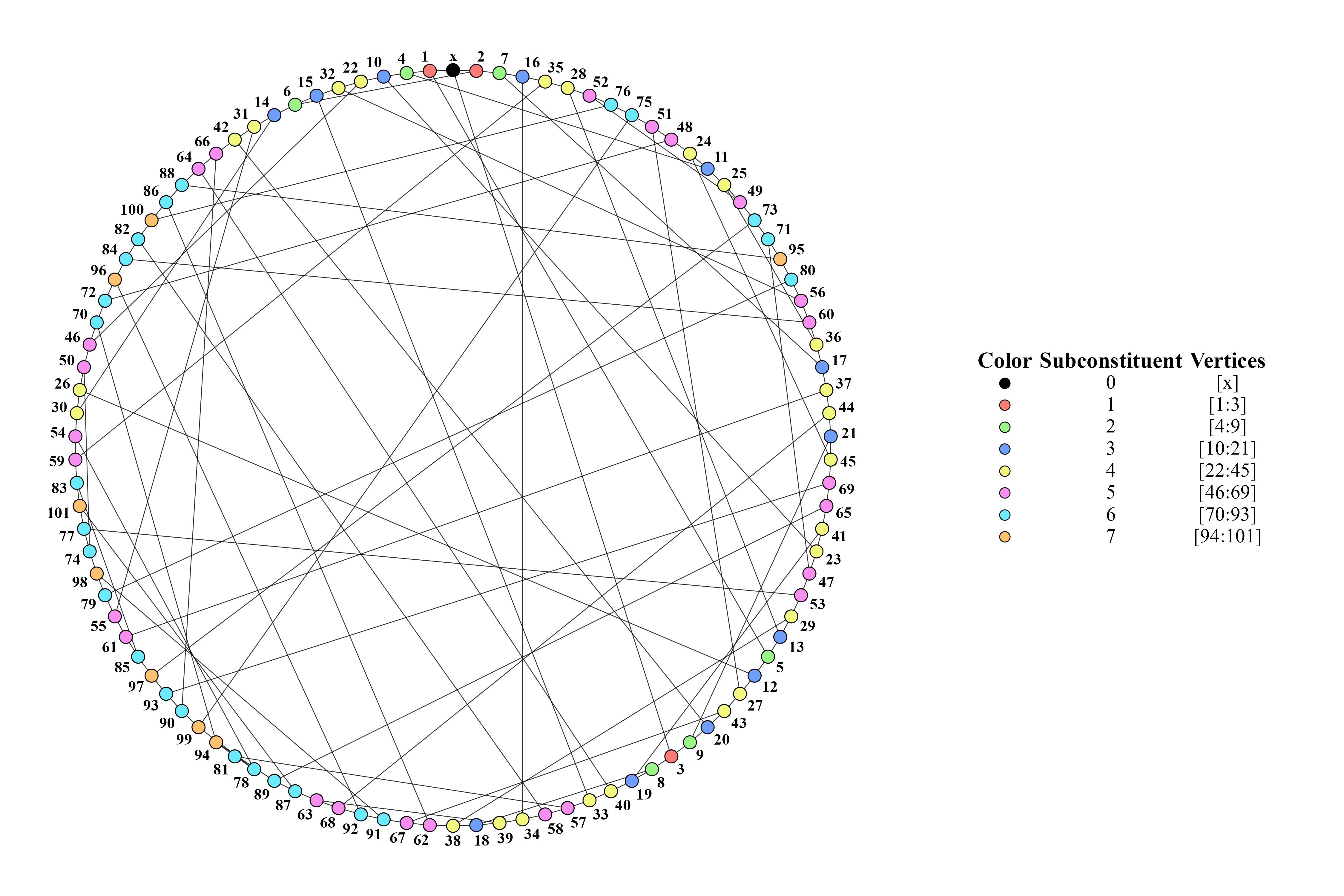}}
    \captionof{figure}{\textit{The Biggs-Smith Graph.}}
  \end{center}
    
The spectrum of the adjacency matrix $A$ is
    $$ 3^{1}\left(\frac{1+\sqrt{17}}{2}\right)^{9} 2^{18} (\theta_1)^{16} 0^{17} (\theta_2)^{16} \left(\frac{1-\sqrt{17}}{2}\right)^{9} (\theta_3)^{16},$$
    where $\theta_1, \theta_2, \theta_3$ are the roots of the polynomial $\theta^3 + 3\theta^2 - 3$.

    The exact values and approximations for $\theta_1, \theta_2,\theta_3$ are
\begin{equation}
\begin{split}
    \theta_1 &= 2\cos(\pi/9) - 1 \approx 0.879, \quad\quad \theta_2 = 2\cos(13\pi/9) - 1 \approx -1.347, \\
    \theta_3 &= 2\cos(7\pi/9) - 1 \approx -2.532.
\end{split}
\label{eq:thetas}
\end{equation}

We decompose the standard module $V$ of $\Gamma$ into an orthogonal direct sum of irreducible $T$-modules. Our decomposition has the form
\[ V = W_0 \oplus W_{1,a} \oplus W_{1,b} \oplus W_{2,a} \oplus W_{2,b} \oplus W_{2,c} \oplus W_{3,a} \oplus W_{3,b} \oplus W_{3,c} \oplus W_4. \]
For each irreducible $T$-module in this decomposition, we give the endpoint, the multiplicity, the dimension, the diameter, the shape, the action of $A$ upon an appropriate pure basis, and the isomorphism type.





\begin{center}
\resizebox{\linewidth}{!}{%
\begin{tabular}{c | c c c c c c c}
\toprule
\textbf{Irred. $T$-modules} & \textbf{Endpt.} & \textbf{Mult.} & \textbf{Dim.} & \textbf{Diam.} & \textbf{Shape} & $A$ \textbf{action} & \textbf{Iso. type} \\
\midrule
$W_0$ & 0 & 1 & 8 & 7 & $(1,1,1,1,1,1,1,1)$ & \eqref{eq:bs-I} & I \\
$W_{1,a}, W_{1,b}$ & 1 & 2 & 9 & 5 & $(1,1,1,2,2,2)$ & \eqref{eq:bs-II} & II \\
$W_{2,a}, W_{2,b}, W_{2,c}$ & 2 & 3 & 13 & 5 & $(1,2,3,3,3,1)$ & \eqref{eq:bs-III} & III \\
$W_{3,a}, W_{3,b}, W_{3,c}$ & 3 & 3 & 11 & 4 & $(1,3,3,3,1)$ & \eqref{eq:bs-IV} & IV \\
$W_{4}$ & 4 & 1 & 4 & 3 & $(1,1,1,1)$ & \eqref{eq:bs-V} & V \\
\bottomrule
\end{tabular}
}
\smallskip
\captionof{table}{\textit{Irreducible $T$-modules for the Biggs-Smith Graph} \\
$(\dim T=8^2+9^2+13^2+11^2+4^2=451)$}
\end{center}

For each irreducible $T$-module in the table above, we now give a pure basis and the matrix representing $A$ on that basis. 

\subsection{Endpoint 0}

We now describe the primary irreducible $T$-module $W_0$. The module $W_0$ has a basis $\{e_i^*\}_{i=0}^7$, where $e_i^*$ is from Definition \ref{def:primary}. With respect to this basis the matrix representing $A$ is
\begin{equation}
    A: \begin{bmatrix}
  0 & 3 & 0 & 0 & 0 & 0 & 0 & 0 \\
  1 & 0 & 2 & 0 & 0 & 0 & 0 & 0 \\
  0 & 1 & 0 & 2 & 0 & 0 & 0 & 0 \\
  0 & 0 & 1 & 0 & 2 & 0 & 0 & 0 \\
  0 & 0 & 0 & 1 & 1 & 1 & 0 & 0 \\
  0 & 0 & 0 & 0 & 1 & 1 & 1 & 0 \\
  0 & 0 & 0 & 0 & 0 & 1 & 1 & 1 \\
  0 & 0 & 0 & 0 & 0 & 0 & 3 & 0 \end{bmatrix}.
\label{eq:bs-I}
\end{equation}
This matrix has eigenvalues $3, \frac{1+\sqrt{17}}{2}, 2, \theta_1, 0, \theta_2, \frac{1-\sqrt{17}}{2}, \theta_3$, where $\theta_1, \theta_2, \theta_3 $ are defined in \eqref{eq:thetas}.

\subsection{Endpoint 1}

We now describe the Type II irreducible $T$-modules in our decomposition. For Type II, the multiplicity is 2 and the modules are $W_{1,a}$ and $W_{1,b}$. For each module, our basis has the form 
\begin{equation*}
    \{\nu, A\nu, E_3^*A^2\nu, E_4^*A^3\nu, E_4^*A^4\nu, E_5^*A^4\nu, E_5^*A^5\nu, E_6^*A^5\nu, E_6^*A^6\nu\},
\end{equation*} 
where the seed vector $\nu$ is given below.
\begin{center}
    \begin{tabular}{c | c}
    \toprule
       \textbf{Module}  & \textbf {Essential part of }$\nu$ \\
       \midrule
       $W_{1,a}$  & $\rowvector{1 & -1 & 0}$ \\[1.67pt]
       $W_{1,b}$  & $\rowvector{1 & 1 & -2}$ \\
    \bottomrule
    \end{tabular}
    \captionof{table}{\textit{The seed vector $\nu$ for each irreducible $T$-module of Type II. Note that $A\nu=E_2^*A\nu$.}}
\end{center}

With respect to each basis, the matrix representing $A$ is 
\begin{equation}
    A: \begin{bmatrix}
  0 & 2 & 0 & 0 & 0 & 0 & 0 & 0 & 0 \\
  1 & 0 & 2 & 0 & 0 & 0 & 0 & 0 & 0 \\
  0 & 1 & 0 & 2 & -1 & 0 & 0 & 0 & 0 \\
  0 & 0 & 1 & 0 & 1 & 1 & 1 & 0 & 0 \\
  0 & 0 & 0 & 1 & 0 & 0 & 1 & 0 & 0 \\
  0 & 0 & 0 & 1 & -1 & 1 & 0 & 1 & 1 \\
  0 & 0 & 0 & 0 & 1 & 0 & 1 & 0 & 1\\
  0 & 0 & 0 & 0 & 0 & 1 & -1 & 1 & 3 \\
  0 & 0 & 0 & 0 & 0 & 0 & 1 & 0 & -1 \end{bmatrix}.
\label{eq:bs-II}
\end{equation}

This matrix has eigenvalues $\frac{1+\sqrt{17}}{2}$, $2$, $2$, $\theta_1$, $0$, $0$, $\theta_2$, $\frac{1-\sqrt{17}}{2}$, $\theta_3$, where $\theta_1, \theta_2, \theta_3 $ are defined in \eqref{eq:thetas}.

\subsection{Endpoint 2}

We now describe the Type III irreducible $T$-modules in our decomposition. For Type III, the multiplicity is 3 and the modules are $W_{2,a},W_{2,b},W_{2,c}$. For each module, our basis has the form 

\vspace{-0.6cm}
\begin{align*}
    \{&\nu, A\nu, E_3^*A^4\nu, E_4^*A^2\nu, E_4^*A^3\nu,
    E_4^*A^6\nu - 36E_4^*A^2\nu, E_5^*A^3\nu, E_5^*A^4\nu + E_5^*A^3\nu, \\ 
    &E_5^*A^6\nu - 10E_5^*A^4\nu-E_5^*A^3\nu, E_6^*A^4\nu, E_6^*A^5\nu,
    E_6^*A(E_5^*A^6\nu - 10E_5^*A^4\nu-E_5^*A^3\nu), E_7^*A^5\nu\},
\end{align*} 

where the seed vector $\nu$ is given below.
\begin{center}
    \begin{tabular}{c | c}
    \toprule
       \textbf{Module}  & \textbf {Essential part of }$\nu$ \\
       \midrule
       $W_{2,a}$  & $\rowvector{1 & -1 & 0 & 0 & 0 & 0}$ \\[1.67pt]
       $W_{2,b}$  & $\rowvector{0 & 0 & 1 & -1 & 0 & 0}$ \\[1.67pt]
       $W_{2,c}$  & $\rowvector{0 & 0 & 0 & 0 & 1 & -1}$ \\
    \bottomrule
   \end{tabular}
    \captionof{table}{\textit{The seed vector $\nu$ for each irreducible $T$-module of Type III. Note that $A\nu=E_3^*A\nu$.}}
\end{center}

With respect to each basis, the matrix representing $A$ is 
\begin{equation}
    A: \begin{bmatrix}
  0 & 2 & 0 & 0 & 0 & 0 & 0 & 0 & 0 & 0 & 0 & 0 & 0 \\
  1 & 0 & 0 & 2 & 0 & 0 & 0 & 0 & 0 & 0 & 0 & 0 & 0 \\
  0 & 0 & 0 & 0 & 1 & 1 & 0 & 0 & 0 & 0 & 0 & 0 & 0 \\
  0 & 1 & 0 & 0 & 1 & 0 & 1 & 0 & 0 & 0 & 0 & 0 & 0 \\
  0 & 0 & 1 & 1 & 0 & 0 & 0 & 1 & 0 & 0 & 0 & 0 & 0 \\
  0 & 0 & 1 & 0 & 0 & 1 & 0 & 0 & 1 & 0 & 0 & 0 & 0 \\
  0 & 0 & 0 & 1 & 0 & 0 & -1 & 0 & 0 & 1 & 0 & 0 & 0 \\
  0 & 0 & 0 & 0 & 1 & 0 & 0 & 1 & 0 & 0 & 1 & 0 & 0 \\
  0 & 0 & 0 & 0 & 0 & 1 & 0 & 0 & -1 & 0 & 0 & 1 & 0 \\
  0 & 0 & 0 & 0 & 0 & 0 & 1 & 0 & 0 & 1 & 0 & 0 & 1 \\
  0 & 0 & 0 & 0 & 0 & 0 & 0 & 1 & 0 & 0 & 0 & -1 & 1 \\
  0 & 0 & 0 & 0 & 0 & 0 & 0 & 0 & 1 & 0 & -1 & 0 & 1 \\
  0 & 0 & 0 & 0 & 0 & 0 & 0 & 0 & 0 & 1 & 1 & 1 & 0 \end{bmatrix}.
\label{eq:bs-III}
\end{equation}

This matrix has eigenvalues $\frac{1+\sqrt{17}}{2}$, $2$, $2$, $2$, $\theta_1$, $\theta_1$, $0$, $0$, $\theta_2$, $\theta_2$, $\frac{1-\sqrt{17}}{2}$, $\theta_3$, $\theta_3$, where $\theta_1, \theta_2, \theta_3 $ are defined in \eqref{eq:thetas}.

\subsection{Endpoint 3}

We now describe the Type IV irreducible $T$-modules in our decomposition. For Type IV, the multiplicity is 3 and the modules are $W_{3,a},W_{3,b},W_{3,c}$. For each module, our basis has the form $\{v_1, v_2, \dots, v_{11}\}$, where 

\vspace{-0.7cm}
\begin{multicols}{2}
\begin{align*}
v_1 &= \nu, \\
v_2 &= A\nu - \tfrac{5}{2}E_4^*A^2\nu + \tfrac{1}{2}E_4^*A^4\nu, \\
v_3 &= \tfrac{1}{2}A\nu - \tfrac{7}{4}E_4^*A^2\nu + \tfrac{1}{4}E_4^*A^4\nu, \\
v_4 &= \tfrac{1}{2}A\nu - \tfrac{17}{4}E_4^*A^2\nu + \tfrac{3}{4}E_4^*A^4\nu, \\
v_5 &= -\tfrac{5}{2}E_5^*A^2\nu + \tfrac{1}{2}E_5^*A^3\nu + \tfrac{1}{2}E_5^*A^4\nu, \\
v_6 &= -\tfrac{17}{4}E_5^*A^2\nu + \tfrac{1}{4}E_5^*A^3\nu + \tfrac{3}{4}E_5^*A^4\nu,
\end{align*}
 
\columnbreak
 
\begin{align*}
v_7 &= -\tfrac{31}{4}E_5^*A^2\nu + \tfrac{3}{4}E_5^*A^3\nu + \tfrac{5}{4}E_5^*A^4\nu, \\
v_8 &= -\tfrac{1}{8}E_6^*A^3\nu + \tfrac{1}{4}E_6^*A^4\nu + \tfrac{1}{8}E_6^*A^5\nu, \\
v_9 &= -\tfrac{19}{16}E_6^*A^3\nu - \tfrac{1}{8}E_6^*A^4\nu + \tfrac{3}{16}E_6^*A^5\nu, \\
v_{10} &= -\tfrac{37}{16}E_6^*A^3\nu + \tfrac{1}{8}E_6^*A^4\nu + \tfrac{5}{16}E_6^*A^5\nu, \\
v_{11} &= \tfrac{1}{4}E_7^*A^5\nu.
\end{align*}
\end{multicols}
The seed vector $\nu$ is given below.

\begin{center}
    \begin{tabular}{c | c}
    \toprule
       \textbf{Module}  & \textbf {Essential part of }$\nu$ \\
       \midrule
       $W_{3,a}$  & $\rowvector{2 & -2 & -2 & 2 & 2 & -2 & -2 & 2 & 0 & 0 & 0 & 0}$ \\[1.67pt] 
       $W_{3,b}$  & $\rowvector{1 & -1 & 1 & -1 & -1 & 1 & -1 & 1 & 2 & -2 & 0 & 0}$ \\[1.67pt]
       $W_{3,c}$ & $\rowvector{1 & -1 & 1 & -1 & 1 & -1 & 1 & -1 & 0 & 0 & -2 & 2}$ \\
    \bottomrule
    \end{tabular}
    \captionof{table}{\textit{The seed vector $\nu$ for each irreducible $T$-module of Type IV. Note that $A\nu=E_4^*A \nu$.}}
\end{center}

With respect to each basis, the matrix representing $A$ is 
\begin{equation}
    A: \begin{bmatrix}
  0 & 1 & 1 & 0 & 0 & 0 & 0 & 0 & 0 & 0 & 0 \\
  1 & 1 & 0 & 0 & 1 & 0 & 0 & 0 & 0 & 0 & 0 \\
  1 & -1 & -1 & 0 & 0 & 1 & 0 & 0 & 0 & 0 & 0 \\
  -1 & -1 & 0 & -1 & 0 & 0 & 1 & 0 & 0 & 0 & 0 \\
  0 & 1 & 0 & 0 & 1 & 1 & 1 & 1 & 0 & 0 & 0 \\
  0 & 0 & 1 & 0 & 0 & -1 & 0 & 0 & 1 & 0 & 0 \\
  0 & 0 & 0 & 1 & 0 & 0 & -1 & 0 & 0 & 1 & 0 \\
  0 & 0 & 0 & 0 & 1 & 0 & 0 & 0 & 0 & 1 & 0 \\
  0 & 0 & 0 & 0 & 0 & 1 & 0 & -1 & -1 & -1 & 0 \\
  0 & 0 & 0 & 0 & 0 & 0 & 1 & 1 & 0 & 0 & 1 \\
  0 & 0 & 0 & 0 & 0 & 0 & 0 & 2 & 1 & 3 & 0 \end{bmatrix}.
\label{eq:bs-IV}
\end{equation}

This matrix has eigenvalues $\frac{1+\sqrt{17}}{2}$, $2$, $\theta_1$, $\theta_1$, $0$, $0$, $\theta_2$, $\theta_2$,$\frac{1-\sqrt{17}}{2}$, $\theta_3$, $\theta_3$, where $\theta_1, \theta_2, \theta_3 $ are defined in \eqref{eq:thetas}.

\subsection{Endpoint 4}

We now describe the Type V irreducible $T$-modules in our decomposition. For Type V, the multiplicity is 1 and the module is  $W_4$. For this module, our basis has the form $\{\nu, E_5^*A\nu, E_6^*A^2\nu, E_7^*A^3\nu\}$, where the seed vector $\nu$ is given below.
\begin{center}
    \begin{tabular}{c | c}
    \toprule
       \textbf{Module}  & \textbf {Essential part of }$\nu$ \\
       \midrule
       $W_{4}$  & $\smallrowvector{1 & -1 & -1 & 1 & 1 & -1 & 1 & -1 & -1 & 1 & -1 & 1 & 1 & -1 & -1 & 1 & 1 & -1 & -1 & 1 & -1 & 1 & -1 & 1}$ \\
    \bottomrule
    \end{tabular}
    \captionof{table}{\textit{The seed vector $\nu$ for the irreducible $T$-module of Type V.}}
\end{center}

With respect to this basis, the matrix representing $A$ is
\begin{equation}
    A: \begin{bmatrix}
  -1 & 1 & 0 & 0 \\
  1 & 1 & 1 & 0 \\
  0 & 1 & -1 & 3 \\
  0 & 0 & 1 & 0 \end{bmatrix}.
\label{eq:bs-V}
\end{equation}

This matrix has eigenvalues $2,\theta_1, \theta_2, \theta_3$, where $\theta_1, \theta_2, \theta_3 $ are defined in \eqref{eq:thetas}.


\renewcommand{\GraphNumber}{13}
\renewcommand{\GraphTitle}{Foster Graph}
\renewcommand{\VertexCount}{90}
\renewcommand{\IntersectionArray}{\{3,2,2,2,2,1,1,1;1,1,1,1,2,2,2,3\}}
\renewcommand{\DateString}{2026.04.12 (orig. 2026.01)}

\section{Foster Graph}\label{sec:foster}
Throughout this section, we take $\Gamma$ to be the Foster Graph. $\Gamma$ has 90 vertices and intersection array $\{3,2,2,2,2,1,1,1;1,1,1,1,2,2,2,3\}$.
\begin{center}
    \framebox[0.97\linewidth]{\includegraphics[width=0.95\linewidth]{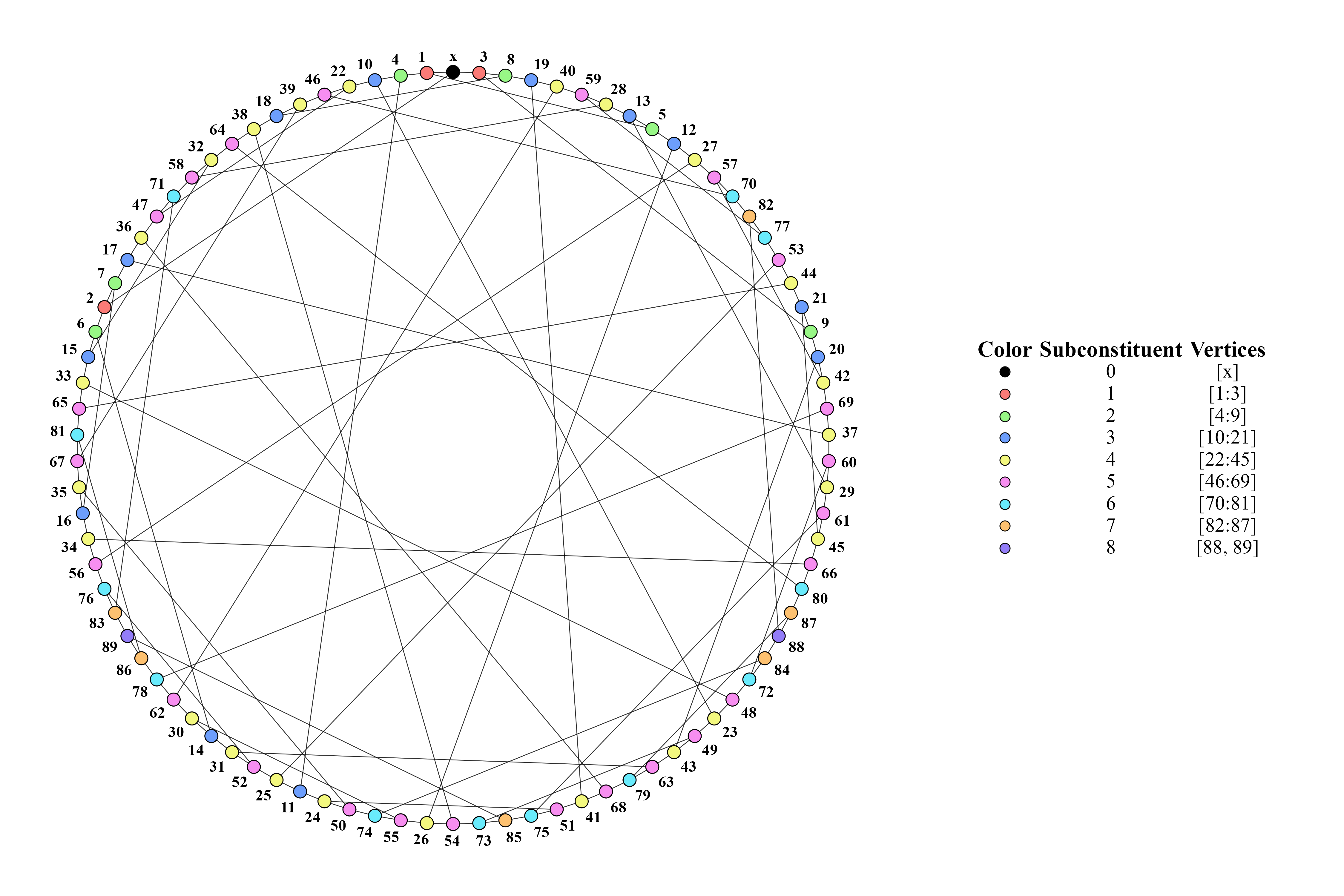}}
    \captionof{figure}{\textit{The Foster Graph.}}
    \end{center}

The spectrum of the adjacency matrix $A$ is $$ 3^1(\sqrt6)^{12}2^91^{18}0^{10}(-1)^{18}(-2)^9(-\sqrt6)^{12}(-3)^1 .$$
We decompose the standard module $V$ of $\Gamma$ into an orthogonal direct sum of irreducible $T$-modules. Our decomposition has the form
\[ V = W_0 \oplus W_{1,a} \oplus W_{1,b} \oplus W_{2,a} \oplus W_{2,b} \oplus W_{2,c} \oplus W_{3,a} \oplus \dots \oplus W_{3,f} \oplus W_{4,a} \oplus \dots \oplus W_{4,f} \oplus W_5. \]
For each irreducible $T$-module in this decomposition, we give the endpoint, the multiplicity, the dimension, the diameter, the shape, the action of $A$ upon an appropriate pure basis, and the isomorphism type.





\begin{center}
\resizebox{\linewidth}{!}{%
    \begin{tabular}{c|c c c c c c c}
\toprule
\textbf{Irred. $T$-modules} & \textbf{Endpt.} & \textbf{Mult.} & \textbf{Dim.} & \textbf{Diam.} & \textbf{Shape} & $A$ \textbf{action} & \textbf{Iso. type} \\
\midrule
$W_0$ & 0 & 1 & 9 & 8 & $(1,1,1,1,1,1,1,1,1)$ & \eqref{eq:fos-I} & I \\
$W_{1,a}, W_{1,b}$ & 1 & 2 & 7 & 6 & $(1,1,1,1,1,1,1)$ & \eqref{eq:fos-II} & II \\
$W_{2,a}, W_{2,b}, W_{2,c}$ & 2 & 3 & 7 & 4 & $(1,1,2,2,1)$ & \eqref{eq:fos-III} & III \\
$W_{3,a}, W_{3,b}, W_{3,c}$ & 3 & 3 & 4 & 2 & $(1,2,1)$ & \eqref{eq:fos-IV} & IV \\
$W_{3,d}, W_{3,e}, W_{3,f}$ & 3 & 3 & 5 & 3 & $(1,1,2,1)$ & \eqref{eq:fos-V} & V \\
$W_{4,a}$ & 4 & 1 & 1 & 0 & $(1)$ & \eqref{eq:fos-VI} & VI \\
$W_{4,b}, W_{4,c}, W_{4,d}$ & 4 & 3 & 2 & 1 & $(1,1)$ & \eqref{eq:fos-VII} & VII\\
$W_{4,e}, W_{4,f}$ & 4 & 2 & 4 & 3 & $(1,1,1,1)$ & \eqref{eq:fos-VIII} & VIII \\
$W_{5}$ & 5 & 1 & 4 & 3 & $(1,1,1,1)$ & \eqref{eq:fos-IX} & IX \\
\bottomrule
    \end{tabular}
}

\smallskip
\captionof{table}{\textit{Irreducible $T$-modules for the Foster Graph} \\
$(\dim T =9^2+7^2+7^2+4^2+5^2+1^2+2^2+4^2+4^2=257)$}
\end{center}

For each irreducible $T$-module in the table above, we now give a pure basis and the matrix representing $A$ on that basis. 

\subsection{Endpoint 0}

We now describe the primary irreducible $T$-module $W_0$. The module $W_0$ has a basis $\{e_i^*\}_{i=0}^8$, where $e_i^*$ is from Definition \ref{def:primary}. With respect to this basis the matrix representing $A$ is
\begin{equation}
    A: \begin{bmatrix}
  0 & 3 & 0 & 0 & 0 & 0 & 0 & 0 & 0 \\
  1 & 0 & 2 & 0 & 0 & 0 & 0 & 0 & 0 \\
  0 & 1 & 0 & 2 & 0 & 0 & 0 & 0 & 0 \\
  0 & 0 & 1 & 0 & 2 & 0 & 0 & 0 & 0 \\
  0 & 0 & 0 & 1 & 0 & 2 & 0 & 0 & 0 \\
  0 & 0 & 0 & 0 & 2 & 0 & 1 & 0 & 0 \\
  0 & 0 & 0 & 0 & 0 & 2 & 0 & 1 & 0 \\
  0 & 0 & 0 & 0 & 0 & 0 & 2 & 0 & 1 \\
  0 & 0 & 0 & 0 & 0 & 0 & 0 & 3 & 0 \end{bmatrix}.
\label{eq:fos-I}
\end{equation}
This matrix has eigenvalues $3, \sqrt{6}, 2, 1, 0, -1, -2, -\sqrt{6}, -3$.

\subsection{Endpoint 1}

We now describe the Type II irreducible $T$-modules in our decomposition. For Type II, the multiplicity is 2 and the modules are $W_{1,a}$ and $W_{1,b}$. For each module, our basis has the form 
\begin{equation*}
    \{\nu, A\nu, E_3^*A^2\nu, E_4^*A^3\nu, E_5^*A^4\nu, E_6^*A^5\nu, E_7^*A^6\nu\},
\end{equation*} 
where the seed vector $\nu$ is given below.
\begin{center}
    \begin{tabular}{c | c}
    \toprule
       \textbf{Module}  & \textbf {Essential part of }$\nu$ \\
       \midrule
       $W_{1,a}$  & $\rowvector{1 & -1 & 0}$ \\[1.67pt]
       $W_{1,b}$  & $\rowvector{1 & 1 & -2}$ \\
    \bottomrule
    \end{tabular}
    \captionof{table}{\textit{The seed vector $\nu$ for each irreducible $T$-module of Type II. Note that $A\nu=E_2^*A\nu$.}}
\end{center}

With respect to each basis, the matrix representing $A$ is 
\begin{equation}
    A: \begin{bmatrix}
  0 & 2 & 0 & 0 & 0 & 0 & 0 \\
  1 & 0 & 2 & 0 & 0 & 0 & 0 \\
  0 & 1 & 0 & 2 & 0 & 0 & 0 \\
  0 & 0 & 1 & 0 & 1 & 0 & 0 \\
  0 & 0 & 0 & 1 & 0 & 2 & 0 \\
  0 & 0 & 0 & 0 & 1 & 0 & 2 \\
  0 & 0 & 0 & 0 & 0 & 1 & 0 \end{bmatrix}.
\label{eq:fos-II}
\end{equation}

This matrix has eigenvalues $\sqrt{6}, 2, 1, 0, -1, -2, -\sqrt{6}$.

\subsection{Endpoint 2}

We now describe the Type III irreducible $T$-modules in our decomposition. For Type III, the multiplicity is 3 and the modules are $W_{2,a},W_{2,b},W_{2,c}$. For each module, our basis has the form 
\begin{equation*}
    \{\nu, A\nu, E_4^*A^2\nu,  E_4^*A^4\nu - 6E_4^*A^2\nu, E_5^*A^3\nu,E_5^*A^5\nu - 7E_5^*A^3\nu, E_6^*A^6\nu\},
\end{equation*} 
where the seed vector $\nu$ is given below.
\begin{center}
    \begin{tabular}{c | c}
    \toprule
       \textbf{Module}  & \textbf {Essential part of }$\nu$ \\
       \midrule
       $W_{2,a}$  & $\rowvector{1 & -1 & 0 & 0 & 0 & 0}$ \\[1.67pt]
       $W_{2,b}$  & $\rowvector{0 & 0 & 1 & -1 & 0 & 0}$ \\[1.67pt]
       $W_{2,c}$  & $\rowvector{0 & 0 & 0 & 0 & 1 & -1}$ \\
    \bottomrule
   \end{tabular}
    \captionof{table}{\textit{The seed vector $\nu$ for each irreducible $T$-module of Type III. Note that $A\nu=E_3^*A\nu$.}}
\end{center}

With respect to each basis, the matrix representing $A$ is 
\begin{equation}
    A: \begin{bmatrix}
  0 & 2 & 0 & 0 & 0 & 0 & 0 \\
  1 & 0 & 2 & 0 & 0 & 0 & 0 \\
  0 & 1 & 0 & 0 & 2 & 0 & 0 \\
  0 & 0 & 0 & 0 & 1 & 2 & 0 \\
  0 & 0 & 1 & 1 & 0 & 0 & 0 \\
  0 & 0 & 0 & 1 & 0 & 0 & 2 \\
  0 & 0 & 0 & 0 & 0 & 1 & 0 \end{bmatrix}.
\label{eq:fos-III}
\end{equation}
This matrix has eigenvalues $\sqrt{6}, 2, 1, 0, -1, -2, -\sqrt{6}$.

\subsection{Endpoint 3}

We now describe the Type IV irreducible $T$-modules in our decomposition. For Type IV, the multiplicity is 3 and the modules are $W_{3,a},W_{3,b},W_{3,c}$. For each module, our basis has the form 
\begin{equation*}
    \{\nu, A\nu, E_4^*A^3\nu, E_5^*A^2\nu\},
\end{equation*} 
where the seed vector $\nu$ is given below.
\begin{center}
    \begin{tabular}{c | c}
    \toprule
       \textbf{Module}  & \textbf {Essential part of }$\nu$ \\
       \midrule
       $W_{3,a}$  & $\rowvector{0 & 0 & 0 & 0 & 1 & -1 & -1 & 1 & 1 & -1 & -1 & 1}$ \\[1.67pt]
       $W_{3,b}$  & $\rowvector{2 & -2 & 0 & 0 & 1 & -1 & 1 & -1 & -1 & 1 & -1 & 1}$ \\[1.67pt]
      $W_{3,c}$  & $\rowvector{0 & 0 & 2 & -2 & -1 & 1 & -1 & 1 & -1 & 1 & -1 & 1}$ \\
    \bottomrule
    \end{tabular}
    \captionof{table}{\textit{The seed vector $\nu$ for each irreducible $T$-module of Type IV. Note that $A\nu=E_4^*A\nu$.}}
\end{center}

With respect to each basis, the matrix representing $A$ is 
\begin{equation}
    A: \begin{bmatrix}
  0 & 2 & 6 & 0 \\
  1 & 0 & 0 & -2 \\
  0 & 0 & 0 & 1\\
  0 & 1 & 5 & 0 \end{bmatrix}.
\label{eq:fos-IV}
\end{equation}

This matrix has eigenvalues $2, 1, -1, -2$.

We now describe the Type V irreducible $T$-modules in our decomposition. For Type V, the multiplicity is 3 and the modules are $W_{3,d},W_{3,e},W_{3,f}$. For each module, our basis has the form 
\begin{equation*}
    \{\nu, A\nu, E_5^*A^2\nu, E_5^*A^4\nu, E_6^*A^3\nu\},
\end{equation*}
where the seed vector $\nu$ is given below.
\begin{center}
    \begin{tabular}{c | c}
    \toprule
       \textbf{Module}  & \textbf {Essential part of }$\nu$ \\
       \midrule
       $W_{3,d}$  & $\rowvector{0 & 0 & 0 & 0 & 1 & -1 & -1 & 1 & -1 & 1 & 1 & -1}$ \\[1.67pt]
       $W_{3,e}$  & $\rowvector{1 & -1 & -1 & 1 & -1 & 1 & -1 & 1 & 0 & 0 & 0 & 0}$ \\[1.67pt]
      $W_{3,f}$  & $\rowvector{1 & -1 & 1 & -1 & 0 & 0 & 0 & 0 & 1 & -1 & 1 & -1}$ \\
    \bottomrule
    \end{tabular}
    \captionof{table}{\textit{The seed vector $\nu$ for each irreducible $T$-module of Type V. Note that $A\nu=E_4^*A\nu$.}}
\end{center}

With respect to each basis, the matrix representing $A$ is 
\begin{equation}
    A: \begin{bmatrix}
  0 & 2 & 0 & 0 & 0 \\
  1 & 0 & 3 & 19 & 0 \\
  0 & 1 & 0 & 0 & -5 \\
  0 & 0 & 0 & 0 & 1 \\
  0 & 0 & 1 & 7 & 0 \end{bmatrix}.
\label{eq:fos-V}
\end{equation}

This matrix has eigenvalues $\sqrt{6}, 1, 0, -1, -\sqrt{6}$.

\subsection{Endpoint 4}

We now describe the Type VI irreducible $T$-modules in our decomposition. For Type VI, the multiplicity is 1 and the module is  $W_{4,a}$. For this module, our basis has the form $\{\nu\}$, where the seed vector $\nu$ is given below.
\begin{center}
    \resizebox{\linewidth}{!}{%
    \begin{tabular}{c | c}
    \toprule
       \textbf{Module}  & \textbf {Essential part of }$\nu$ \\
       \midrule
       $W_{4,a}$  & $\rowvector{1 & -1 & -1 & 1 & -1 & 1 & 1 & -1 & 1 & -1 & -1 & 1 & -1 & 1 & -1 & 1 & 1 & -1 & -1 & 1 & -1 & 1 & -1 & 1}$ \\
    \bottomrule
    \end{tabular}
    }
    \captionof{table}{\textit{The seed vector $\nu$ for the irreducible $T$-module of Type VI.}}
\end{center}

With respect to this basis, the matrix representing $A$ is
\begin{equation}
    A: \begin{bmatrix}
  0
\end{bmatrix}.
\label{eq:fos-VI}
\end{equation}

This matrix has eigenvalue $0$.

We now describe the Type VII irreducible $T$-modules in our decomposition. For Type VII, the multiplicity is 3 and the modules are $W_{4,b},W_{4,c},W_{4,d}$. For each module, our basis has the form $\{\nu, A\nu\}$, where the seed vector $\nu$ is given below.
\begin{center}
    \resizebox{\linewidth}{!}{%
    \begin{tabular}{c | c}
    \toprule
       \textbf{Module}  & \textbf {Essential part of }$\nu$ \\
       \midrule
       $W_{4,b}$  & $\rowvector{1 & -1 & 1 & -1 & 1 & -1 & 1 & -1 & -1 & 1 & -1 & 1 & 1 & -1 & -1 & 1 & 0 & 0 & 0 & 0 & 0 & 0 & 0 & 0}$ \\[1.67pt]
       $W_{4,c}$  & $\rowvector{-1 & 1 & -1 & 1 & 1 & -1 & 1 & -1 & 1 & -1 & 1 & -1 & 1 & -1 & -1 & 1 & -2 & 2 & -2 & 2 & 0 & 0 & 0 & 0}$ \\[1.67pt]
       $W_{4,d}$  & $\rowvector{1 & -1 & 1 & -1 & -1 & 1 & -1 & 1 & 1 & -1 & 1 & -1 & 1 & -1 & -1 & 1 & 0 & 0 & 0 & 0 & -2 & 2 & 2 & -2}$ \\
  \bottomrule
    \end{tabular}}
    \captionof{table}{\textit{The seed vector $\nu$ for each irreducible $T$-module of Type VII. Note that $A\nu=E_5^*A\nu$.}}
\end{center}

With respect to each basis, the matrix representing $A$ is 
\begin{equation}
    A: \begin{bmatrix}
  0 & 1 \\
  1 & 0 \end{bmatrix}.\label{eq:fos-VII}
\end{equation}

This matrix has eigenvalues $1, -1$.

We now describe the Type VIII irreducible $T$-modules in our decomposition. For Type VIII, the multiplicity is 2 and the modules are $W_{4,e}$ and $W_{4,f}$. For each module, our basis has the form $$\{\nu, A\nu, E_6^*A^2\nu, E_7^*A^3\nu\},$$ where the seed vector $\nu$ is given below.
\begin{center}
    \resizebox{\linewidth}{!}{%
    \begin{tabular}{c | c}
    \toprule
       \textbf{Module}  & \textbf {Essential part of }$\nu$ \\
       \midrule
       $W_{4,e}$  & $\rowvector{1 & -1 & -1 & 1 & -1 & 1 & 1 & -1 & -1 & 1 & 1 & -1 & 1 & -1 & 1 & -1 & 0 & 0 & 0 & 0 & 0 & 0 & 0 & 0}$ \\[1.67pt]
       $W_{4,f}$  & $\rowvector{1 & -1 & -1 & 1 & -1 & 1 & 1 & -1 & 1 & -1 & -1 & 1 & -1 & 1 & -1 & 1 & -2 & 2 & 2 & -2 & 2 & -2 & 2 & -2}$ \\
    \bottomrule
    \end{tabular}}
    \captionof{table}{\textit{The seed vector $\nu$ for each irreducible $T$-module of Type VIII. Note that $A\nu=E_5^*A\nu$.}}
\end{center}

With respect to each basis, the matrix representing $A$ is 
\begin{equation}
    A: \begin{bmatrix}
  0 & 3 & 0 & 0 \\
  1 & 0 & 2 & 0 \\
  0 & 1 & 0 & 2 \\
  0 & 0 & 1 & 0 \end{bmatrix}.
\label{eq:fos-VIII}
\end{equation}

This matrix has eigenvalues $\sqrt{6}, 1, -1, -\sqrt{6}$.

\subsection{Endpoint 5}

We now describe the Type IX irreducible $T$-modules in our decomposition. For Type IX, the multiplicity is 1 and the module is  $W_5$. For this module, our basis has the form $\{\nu, A\nu, E_7^*A^2\nu, E_8^*A^3\nu\}$, where the seed vector $\nu$ is given below.
\begin{center}
    \begin{tabular}{c | c}
    \toprule
       \textbf{Module}  & \textbf {Essential part of }$\nu$ \\
       \midrule
       $W_{5}$  & $\smallrowvector{1 & -1 & 1 & -1 & 1 & -1 & -1 & 1 & -1 & 1 & -1 & 1 & -1 & 1 & 1 & -1 & -1 & 1 & 1 & -1 & 1 & -1 & 1 & -1}$ \\
    \bottomrule
    \end{tabular}
    \captionof{table}{\textit{The seed vector $\nu$ for the irreducible $T$-module of Type IX. Note that $A\nu=E_6^*A\nu$.}}
\end{center}

With respect to this basis, the matrix representing $A$ is
\begin{equation}
    A: \begin{bmatrix}
  0 & 2 & 0 & 0 \\
  1 & 0 & 2 & 0 \\
  0 & 1 & 0 & 3 \\
  0 & 0 & 1 & 0 \end{bmatrix}.
\label{eq:fos-IX}
\end{equation}

This matrix has eigenvalues $\sqrt{6}, 1, -1, -\sqrt{6}$.

\section{Conclusion}\label{sec:conclusion}

In the table below, we list the 13 distance-regular graphs of valency three. For each graph $\Gamma$ we give the diameter $D$, the number of vertices $|X|$, and the dimension of the Terwilliger algebra $T$.

\begin{center}
\begin{tabular}{lccr}
    \toprule
       Graph $\Gamma$ & Diam. $D$& $|X|$ & $\dim T$ \\
       \midrule
       $K_4$ &1& 4  & $2^2+1^2= \mathbf 5$ \\
    $K_{3,3}$ &2& 6 & $3^2 +1^2+1^2 =\mathbf {11}$ \\
    Petersen &2& 10   & $3^2+2^2+1^2 +1^2 =\mathbf{15}$ \\
    $3$-cube &3& 8 & $4^2+2^2=\mathbf{20}$ \\   
    Heawood &3& 14   & $4^2+2^2+2^2=\mathbf{24}$ \\
    Pappus &4& 18  & $5^2+3^2+1^2 +2^2+2^2 =\mathbf{43}$\\
    Coxeter &4& 28    & $5^2+5^2+2^2+4^2+1^2+2^2=\mathbf{75}$\\
    Tutte's 8-cage &4& 30   & $5^2+3^2+3^2+1^2+2^2+1^2=\mathbf{49}$ \\
    Dodecahedron &5& 20   & $6^2+6^2+2^2=\mathbf{76}$ \\
    Desargues &5& 20    & $6^2+4^2+2^2+2^2=\mathbf{60}$ \\
    Tutte's 12-cage &6& 126    & For $x\in X^+,$ \text{ }\hfill $7^2 + 4  (5^2) + 3^2 +2 (2^2) +2 (1^2)=\mathbf{168}$ \\
     &&&For $x\in X^-,$ \hfill $7^2 + 4  (5^2) +3 (2^2) +2 (1^2)=\mathbf{163}$ \\
        Biggs-Smith &7& 102   & $8^2 + 9^2 +13^2 +11^2+4^2=\mathbf{451}$ \\
    Foster &8& 90    & $9^2+2(7^2)+5^2+3(4^2)+2^2+1^2=\mathbf{257}$ \\
    \bottomrule
    \end{tabular}
    \captionof{table}{\textit{The 13 distance-regular graphs with valency three.}}
\end{center}

\section{Supplementary Material}

For this paper, we provide some supplementary material at \url{https://valkobarnabas.github.io/tmodules}. This material may be visualized on-screen or downloaded and used by the reader.

\section{Acknowledgements} 

This project was part of the Madison Experimental Mathematics Laboratory (MXM) at the University of Wisconsin--Madison for the academic year 2025/26. The authors would like to thank the MXM organizers Prof.\ Ça\u{g}lar Uyan{\i}k and Prof.\ Grace Work for their leadership and support.

Many of the calculations in this paper were performed using the R programming language. Graph figures were generated using R along with data from Wikipedia.

\noindent
Kevin Kauflin, Paul Terwilliger, Barnab\'as Valk\'o, Hanyi Wu\\
Department of Mathematics\\
University of Wisconsin--Madison\\
480 Lincoln Drive\\
Madison, WI 53706-1388 USA\\
emails: \url{kkauflin@wisc.edu}, \url{terwilli@math.wisc.edu}, \url{bvalko@wisc.edu}, \url{hwu496@wisc.edu}

\noindent
Jimmy Vineyard\\
Department of Mathematics\\
The Ohio State University\\
231 W 18th Ave\\
Columbus, OH 43210-1101 USA\\
email: \url{vineyard.23@osu.edu}

\section{Statements and Declarations}

\textbf{Funding:} The authors declare that no funds, grants, or other support were received during
the preparation of this manuscript.

\textbf{Competing interests:} The authors have no relevant financial or non-financial interests to
disclose.

\textbf{Data availability:} All data generated or analyzed during this study are included in this
published article or in the supplementary material at the website \url{https://valkobarnabas.github.io/tmodules}.

\end{document}